\documentclass{amsart}
\usepackage{hyperref}
\usepackage{times}
\usepackage{verbatim}
\usepackage[T1]{fontenc}
\usepackage[utf8]{inputenc}
\usepackage[english]{babel}
\usepackage{amssymb}
\usepackage{amsthm}
\usepackage{amsmath}
\usepackage{amstext}
\usepackage{multirow}
\usepackage{graphicx}
\usepackage[all]{xy}
\usepackage{mathrsfs}
\usepackage{xcolor}
\usepackage{fixfoot}
\newtheorem{thm}{Theorem}[section]
\newtheorem{dfn}[thm]{Definition}
\newtheorem{dfnprop}[thm]{Definition/Proposition}
\newtheorem{lemma}[thm]{Lemma}
\newtheorem{prop}[thm]{Proposition}

\newtheorem{cor}[thm]{Corollary}

\newtheorem{THM}{Theorem}

\newtheorem{COR}[THM]{Corollary}
\newtheorem{LEM}[THM]{Lemma}

\newtheorem{remark}[thm]{Remark}

\newcommand{\mb}{\mathbb}
\newcommand{\mc}{\mathcal}
\newcommand{\R}{\mb R}
\newcommand{\C}{\mb C}
\newcommand{\PP}{\mb P}
\newcommand{\Z}{\mb Z}
\newcommand{\Q}{\mb Q}
\newcommand{\N}{\mb N}
\newcommand{\bS}{\mb S}

\newcommand{\OO}{\mc O}
\newcommand{\F}{\mc F}
\newcommand{\G}{\mc G}

\newcommand{\acts}{\rotatebox[origin=c]{90}{$\curvearrowright$}\ }

\DeclareMathOperator{\im}{Im}

\DeclareFixedFootnote{\repnote}{ These notations are explained in Subsection \ref{sec:sheavesautomorphisms}, but should be meaningful when looking at Picture \ref{pic:SecNorm}.}

\newcommand{\Diff}{\mathrm{Diff}(\C,0)}

\newcommand{\lattice}{\Gamma_\tau}

\newcommand{\an}{\stackrel{\text{an}}{\sim}}
\newcommand{\formal}{\stackrel{\text{for}}{\sim}}

\usepackage{mathtools}
\usepackage{ifmtarg}
\addtocounter{section}{0}             
\numberwithin{equation}{section}       

\title[Neighborhoods of elliptic curves]{Two dimensional neighborhoods of elliptic curves: analytic classification in the torsion case.}

\author[F. Loray]{Frank Loray}
\author[F. Touzet]{Fr\'ed\'eric Touzet}
\author[S. M. Voronin]{Sergei M. Voronin}
\address{Univ Rennes, CNRS, IRMAR - UMR 6625, F-35000 Rennes, France}
\email{frank.loray@univ-rennes1.fr}
\email{frederic.touzet@univ-rennes1.fr}
\address{Department of Mathematics, Chelyabinsk State University, Chelyabinsk, Russia}
\email{voron@csu.ru}

\thanks{The two first authors are supported by grant ANR-16-CE40-0008 ``Foliage'',
and thank CAPES-COFECUB project MA 932/19.
The third author is supported by grant RFBR-17-01-00739-a.
This work started when the third author was invited in Rennes, and we would like 
to thank Universit\'e de Rennes 1, IRMAR, CNRS and Henri Lebesgue Center for constant support.
We thank T. Ohsawa to let us know the reference \cite{GS}.}

\begin{document}
\begin{abstract} 
We investigate the analytic classification of two dimensional neighborhoods of an elliptic curve with torsion normal bundle.
We provide the complete analytic classification for those neighborhoods in the simplest formal class and we indicate how to generalize this construction to general torsion case.
\end{abstract}
\maketitle

\sloppy
\tableofcontents

\section{Introduction and results}

Let $C$ be a smooth elliptic curve: $C=\C/\lattice$, where $\lattice=\Z+\tau\Z$, with $\Im({\tau})>0$.
Given an embedding $\iota:C\hookrightarrow U$ of $C$ into a smooth complex surface $U$,
we would like to understand the germ $(U,\iota(C))$ of neighborhood of $\iota(C)$ in $U$.
Precisely, we will say that two embeddings $\iota,{\iota}':C\hookrightarrow U,U'$ are (formally/analytically) equivalent
if there is a (formal/analytic) isomorphism $\Psi:(U,\iota(C))\to(U',\iota'(C))$ between germs of neighborhoods
making commutative the following diagram
\begin{equation}\label{eq:equivNeighbId}
\xymatrix{\relax
    C \ar[r]^\iota \ar[d]_{\text{id}}  & U \ar[d]^\Psi \\
    C \ar[r]^{\iota'} & U'
}\end{equation}
By abuse of notation, we will still denote by $C$ the image $\iota(C)$ of its embedding in $U$, 
and we will simply denote by $(U,C)$ the germ of neighborhood.

\subsection{Some historical background}
The problem of analytic classification of neighborhoods of compact complex curves in complex surfaces
goes back at least to the celebrated work of Grauert \cite{Grauert}. There, he considered the normal 
bundle $N_C$ of the curve in $U$. The neighborhood of the zero section in the total space of $N_C$,
that we denote $(N_C,0)$,
can be viewed\footnote{Strictly speaking, the linear part is more complicated in general, 
as it needs not fiber over the curve, as it is the case for the neighborhood of a conic in $\mathbb P^2$.} 
as the linear part of $(U,C)$. 
A coarse invariant is given by the degree $\deg{N_C}$ which is also the self-intersection $C\cdot C$ of the curve.
In this paper \cite{Grauert}, Grauert proved that the germ of neighborhood is ``linearizable'',
i.e. analytically equivalent to the germ of neighborhood $(N_C,0)$, provided that $\deg(N_C)$ is negative enough,
namely $\deg(N_C)<4-4g$ for a curve of genus $g>0$, and $\deg(N_C)<0$ for a rational curve $g=0$.
It was also clear from his work that even the formal classification was much more complicated when 
$\deg(N_C)>0$. At the same period, Kodaira investigated the deformation of compact submanifolds 
of complex manifolds in \cite{Kodaira}. His result, in the particular case of curves in surfaces, says
that the curve can be deformed provided that $\deg(N_C)$ is positive enough,
namely $\deg(N_C)>2g-2$ for a curve of genus $g>0$, and $\deg(N_C)\ge0$ for a rational curve $g=0$.
Using these deformations, it is possible to provide a complete set of invariants for analytic classification
for $g=0$: $(U,C)$ is linearizable when $\deg(N_C)\le0$ (Grauert for $C.C<0$ and Savelev \cite{Savelev}
for $C.C=0$), and there is a functional moduli\footnote{The moduli space is comparable with the ring of convergent power series $\mathbb C\{X,Y\}$.} 
when $\deg(N_C)>0$ following Mishustin \cite{Mishustin0} (see also \cite{MaycolFrank}).
Also, when $g>0$ and $\deg(N_C)>2g-2$, the analytic classification has been carried out by 
Ilyashenko \cite{Ilyashenko} and Mishustin \cite{Mishustin1}. In all these results, it is important
to notice that formally equivalent neighborhoods are also analytically equivalent: the two 
classifications coincide for such neighborhoods. Such a rigidity property is called \textit{the formal principle} 
(see the recent works \cite{Hwang,JorgeOlivier} on this topic).

The case of an elliptic\footnote{Elliptic means $g=1$ and that we have moreover fixed
a (zero) point on the curve, to avoid considering automorphisms of the curve in our study.} 
 curve $g=1$ with $\deg(N_C)=0$, which is still open today, has been
investigated by Arnold \cite{Arnold} in another celebrated work. In this case, the normal bundle
$N_C$ belongs to the Jacobian curve $\mathrm{Jac}(C)\simeq C=\mathbb C/\Gamma_\tau$
and can be torsion\footnote{Torsion means that some iterate for the group law $\otimes$ is
the trivial bundle $\mathcal O_C$.} 
or not. Torsion points correspond to the image of 
$\mathbb Q+\tau\mathbb Q\subset C$ in the curve. Arnold investigated the non torsion case
and proved in that case
\begin{itemize}
\item if $N_C$ is non torsion, then $(U,C)$ is formally linearizable;
\item if $N_C$ is generic\footnote{i.e. belongs to some subset of total Lebesgue measure 
defined by a certain diophantine condition}  
enough in $\mathrm{Jac}(C)$, then $(U,C)$ is analytically linearizable;
\item for non generic (and still non torsion) $N_C$, there is a huge\footnote{Thanks to the works 
of Yoccoz \cite{Yoccoz} and Perez-Marco \cite{PerezMarco}, 
we can embed at least $\C\{X\}$ in the moduli space with a huge degree of freedom.} 
moduli space for the analytic classification.
\end{itemize}
However, we are still far, nowadays, to expect a complete description of the analytic classification in the non torsion case.
It is the first case where the divergence between formal and analytic classification arises. Also, it is interesting to note that 
the study of neighborhoods of elliptic curves in the case $\deg(N_C)=0$ has strong reminiscence with the classification
of germs of diffeomorphisms up to conjugacy. It will be more explicit later when describing the torsion case.

The goal of this paper is to investigate the analytic classification when the normal bundle is torsion,
and show that we can expect to provide a complete description of the moduli space in that case.
More precisely, the formal classification of such neighborhoods has been  achieved in \cite{LTT};
we provide the analytic classification inside the simplest formal class, 
and we explain in Sections \ref{S:Generalization} and \ref{S:torsion} how we think
it should extend to all other formal classes, according to the same "principle". 

\subsection{Formal classification}\label{SS:formalclass}
An important formal invariant has been introduced by Ueda \cite{Ueda} in the case $\deg(N_C)=0$ and $g>0$.
There, among other results, he investigates the obstruction for the curve to be the fiber of a fibration
(as it would be in the linear case $(N_C,0)$ when $N_C$ is torsion). The Ueda type $k\in\Z_{>0}\cup\{\infty\}$
is the largest integer for which the aforementioned fibration\footnote{More generally, in the non torsion case, 
we may try to extend the foliation defined by the unitary connection on $N_C$.} 
of $N_C$ can be extended to the 
$k^{\text{th}}$ infinitesimal neighborhood of $C$ (see \cite[section 2]{leaves} for a short exposition).
When $k=\infty$, then we have a formal fibration, that can be proved to be analytic; the classification 
in that case goes back to the works of Kodaira, in particular in the elliptic case $g=1$. 

Inspired by Ueda's approach, it has been proved by Claudon, Pereira and the two first named authors of this paper 
(see \cite{leaves}) that a formal neighborhood $(U,C)$ with $\deg(N_C)=0$
carries many regular (formal) foliations such that $C$ is a compact leaf. This construction has been
improved in \cite{LTT,Olivier} showing that one can choose two of these foliations in a canonical 
way and use them to produce a complete set of formal invariants.
In the elliptic case $g=1$, there are $\frac{k}{m}+1$ independant formal invariants for finite fixed Ueda type $k$
where $m$ is the torsion order $N_C^{\otimes m}=\mathcal O_C$ (see \cite{LTT});
for $g>1$ and $N_C=\mathcal O_C$ (the trivial bundle), Thom founds infinitely many 
independant formal invariants in \cite{Olivier}.

In this paper, we only consider the case $g=1$ where $N_C=\mathcal O_C$ is the trivial bundle,
to which we can reduce via a cyclic cover whenever $N_C$ is torsion. Let us recall the formal 
classification in that case. For each Ueda type $k\in\Z_{>0}$, let $P\in\C[X]$ be any polynomial
of degree $<k$ and $\nu\in\mathbb C$ a scalar. To these data, we associate a germ of neighborhood
$(U_{k,\nu,P},C)$ as follows. Writing $C$ as a quotient of $\C^*$ by a contraction:
$$C=\mathbb C_z^*/<z\mapsto q z>\ \ \ \text{with}\ z=e^{2i\pi x}\ \text{and}\ q=e^{2i\pi\tau},\ \ \ (|q|<1)$$
we similarly define $(U_{k,\nu,P},C)$ as the quotient of the germ of neighborhood 
$$(\C_z^*\times \C_y,\{y=0\})$$
by the germ of diffeomorphism
$$F_{k,\nu,P}=\exp( v_0+v_\infty),\ \ \ \text{where}\ 
\left\{\begin{matrix}v_0&=&\frac{y^{k+1}}{1+\nu y^k}\partial_y+2i\pi\tau \frac{y^kP(\frac{1}{y})}{1+\nu y^k}z\partial_z\\
v_\infty&=&2i\pi\tau z\partial_z\hfill\end{matrix}\right.$$
or equivalently
$$\left\{\begin{matrix}v_0&=&\frac{1}{{\xi}^k +\nu} (-\xi{\partial}_{\xi} +2i\pi\tau zP(\xi){\partial_ z})\\
	v_\infty&=&2i\pi\tau z\partial_z\hfill\end{matrix}\right.$$
by setting $\xi=\frac{1}{y}$, in which case $(U_{k,\nu,P})$ is regarded as a quotient of $(\C_z^*\times \overline{\C_\xi},\{\xi=\infty\})$.

\textit {In the specific situation where $k=1$}, and then $P$  is reduced to a constant $P\equiv\mu$,  we will use the notation $(U_{1,\nu,\mu},C)$ for the corresponding neighborhood.

The two vector fields $v_0$ and $v_\infty$ span a commutative Lie algebra, and therefore an infinitesimal $\C^2$-action
on the quotient neighborhood. By duality, we have a $2$-dimensional vector space of closed meromorphic $1$-forms 
spanned by
$$-\omega_0=\frac{dy}{y^{k+1}}+\nu \frac{dy}{y}\ \ \ \text{and}\ \ \ \omega_\infty=\frac{1}{2i\pi\tau}\frac{dz}{z}-\frac{P(\frac{1}{y})}{y}dy.$$ expressed also in the $(z,\xi)$ coordinates as 
$$\omega_0=\xi^{k-1}{d\xi}+\nu \frac{d\xi}{\xi}\ \ \ \text{and}\ \ \ \omega_\infty=\frac{1}{2i\pi\tau}\frac{dz}{z}+\frac{P(\xi)}{\xi}d\xi.$$
In particular, we get a pencil of foliations $\mathcal F_t$, $t\in\PP^1$, by considering 
\begin{itemize}
\item either the phase portrait of the vector fields $v_t=tv_0+v_\infty$,
\item or $\omega_t=0$ where $\omega_t=\omega_0+t\omega_\infty$.
\end{itemize}
When $P=0$, $\mathcal F_\infty$ defines a fibration transversal to the curve $C$ and the neighborhood
is the suspension\footnote{in the sense of foliations} 
of a representation $\varrho:\pi_1(C)\to\mathrm{Diff}(\C,0)$
taking values into the one-parameter group generated by $v_0=\frac{y^{k+1}}{1+\nu y^k}\partial_y$.
For $t\in\C$ finite, $\mathcal F_t$ is always (smooth) tangent to $C$, i.e. $C$ is a compact leaf;
when $P\not=0$, the same holds for $\mathcal F_\infty$.
For $m+\tau n\in\Gamma\setminus\{0\}$, $\mathcal F_{\frac{\tau n}{m+\tau n}}$ is the unique foliation
of the pencil whose holonomy along the corresponding loop $m+\tau n$ in $\pi_1(C)\sim\Gamma$ is trivial (see Section \ref{sec:Foliations}). 
As proved in \cite[Theorem 1.3]{LTT}, the neighborhoods $(U_{k,\nu,P},C)$ span all formal classes 
of neighborhoods with trivial normal bundle $N_C=\mathcal O_C$ and finite Ueda type $k$;
moreover, any two such neighborhoods are formally equivalent $(U_{k,\nu,P},C)\sim^{\text{for}}(U_{k',\nu',P'},C)$ if, and only if there is a $k^{\text{th}}$-root of unity $\zeta$ such that:
$$k=k',\ \ \ \nu=\nu'\ \ \ \text{and}\ \ \ P'(y)=  P(\zeta y),\ \ \ \zeta^k=1.$$

As explained in \cite[Theorem 1.5]{LTT}, the moduli space of those neighborhoods with two convergent foliations
in a given formal class up to analytic conjugacy is infinite dimensional\footnote{isomorphic to \'Ecalle-Voronin moduli spaces}, 
comparable with $\C\{X\}$. {\it A contrario}, if a third foliation is convergent, then the neighborhood is analytically equivalent to 
its formal model $(U_{k,\nu,P},C)$. However, an example of a neighborhood without convergent foliation is given 
by Mishustin in \cite{Mishustin}, and it is expected to be a generic property.
In this paper, we describe the analytic classification of neighborhoods with Ueda type $k=1$ in the most simple situation (Serre's example).
We also provide in Sections \ref{S:Generalization} and \ref{S:torsion} some evidence to the fact that 
a similar result holds more generally for torsion normal bundle $N_C^{\otimes m}=\mathcal O_C$, 
and finite Ueda type $k<\infty$. As we shall see, the moduli space is comparable with $\C\{X,Y\}$.

\subsection{The fundamental isomorphism}\label{sec:IntroFondamIso}
In order to explain our classification result, it is convenient to recall the following classical construction. 
For the simplest formal type $(k,\nu,P)=(1,0,0)$, the neighborhood $(U_{1,0,0},C)$ actually embeds
into a ruled surface  $S_0\to C$, namely one of the two indecomposable ruled surfaces over $C$ after Atiyah \cite{Atiyah}.
Indeed, setting $y=1/\xi$, the ruled surface is defined as the quotient 
$$S_0=\tilde U_0/<F_{1,0,0}>\ \ \ \text{where}\ \ \ \tilde U_0=\C_z^*\times\overline{\C}_\xi\ \ \ 
\text{and}\ \ \ F_{1,0,0}(z,\xi)=(qz,\xi-1)$$
and the infinity section $\xi=\infty$ defines the embedding of the curve $C\subset S_0$.
The complement of the curve $S_0\setminus C$ is known to be isomorphic to the moduli space
of flat line bundles\footnote{i.e. lines bundles together with a holomorphic connection} 
over the elliptic curve, and has the structure of an affine bundle. 
The Riemann-Hilbert correspondance provides an analytic isomorphism with the space of 
characters $\mathrm{Hom}(\pi_1(C),\mathrm{GL}_1(\C))$, which is isomorphic to $\C^*\times\C^*$.
Explicitely, the isomorphism is induced on the quotient $S_0$ by the following map
$$\Pi:S_0\setminus C \stackrel{\sim}{\longrightarrow}\C_X^*\times\C_Y^*\ ;\ (z,\xi)\mapsto(e^{2i\pi\xi},ze^{2i\pi\tau\xi}).$$
In this sense, we can view $S_0$ and $\PP^1_X\times\PP^1_Y\supset\C_X^*\times\C_Y^*$ as two non algebraically
equivalent compactifications of the same analytic variety. In fact, the algebraic structures of the two open sets
are different as $\C_X^*\times\C_Y^*$ is affine, while $S_0\setminus C$ is not: there is no non constant regular function on it.
This construction, due to Serre, provides an example of a Stein quasiprojective variety which is not affine
(see \cite[page 232]{Hartshorne}).
Denote by $D\subset\PP^1_X\times\PP^1_Y$ the compactifying divisor, union of four projective lines:
$$D=L_1\cup L_2\cup L_3\cup L_4\ \ \ \text{with}$$
$$L_1:\{Y=0\},\ \ \ L_2:\{X=\infty\},\ \ \ L_3:\{Y=\infty\}\ \ \ \text{and}\ \ \ L_4:\{X=0\}$$

\begin{figure}[htbp]
\begin{center}
\includegraphics[scale=0.4]{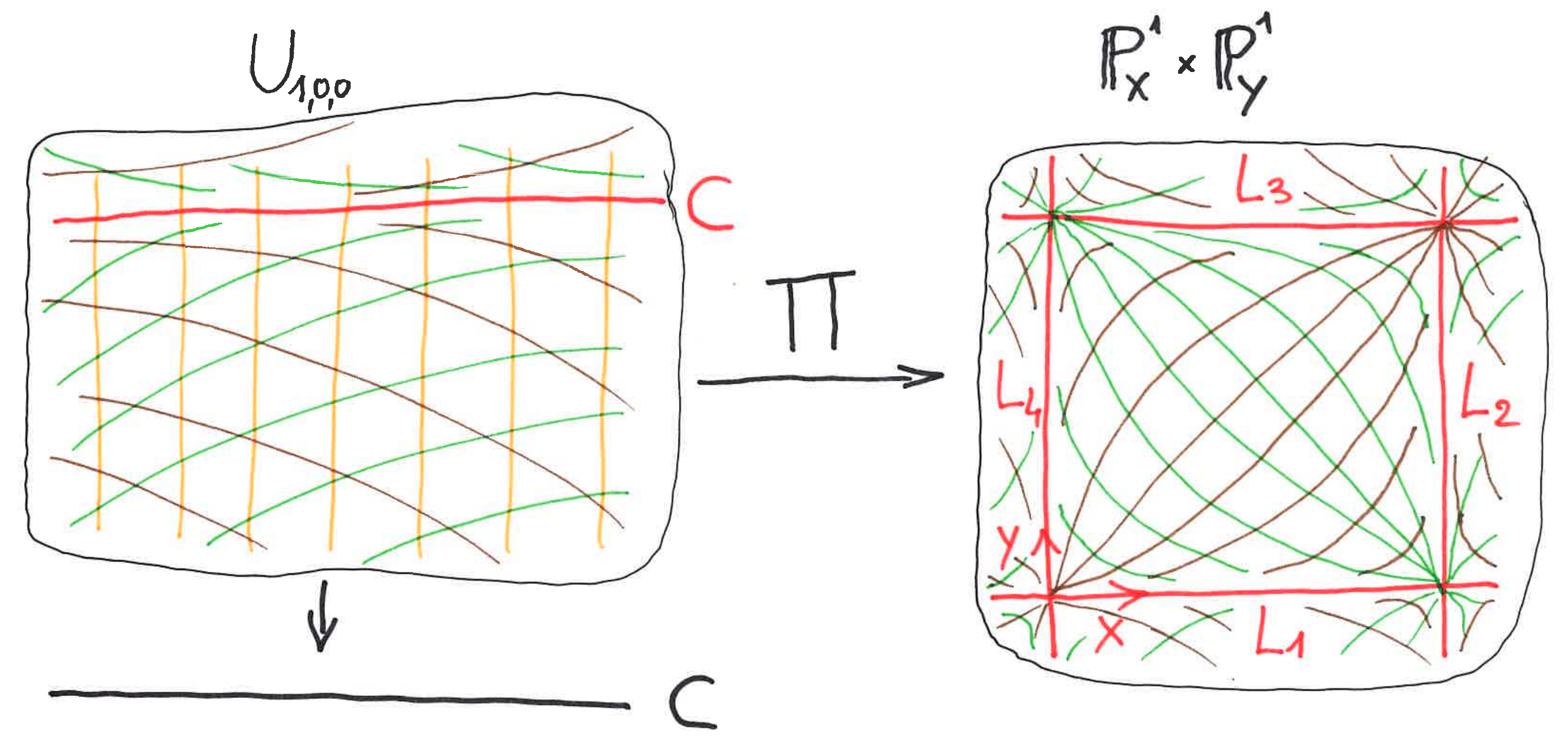}
\caption{{\bf Serre isomorphism}}
\label{pic:SerreIsom}
\end{center}
\end{figure}

Logarithmic one-forms with poles supported on $D$ correspond to the space
of closed one-forms considered above via the isomorphism:
$$\left\{\begin{matrix}\omega_0=d\xi\hfill\\ \omega_\infty=\frac{1}{2i\pi\tau}\frac{dz}{z}\end{matrix}\right.\ \ \ \text{and}\ \ \ 
\left\{\begin{matrix}\frac{1}{2i\pi}\frac{dX}{X}=\omega_0\hfill\\ \frac{1}{2i\pi}\frac{dY}{Y}=\tau(\omega_0+\omega_\infty)
\end{matrix}\right.$$
Therefore, at the level of foliations, we have the following correspondance
\begin{equation}\label{eq:DefFolt}
\mathcal F_t:\{\omega_0+t\omega_\infty=0\}\ \ \ \longleftrightarrow\ \ \ (1- t)\tau\frac{dX}{X}+t\frac{dY}{Y}=0.
\end{equation}
In particular, for $m+\tau n\in\Gamma\sim\pi_1(C)$ in the lattice, the unique foliation with trivial holonomy along  $m+\tau n$ 
corresponds to the one with rational first integral $X^mY^n$:
$$\mathcal F_{\frac{\tau n}{m+\tau n}}\ \ \ \longleftrightarrow\ \ \ m\frac{dX}{X}+n\frac{dY}{Y}=0$$
and the ruling corresponds to a foliation with transcendental leaves:
$$\mathcal F_\infty\ \ \ \longleftrightarrow\ \ \ \tau\frac{dX}{X}-\frac{dY}{Y}=0.$$
Let us now study the isomorphism $\Pi:S_0\setminus C\to \PP^1\times\PP^1\setminus D$ near the compactifying divisors.
Denote by $V_i$ a tubular neighborhood of $L_i$ in $\PP^1_X\times\PP^1_Y$, of the form $L_i\times\text{disc}$ say,
and $V=V_1\cup V_2\cup V_3\cup V_4$ the corresponding neighborhood of $D$. 
Denote by $V^*=V\setminus D$   the complement of $D$ in $V$. 
On may think of $U_{1,0,0}\setminus C$ as $\Pi^{-1}(V^*)$ \footnote{Actually, this correspondance is meaningful 
in the germified setting in the sense that a basis of neighborhood $(V^\alpha)$ of $D$ gives rise to a basis 
of neighborhoods $(U^\alpha)$ of $C$ where $U^\alpha=\Pi^{-1}( V^\alpha\setminus D)\cup C$. }. 
Similarly, define  $V_i^*:=V_i\setminus D$ the complement of the divisor in $V_i$ and by $U_i=\Pi^{-1}(V_i^*)$
the preimage: we have a decomposition neighborhood $U_{1,0,0}\setminus C=U_1\cup U_2\cup U_3\cup U_4$.
One can show that $U_i's$ look like sectorial domains of opening $\pi$ in the variable $y$ 
saturated by variable $z$ (see Section \ref{SS:sheaves}). Our main result is that this sectorial decomposition together
with isomorphisms $\Pi_i:U_i\to V_i^*$ persists for general neighborhoods $(U,C)$ in the formal class $(U_{1,0,0},C)$;
we conjecture and actually give the strategy to prove that a similar result holds true for all formal types, whenever $N_C$ is torsion (see Section \ref{S:Generalization} ).

\begin{figure}[htbp]
\begin{center}
\includegraphics[scale=0.4]{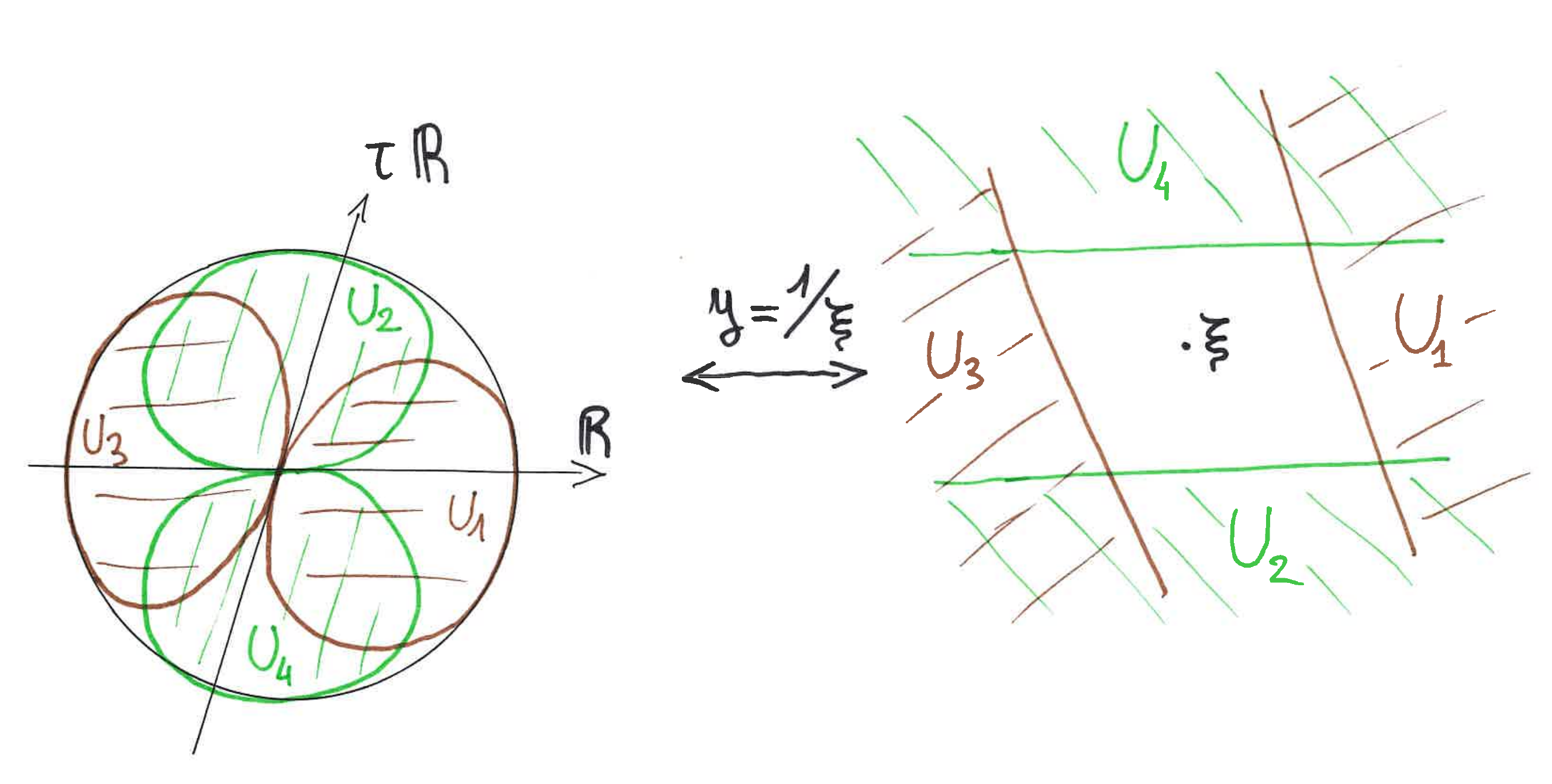}
\caption{{\bf Neighborhoods correspondence}}
\label{pic:NeighCor}
\end{center}
\end{figure}

\subsection{Analytic classification: main result}\label{sec:IntroAnalClass}
A general neighborhood $(U,C)$ formally conjugated to $(U_{1,0,0},C)$ can be described as quotient (see Proposition \ref{prop:PrenormalFormU100})
$$U=\tilde U/<F> \ \ \ \text{where}\ \ \ \tilde U\subset\C_z^*\times\C_y$$
is a neighborhood of the zero section $\tilde C=\{y=0\}$, and 
$$F(z,y)=(qz+O(y^2),y+y^2+y^3+O(y^4)).$$
There is a formal isomorphism 
$$\hat\Psi=\left(z+\sum_{m\ge2}a_m(z)y^m, y+\sum_{n\ge4}b_n(z)y^n\right)$$
such that $\hat\Psi\circ F=F_{1,0,0}\circ\hat\Psi$; we have $a_m,b_n\in\mathcal O(\C^*_z)$ 
and no convergence assumption in $y$-variable. We can also consider $\hat\Psi$ as 
a formal diffeomorphism $(U,C)\to (U_{1,0,0},C)$.
The main ingredient of our classification result, proved in Section \ref{S:sectorialnorm1}, is the

\begin{LEM}\label{LEM:sectorialnormalization}{\bf Sectorial normalization.}
 Denote $\varpi=\arg\tau$. For each interval
\begin{equation}\label{eq:sectory}
I_1=]\varpi-\pi,\varpi[,\ \ \ I_2=]0,\pi[,\ \ \ I_3=I_1+\pi,\ \ \ I_4=I_2+\pi
\end{equation}
there is a transversely sectorial domain\footnote{Given an interval $I=[\theta_1,\theta_2]\subset\R$, 
an open subset $U_0\subset U$ is said to be a \textit{transversely sectorial of opening $I$} 
if the lift $\tilde U_0\subset\tilde U\subset \C_z^*\times\C_y$
contains, for arbitrary large and relatively compact open set $\mathcal C\Subset\C^*$ and arbitrary small $\epsilon>0$, a sector $\mathcal C\times\mathcal S(I_\varepsilon,r)$ where
$$\mathcal S(I_\varepsilon,r)=\{y\in\C\ ;\ \arg(y)\subset I_\varepsilon,\ 0<|y|<r\},\ \ \ I_\varepsilon=]\theta_1+\varepsilon,\theta_2-\varepsilon[$$
for some $r>0$ (See also Definitions \ref{def:sector} and \ref{def:sectorialopen}).} 
$U_i\subset U$ of opening $I_i$ and a diffeomorphism 
$$\Psi_i:U_i\to U_{1,0,0}$$
(onto its image) having $\hat\Psi$ as asymptotic expansion\footnote{The diffeomorphism $\Psi_i:U_i\to U_{1,0,0}$ 
admits $\hat\Psi_i$ as an asymptotic expansion along
$C$ if the entries 
of its lift $\tilde\Psi_i:\tilde U_i\to\C_z^*\times\C_y$ 
admit the entries of $\hat\Psi$ as asymptotic expansion on each open subset $\mathcal C\times\mathcal S(I_\varepsilon,r)$
(see Section \ref{SS:sheaves}).} 
along $C$, satisfying 
$$\Psi_i\circ F=F_{1,0,0}\circ\Psi_i.$$
\end{LEM}

After composition with the fundamental isomorphism $\Pi:U_{1,0,0}\to\C^*_X\times\C^*_Y$, we get

\begin{COR}\label{COR:cocycle}
The composition $\Pi_i=\Pi\circ\Psi_i$ provides an isomorphism germ
$$\Pi_i:(U_i,C)\to(V_i^*,L_i)\ \ \ \text{such that}\ \ \ \Pi_{i}=\varphi_{i,i+1}\circ\Pi_{i+1}\ \text{on}\ U_i\cap U_{i+1}$$
for some diffeomorphism germs $\varphi_{i,i+1}\in\mathrm{Diff}(V_{i,i+1},p_{i,i+1})$ \repnote
\end{COR}

\begin{figure}[htbp]
\begin{center}
\includegraphics[scale=0.4]{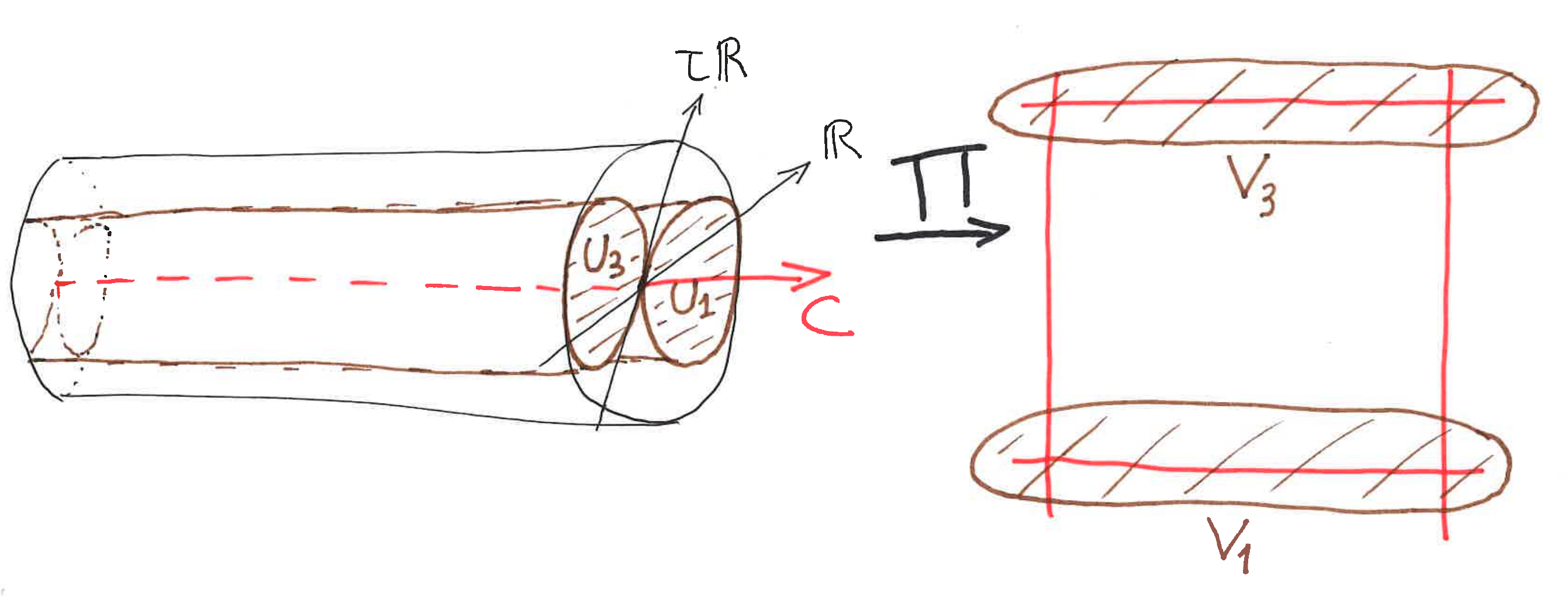}
\caption{{\bf Sectorial normalization}}
\label{pic:SecNorm}
\end{center}
\end{figure}

After patching copies of germs $(V_i,L_i)\simeq(\PP^1_X\times \PP^1_Y,L_i)$ by the $\varphi_{i,i+1}:(V_{i+1},p_{i,i+1})\to(V_{i},p_{i,i+1})$,
we get a new neighborhood germ $(V_\varphi,D)$ of the divisor $D$, 
where $\varphi=(\varphi_{i,i+1})_{i\in\Z_4}$, together with a diffeomorphism germ
$$\Pi:(U\setminus C,C)\stackrel{\sim}{\longrightarrow}(V_\varphi\setminus D,D)$$
which does not depend on the choice of sectorial normalisations $\Psi_i$.

More generally, consider a neighborhood $(V,D)$ in which each component $L_i\subset D$
has zero self-intersection. Then after \cite{Savelev}, the neighborhood $(V,L_i)$ is trivial
(a product $L_i\times\text{disc}$). After identification with our model 
$\psi_i:(V,L_i)\stackrel{\sim}{\rightarrow}(\PP^1_X\times \PP^1_Y,L_i)$, we get that $V$ takes the form $V_\varphi$ 
for a convenient $4$-uple of diffeomorphisms $\varphi$. The gluing data $\varphi$ is not unique as we can 
compose each embedding $\psi_i$ by an automorphism germ $\varphi_i\in\mathrm{Diff}(V_{i},L_{i})$ \repnote
Therefore, it is natural to introduce the following equivalence relation 
$$\varphi\sim\varphi'\ \ \ \Leftrightarrow\ \ \ 
\exists\left(\varphi_i\in\mathrm{Diff}(V_{i},L_{i})\right)_{i\in\Z_4}\ \text{such that}\ \varphi_{i}\circ\varphi_{i,i+1}'=\varphi_{i,i+1}\circ\varphi_{i+1}.$$
Clearly, the moduli space $\mathcal V$ of neighborhoods $(V,D)$ 
up to analytic equivalence\footnote{One requires each component of the cycle to be preserved} 
identifies
with the set of equivalence classes for $\sim$. Notice that each equivalence classe contains a representative
$\varphi$ such that $\varphi_{1,2},\varphi_{2,3},\varphi_{3,4}$ are tangent to the identity, and the linear part 
$$\varphi_{4,1}(X,Y)=(aX+\cdots,bY+\cdots)$$
does not depend on the choice of such representative $\varphi$. Therefore, $a,b\in \C^*$ are invariants 
for the equivalence relation, and we denote by $\mathcal V_{a,b}$ the moduli space of those triples.
With this in hand, we are able to prove:

\begin{THM}\label{THM:MainClassif}
We have a one-to-one correspondance between
$$\mathcal U_{1,0,0}\leftrightarrow\mathcal V_{1,1}$$
\begin{itemize}
\item the moduli space $\mathcal U_{1,0,0}$ of neighborhoods $(U,C)$ formally equivalent to $(U_{1,0,0},C)$ up to analytic equivalence\footnote{More precisely, we allow for this statement analytic isomorphisms inducing translations on $C$; see
Proposition \ref{prop:AnalyticClassif1} for a more precise statement.}
\item the moduli space $\mathcal V_{1,1}$ of neighborhoods $(V_\varphi,D)$ with all $\varphi_{i,i+1}$ tangent to the identity.
\end{itemize}
\end{THM}

\begin{remark}
A thorough look to the proof of the Sectorial Normalization Lemma (\ref{S:sectorialnorm1}) may prove that the correspondance is analytic in the sense that analytic families of neighborhoods $t\mapsto (U_t,C)$
correspond to analytic families of cocycles $t\mapsto\varphi_t$. 
As the freedom lies in the choice of (essentially) one-dimensional diffeomorphisms $\varphi_i$, it is quite clear that 
the moduli space is essentially parametrized by two-dimensional diffeomorphisms, and therefore quite huge. 
\end{remark}

In a similar vein, it is reasonable to expect that the analytic moduli space $\mathcal U_{1,\nu,\mu}$ of neighborhoods $(U,C)$ formally equivalent to $(U_{1,\nu,\mu},C)$ is in one to one correspondance with $\mathcal V_{a,b}$ 
with $a=e^{-4\pi^2\nu}$ and $b=e^{-4\pi^2\tau(\nu+\mu)}$. Actually, we explain in Section \ref{S:Generalization}
how to construct an embedding $\mathcal V_{a,b}\hookrightarrow \mathcal U_{1,\nu,\mu}$,
but the surjectivity needs to adapt our Sectorial Normalization Lemma. This creates additional issues (of purely technical nature) and we just indicate briefly how to proceed.
Actually, one directly adresses in loc.cit  the general case   Ueda type $= k$, where we inherit $4k$  sectors with opening $\frac{\pi}{k}$ 
and the moduli space would be then equivalent to the moduli of neighborhoods of cycles 
of $4k$ rational curves (the model must be thought as a degree $k$ cyclic \'etale cover of $(V,D)$). 
A precise statement, summarizing the structure of the analytic moduli space when $N_C\backsimeq\mathcal {\mathcal O}_C$, is given in Section \ref{S:Generalization}, Theorem \ref{prop:AnalyticClassifgen}.
With this in hand, it is not difficult to undertake the analytic classification of neighborhood when $N_C$ is torsion. The idea consists in reducing to the case of trivial normal bundle by an appropriate cyclic cover. This is settled in Section \ref{S:torsion}.

\subsection{Foliations}

A neighborhood $(U,C)$ formally conjugated to $(U_{1,0,0},C)$ admits a pencil of formal foliations
$\hat{\mathcal F}_t$ corresponding to ${\mathcal F}_t$ in (\ref{eq:DefFolt}) via the formal normalization $\hat\Psi$.

\begin{THM}\label{THM:ExistenceFoliations}
The foliation $\hat{\mathcal F}_t$ is convergent if, and only if, there exists a representative
$\varphi$ in the corresponding equivalence class such that each $\varphi_{i,i+1}$ preserves the 
foliation 
$$(1-t)\tau\frac{dX}{X}+t\frac{dY}{Y}=0.$$ 
In that case, these two foliations are conjugated
via the isomorphism $U\setminus C\to V\setminus D$.
\end{THM}

When $\mathcal F_t$ is not of rational type, i.e. $\tau(1-\frac{1}{t})\not\in\mathbb Q\cup\{\infty\}$, 
then $\mathcal F_t$ is defined by a closed meromorphic $1$-form and the logarithmic $1$-form
of the statement is also preserved by all $\varphi_{i,i+1}$ and defines a global logarithmic $1$-form
on $(V,D)$. On the other hand, in the rational case, \'Ecalle-Voronin moduli of the holonomy provide
obstruction to define the foliation by a closed meromorphic $1$-form. For instance, when $\mathcal F_0$ is convergent,
Martinet-Ramis cocycle are given by the $X$-coordinate of $\varphi_{1,2}\circ\varphi_{2,3}$ and $\varphi_{3,4}\circ\varphi_{4,1}$
(see Section \ref{prop:ExampleMartinetRamis} for details).

In \cite{LTT}, the two first authors with O. Thom provided the analytic classification of neighborhoods with $2$ convergent foliations.
In Section \ref{sec:Foliations}, we provide examples of neighborhoods with only one foliation, and also without foliation
which is the generic case. An example without foliations has been given by Mishustin in \cite{Mishustin} few years ago and 
it would be nice to understand what is the corresponding invariant  $\varphi$.

In Section \ref{sec:symmetries}, we investigate the automorphism group of neighborhood germs.
We prove in Theorem \ref{thm:symmetries} that it can be of three types: finite (the generic case), 
one dimensional and we get an holomorphic vector field (and in particular a convergent foliation), or two dimensional
only in the Serre example.

\subsection{$\mathrm{SL}_2(\Z)$ action}\label{ss:SL2intro}
The analytic classification of resonant diffeomorphism germs of one variable is reminiscent in our
classification result. However, there are strong differences like the fact that the sectorial trivialization 
is not unique in our case. 
Indeed, our sectorial decomposition $U\setminus C=U_1\cup U_2\cup U_3\cup U_4$
has been imposed by our choice of a basis for the lattice $\Gamma=\Z+\tau\Z$.
It comes from the sectorial decomposition of the holonomy maps of the two foliations $\mathcal F_0$ 
and $\mathcal F_1$ having cyclic holonomy, trivial along $1$ and $\tau$ respectively.
If we change for another basis  $(m+\tau n,m'+\tau n')$, with 
$$\begin{pmatrix}m& m'\\ n& n'\end{pmatrix}\in\mathrm{SL}_2(\Z)$$
then the change of coordinates 
$$x'=\frac{x}{m+\tau n},\ \ \ \xi'=(m+\tau n)\xi+nx\ \ \ \rightsquigarrow\ z'=e^{2i\pi x'}=z^{\frac{1}{m+\tau n}}$$
gives $(S_0,C)$ as the quotient of $\C_{z'}^*\times\overline{\C}_{\xi'}$ \footnote{It may be useful to think of $\C_{z'}^*\times\overline{\C}_{\xi'}$ as the cyclic cover of $S_0$ associated to the subgroup $\langle m+n\tau \rangle$ of $\Gamma$ (see \ref{SS:overview}).} 
by the transformation
$$(z',\xi')\mapsto(q' z',\xi'-1),\ \ \ q'=e^{2i\pi\tau'},\ \ \ \tau'=\frac{m'+\tau n'}{m+\tau n}.$$
The new isomorphism is related to the previous one by a monomial transformation
$$(X',Y')=(e^{2i\pi\xi'},{z'}e^{2i\pi\tau'\xi'})=(X^mY^{n},X^{m'}Y^{n'}).$$
Using sectorial normalization for a general neighborhood $(U,C)$ with this new basis gives 
a new compactification $(V',D)$ which is bimeromorphically equivalent to $(V,D)$.

\subsection{Concluding remarks}
Contrary to the diophantine case (non torsion normal bundle), the classification of neighborhoods of elliptic curves with torsion 
normal bundle can be completely described, as shown in Theorem \ref{prop:AnalyticClassifgen}. A naive
reading of Arnold's work \cite{Arnold} might suggest 
that classification of neighborhoods of elliptic curves with topologically trivial normal bundle could be  similar to that
of germs of one dimensional diffeomorphisms. In fact, the suspension of a representation $\Pi_1(C)\to\Diff$
permits to embed the moduli space of diffeomorphisms into that of neighborhoods. However, this latter one
turns out to be much more complicated, even if the general approach by sectorial normalization and classifying cocycle
is still in the spirit of \'Ecalle-Voronin classification for resonant diffeomorphisms, or Martinet-Ramis' version. 
For instance, an unexpected phenomenon in the case of neighborhoods is that the sectorial covering is not unique,
due to the $\mathrm{SL}_2(\Z)$-action. We can expect that the sectorial normalizations involve resurgent functions
with lattice of singularities isomorphic to the lattice of the elliptic curve. Recall that the lattice of resurgence for 
resonant diffeomorphism has rank one. It would be interesting to better understand this phenomenon. 

An important motivation to study neighborhoods was initially raised by Arnold: 
there is a close link with the study of germs of analytic
diffeomorphisms of $(\C^2,0)$. Indeed, if we consider our model $F_{1,0,0}(z,y)=(qz,\frac{y}{1-y})$ at the neighborhood
of $(z,y)=(0,0)$, then we get a semi-hyperbolic map whose space of orbits (when deleting $z=0$) is obviously
the neighborhood $(U_{1,0,0},C)$ where $C=\C/<qz>$. One can investigate the analytic classification of small perturbations
$F:=F_{1,0,0}+\cdots$ where dots are vanishing at sufficiently high order at the origin. Then, it is a classical fact that 
$F$ has also an invariant manifold in the contracting direction $\vert q\vert<1$, and the space of orbits gives rise 
to a neighborhood $(U,C)$ formally equivalent to $(U_{1,0,0},C)$. The analytic classification of these germs of 
semi-hyperbolic maps has been done by the last author with P. A. Fomina-Sha\u{\i}khullina (see \cite{VF,SV})
and comparing the two moduli shows that moduli of maps embed in moduli of neighborhoods but is infinite 
codimensional. In fact, the analytic extension of the map $F$ to the origin imposes strong restrictions on the 
corresponding invariants $\varphi$ defined in subsection \ref{sec:IntroAnalClass}.
It is interesting to consider the following hierarchy:
\begin{enumerate}
\item one dimensional resonant diffeomorphisms in $(\C,0)$,
\item singular points of foliation  in $(\C^2,0)$ of resonant-saddle or saddle-node type,
\item singular points of vector fields in $(\C^2,0)$ of resonant-saddle or saddle-node type,
\item singular points of diffeomorphisms in $(\C^2,0)$ of resonant-saddle or saddle-node 
(i.e. semi-hyperbolic) type,
\item neighborhoods of elliptic curves with torsion normal bundle.
\end{enumerate}
The first occurence gives rise to \'Ecalle-Voronin moduli (see \cite{Ecalle2,Voronin}, and also \cite{Malgrange}). 
One dimensional resonant diffeomorphisms also occur as monodromy map of those foliations
arising in case (2). These latter ones have been classified by Martinet-Ramis in \cite{MartinetRamis,MartinetRamis2}
and the classification on resonant-saddles and their monodromy map (1) turns out to be equivalent;
however, saddle-node impose strong restriction to the invariants of its holonomy map
(we can realize half of the moduli only). See \cite{MartinetRamis,MartinetRamis2,Malgrange} for details.
Classification of vector fields has been done by the third author with Meshcheryakova
(see \cite{VM} for instance) and independently by Teyssier \cite{Teyssier}. 
This gives rise to twice the moduli space of foliations:
the classification of vector fields with same underlying foliation is solved by the linearization 
of the conjugacy equation for foliations. Still, the moduli space is parametrized by finitely many
copies of $\C\{x\}$ (power-series in one variable). There is a huge step when we pass to diffeomorphisms in $(\C^2,0)$
as the moduli space is now parametrized by copies of $\C\{x,y\}$. Diffeomorphisms occuring in (4)
are actually one-time-map of formal vector fields of type (3), but divergent as a rule.
We expect that resonant-saddle diffeomorphisms (4)
have same classification as neighborhoods (5), but classification of former ones looks somehow more delicate.
It would also be nice to understand how Ueda's results \cite{Ueda1,Ueda2} can be related to our work
from that point of view. 
Also, diffeomorphisms of $(\C^2,0)$ arise as monodromy map of reduced singular foliations by curves in $(\C^3,0)$
and we can expect that the analytic classification is similar under generic conditions on the spectrum.

One might expect to investigate higher dimensional neighborhood of elliptic curve with trivial normal bundle
by mimicking what has been done in \cite{LTT} for the formal classification, and in the present paper regarding
analytic classification. We haven't considered this direction at all. However, it might be interesting to note
that Ueda's Theory has been generalized to higher codimension by Koike in \cite{Koike}. What would be 
the higher dimensional analogue of Serre's isomorphism ?

\section{Preliminary remarks}\label{S:preliminary}

Recall that $C=\C^*/<q>$, and we denote by $\tilde C\simeq\C^*\to C$ the corresponding cyclic cover.
Denote $\tilde U=\C_z^*\times \C_y$ and $\tilde C=\{y=0\}\subset\tilde U$. The following is already mentioned 
by Arnol'd \cite{Arnold}.
 
\begin{lemma}\label{L:firstpresentation}
Any germ of neighborhood $(U,C)$ with $C^2=0$ is biholomorphic to a germ of the form 
$(\tilde U,\tilde C)/<F>$ where
\begin{equation}\label{eq:SiuNeighborhood}
F(z,y)= (q z + yf(z,y), \lambda(z) y+ y^2g(z,y))
\end{equation}
with $f,g$ holomorphic on a neighborhood of $\{y=0\}$, where $q=e^{2i\pi\tau}$ and $\lambda\in{\OO}^*(\C_z^*)$.
\end{lemma}

\begin{proof}
 The self-intersection $C\cdot C$ determines topologically the germ of neighborhood.
Then, by taking a suitable  small representative, $U$ is homeomorphic to a product $\mathbb D \times C$. 
So, one can consider the cyclic covering $\tilde U\to U$ extending the cyclic cover $\tilde C\to C$.
This gives rise to a neighborhood $\tilde U$  of $\tilde C\simeq{\C}^*$. Following Siu \cite{Siu}, 
the germ of this neighborhhood along $\tilde C$ is isomorphic to the germ of a neighborhood 
of the zero section $\{y=0\}$ in the normal bundle $N_{\tilde C}\simeq\C^*_z\times \C_y$. 
The deck transformation of the (germ of) covering takes the form $F$ of the statement.
\end{proof}

\begin{dfnprop}\label{def:an/forequivalence} \footnote{Actually, one can, as usually, define  analytic/formal conjugations between both neighborhoods in terms involving only the structural analytic/formal sheaves along $C$. 
The two definitions obviously coincide.} 
Any two quotients $(\tilde U,\tilde C)/<F>$ and $(\tilde U,\tilde C)/<F'>$ are analytically
(resp. formally) equivalent, and we note
$$(U,C)\an(U',C)\ \ \ \text{(resp. $(U,C)\formal(U',C)$)},$$
if there is a germ of analytic (resp. formal) diffeomorphism 
\begin{equation}\label{eq:PrenormNeigh100}
\Psi(z,y)=\left(z+\sum_{n=1}^\infty a_n(x)y^n, \sum_{n=1}^\infty b_n(x)y^n\right)\ \ \ 
\text{such that}\ \ \ \Psi\circ F=F'\circ \Psi.
\end{equation}
\end{dfnprop}

Although the formal classification is already done in \cite{LTT}, we need the following formulation
and give some basic steps.

\begin{prop}\label{prop:PrenormalFormU100}
A germ of neighborhood $(U,C)$ is formally equivalent to 
$$(U_0,C)=(\tilde U, \tilde C)/<F_{1,0,0}>,\ \ \ F_{1,0,0}(z,y)=\left(qz,\frac{y}{1-y}\right)$$ 
if, and only if, it is biholomorphic to a germ of the form $(\tilde U, \tilde C)/<F>$ where
\begin{equation}\label{eq:PrenormalFormU100}
F(z,y)= (q z + y^2f(z,y), y+y^2+y^3+ y^4g(z,y)).
\end{equation}
In that case, there exists a formal diffeomorphism (tangent to the identity along $C$)
\begin{equation}\label{eq:FormalDiffeoTangentIdentity}
\hat\Psi(z,y)=\left(z+\sum_{n>0}a_n(z)y^n,y+\sum_{n>1}b_n(z)y^n\right)
\end{equation}
with $a_n,b_n\in\mathcal O(\C^*_z)$ (and no convergence condition on $y$), such that
\begin{equation}\label{eq:FormalConjugacy}
\hat\Psi\circ F=F_{1,0,0}\circ \hat\Psi.
\end{equation}
Moreover, any other formal diffeomorphism $\hat\Psi'$ of the form (\ref{eq:FormalDiffeoTangentIdentity})
satisfying (\ref{eq:FormalConjugacy}) writes
\begin{equation}\label{eq:LackUnicityFormalConjugacy}
\hat\Psi'=\Phi\circ\hat\Psi\ \ \ \text{where}\ \ \ \Phi(z,y)=(z,\frac{y}{1-ty})=(z,y+ty^2+\cdots),\ \ \ t\in\C.
\end{equation}
\end{prop}

\begin{proof}Let $g=\sum_{n\in\Z}g_nz^n$ be holomorphic on $\C_z^*$. The functional equation 
\begin{equation}\label{eq:qdiffeq0}
\phi(qz)-\phi(z)=g(z)
\end{equation}
admits a solution $\phi$ holomorphic on $\C_z^*$ if, and only if, $g_0=0$; then $\phi$ is unique up to the choice of $\phi(0)$.
Indeed, if we write $\phi(z)=\sum_{n\in\Z}\phi_nz^n$, then equation (\ref{eq:qdiffeq0}) writes $\phi_n(q^n-1)=g_n$ for all $n$.

Let $f$ be a holomorphic non vanishing function on $\C_z^*$. The functional equation 
\begin{equation}\label{eq:qdiffeq}
\varphi(qz)/\varphi(z)=f(z)
\end{equation}
admits a solution $\varphi$ holomorphic and non vanishing on $\C_z^*$ if, and only if, 
\begin{itemize}
\item $f:\C^*\to\C^*$ has topological index 0 so that $g=\log(f)$ is well-defined,
\item the coefficient $g_0$ of $g=\sum_{n\in\Z}g_nz^n$ vanishes.
\end{itemize}
Indeed, topological index is multiplicative and those of $\varphi(qz)$ and $\varphi(z)$ 
are equal and cancel each other.
Then we can solve the corresponding equation (\ref{eq:qdiffeq}) for $g$ and set $\varphi=\exp(\phi)$,
which is unique up to a multiplicative constant. Note that, if $g_0\not=0$, then we can solve
\begin{equation}\label{eq:qdiffeqa}
\varphi(qz)/\varphi(z)=\frac{f(z)}{a}
\end{equation}
for $a=\exp(g_0)$.

Let us start with $F$ like in (\ref{eq:SiuNeighborhood}). The change of coordinate $\Psi_1(z,y)=(z,f(z)y)$ yields
$$\Psi_1^{-1}\circ F\circ\Psi_1(z,y)=(qz+O(y),\frac{\varphi(z)}{\varphi(qz)}f(z)y+O(y^2)).$$
We can easily check that the coefficient $f$ in $F$ defines the normal bundle $N_C$ in the quotient,
and its topological index coincides with $\deg(N_C)$ which is zero in our case. Then we can find
$\varphi\in\OO^*(\C_z^*)$ satisfying (\ref{eq:qdiffeq}) and get
$$F_1(z,y)=\Psi_1^{-1}\circ F\circ\Psi_1(z,y)=(qz+O(y),ay+O(y^2)).$$
Moreover, $\varphi$ is unique up to a multiplicative constant. The coefficient $a$ can be interpreted 
as a flat connection on $N_C$ with trivial monodromy along the loop $1\in\Gamma$
and monodromy $a$ along the loop $\tau\in\Gamma$. 
In our case, $N_C=\OO_C$ and $a=1$ and we can write 
$$F_1(z,y)=(qz+O(y),y+g(z)y^2+O(y^3)).$$
Now the change of coordinate $\Psi_2(z,y)=(z,y+\phi(z)y^2)$ gives
$$\Psi_2^{-1}\circ F_1\circ\Psi_2(z,y)=(qz+O(y),y+[g(z)+\phi(z)-\phi(qz)]y^2+O(y^3)).$$
Solving equation (\ref{eq:qdiffeq0}), we get 
$$F_2(z,y)=\Psi_2^{-1}\circ F_1\circ\Psi_2(z,y)=(qz+O(y),y+by^2+O(y^3)).$$
In our case, $b\not=0$ (i.e. Ueda type $k=1$). By using a change $(z,\lambda y)$
(freedom in the choice of $\varphi$ above) we can set $b=1$ and write 
$$F_2(z,y)=(qz+zf(z)y+ O(y^2),y+y^2+g(z)y^3+O(y^4)).$$ 
The change of coordinate 
$\Psi_3(z,y)=(z+\varphi(z)y,y+\phi(z)y^3)$ gives 
$$F_3(z,y)=\Psi_3^{-1}\circ F_2\circ\Psi_3(z,y)=$$
$$(qz+z[f(z)+\varphi(z)-\varphi(qz)]y+O(y^2),y+y^2+[g(z)+\phi(z)-\phi(qz)]y^3+O(y^4)).$$
Solving twice equation (\ref{eq:qdiffeq0}), we get 
$$F_3(z,y)=(qz+\alpha zy+O(y^2),y+y^2+\beta y^3+O(y^4)).$$
Here, we have no freedom and $\alpha,\beta$ are formal invariants corresponding to $\mu,\nu$ in the end of Section \ref{sec:IntroAnalClass}: in the formal class $U_{1,0,0}$
we get $\alpha=0$ and $\beta=1$. Then, we can kill-out all higher order terms in $F$ by a formal change of 
coordinate, or better normalize it to $F_{1,0,0}$. Indeed, at the $N^{th}$ step, we get 
$$F_N(z,y)=(qz+zf(z)y^{N-1}+O(y^N),y+y^2+y^3+\cdots+g(z)y^{N+1}+O(y^{N+2}));$$
the coordinate change $\Psi_{N+1}(z,y)=(z+azy^{N-2}+\varphi(z)y^{N-1},y+by^N+\phi(z)y^{N+1})$ gives
$$F_{N+1}(z,y)=\Psi_{N+1}^{-1}\circ F_N\circ\Psi_{N+1}(z,y)=$$
$$(qz+z[f(z)+\varphi(z)-\varphi(qz)-(N-2)aq]y^{N-1}+O(y^N),$$
$$y+y^2+y^3+\cdots+[g(z)+\phi(z)-\phi(qz)-(N-4)b]y^{N+1}+O(y^{N+2})).$$
We can clearly normalize the two coefficients into brackets by a constant, and can even choose
the constant by means of $a,b$. 

The composition of all changes of coordinates $\hat\Psi^{-1}:=\Psi_1\circ\Psi_2\circ\Psi_3\circ\cdots$
converges in the formal topology as a formal diffeomorphism satisfying (\ref{eq:FormalConjugacy}).
For any other formal diffeomorphism $\hat\Psi'$ of the form (\ref{eq:FormalDiffeoTangentIdentity})
satisfying (\ref{eq:FormalConjugacy}), we have that $\hat\Phi:=\hat\Psi'\circ\hat\Psi^{-1}$ is 
an automorphism of $(U_0,C)$ inducing the identity on $C$. 
As we shall see in Lemma \ref{L:formalcentralizer}, $\hat\Phi$ is necessarily convergent and of the form (\ref{eq:LackUnicityFormalConjugacy}).
\end{proof}

\section{Sectorial decomposition and sectorial symmetries}\label{sec:SectorialDecompSym}
In this section, we introduce the sectorial decomposition of $U$ by transversely sectorial domains $U_i=\Pi^{-1}(V_i^*)$
and compare spaces of functions on both sides. From now on, we work in the variable $\xi=1/y$,
at the neighborhood of $\xi=\infty$; this is much more convenient for computations.
Notations are as in Section \ref{sec:IntroFondamIso}.

\subsection{Some sheaves of functions on the circle of directions}\label{SS:sheaves}
Let $\bS^1:=\R/2\pi\mathbb Z$  and $I$ be an open interval of $\R$ 
(regarded as the universal covering of $\bS^1$). 

\begin{dfn}\label{def:sector}
For $(c,R)\in ]0,+\infty]\times [0,+\infty[,$ denote by
$$S(I,R;c)=\{(z,\xi)\in\C^*_z \times {\C}_\xi\ ;\ \arg(\xi)\subset I,\ R<|\xi|,\ e^{-c}<|z|<e^c\}.$$
A sector of opening $I$ is an open subset $\Sigma_I\subset S(I,0;\infty)$
such that for all $c>>0$, there exists $R_c>0$ such that
$$S(I,R_c;c)\subset\Sigma_I.$$
\end{dfn}

Let $\Sigma_I$ be an open sector as above. Then, ${\OO}(\Sigma_I)$ contains the subalgebra 
${\mc A}(\Sigma_I)$ of holomorphic functions admitting an asymptotic expansion along ${\C}_z^*$ in the sense defined below:

\begin{dfn}\label{def:asymptoticexpansion}
Let $m$ be a positive integer. A function $f\in{\OO}(\Sigma_I)$ belongs to ${\mathcal A}^m(\Sigma_I)$ 
if there exists a polynomial $P_m(f)=\sum_{0\leq k\leq m} a_k\xi^{-k}\in{\OO}({\C}_z^*)[\xi^{-1}]$ 
such that $\forall\ c>>0$,
$\exists\ C_{c},R_{c} >0$ such that $\forall (z,\xi)\in S(I,R_{c};c)\subset \Sigma_I$, we have

\begin{equation}\label{E:asymptexpan}
\left| f(z,\xi)-P_m(f)(z,\xi)\right|\leq \frac{C_{c}}{|\xi^{m+1}|}.
\end{equation}
Note that $P_m(f)$ is necessarily unique. 

Define $$\mathcal A(\Sigma_I)= \bigcap_m {\mathcal A}^m(\Sigma_I).$$ Then one can associate to $f$ its \textit{asymptotic expansion} along $\{\xi=\infty\}$. This is a formal power series $\hat f\in {\OO}({\C}_z^*)[[\xi^{-1}]]$ whose truncation at order $m$ coincide with $P_m(f)$. The asymptotic expansion is then unique, and we have a well-defined morphism of $\C$-algebras
$$\mathcal A(\Sigma_I)\to{\OO}({\C}_z^*)[[\xi]]\ ;\ f\mapsto \hat f,$$
whose kernel, denoted ${\mathcal A}^\infty(\Sigma_I)$, consists of flat functions.
\end{dfn}

When fixing only $I$ and taking inductive limits associated to restriction maps, 
the collection of algebras of the form ${\OO}(\Sigma_I)$ define an algebra of germs  ${\mathcal O}_I$.  
The presheaf on $\bS^1$ defined by $I\to {\OO}_I$ naturally gives rise to a sheaf on $\bS^1$ which we will denote by $\mathcal O$. One can define on the same way the sheaves ${\mathcal A}^m$, ${\mathcal A}$, $\mathcal A^\infty$ respectively associated to $I\to {\mathcal A}_I^m$, $I\to {\mathcal A}_I$, $I\to {\mathcal A}_I^\infty$ and the last two are sheaves of differential algebras with respect to $\partial_z$ and $\partial_\xi$.
The stability by derivation is indeed a straighforward consequence of Cauchy's formula. 
As the asymptotic expansion is independant of the representative, we have a morphism of sheaves
$$\mathcal A\to{\OO}({\C}_z^*)[[\xi]]\ ;\ f\mapsto \hat f$$
whose kernel is $\mathcal A_I^\infty$ (here ${\OO}({\C}_z^*)[[\xi]]$ is viewed as a constant sheaf over $\bS^1$).

\begin{remark}
Mind that the inclusion ${\mathcal O}_I\to {\mathcal O}(I)$ (resp. ${\mathcal A}_I\to {\mathcal A}(I)$) is  strict. 
For instance, one must think that a section  $f\in{\mathcal A}(I)$ can be represented for every interval $J\Subset I$ by a function belonging to ${\mathcal A}(\Sigma_J)$ for suitable sectors of opening $J$ but does not necessarily admit a representative on a sector of the form $\Sigma_I$. In other words, the domain of definition of $f$ is a transversely sectorial open set in the following sense.
\end{remark}

\begin{dfn}\label{def:sectorialopen}
Given an interval $I=]\theta_1,\theta_2[\subset\R$, an open subset $\Sigma\subset S(I,0;\infty)\subset \C_z^*\times\C_\xi$ is said transversely sectorial of opening $I$ if, for arbitrary large $c>>0$ and small $\epsilon>0$, there exists $R_{c,\epsilon}>0$ such that
$$\mathcal S(I_\epsilon,R_{c,\epsilon};c)\subset\Sigma,\ \ \ \text{where}\ I_\epsilon=]\theta_1+\epsilon,\theta_2-\epsilon[.$$
\end{dfn}

\begin{remark} The sheaves $\OO$, $\mc A$ and $\mc A^\infty$ are invariant
under the action of a diffeomorphism $F$ of the form (\ref{eq:PrenormalFormU100}) (expressed in the $(z,\xi)$ coordinates).
Moreover, this action is stalk-preserving due to the fact that $F$ is tangent to the identity
along $\tilde C$ on the transversal direction $\xi$. In particular, they
define similar sheaves of sectorial functions on the quotient $(U,C)=(\tilde U,\tilde C)/<F>$ by considering
those sections invariant under $F$. We will denote by $\OO[F]$, $\mc A[F]$ and $\mc A^\infty[F]$
these latter sheaves. In the next section, we characterize sections of $\mc A^\infty[F_{1,0,0}](I)$ for 
special intervals $I$.
\end{remark}

\subsection{Sectorial decomposition}\label{S:SectorialDecomposition}
Denote $\varpi=\arg(\tau)\in]0,\pi[$ and let us define\footnote{Mind that these intervals for $\arg(\xi)$ 
correspond to those defined in Lemma \ref{LEM:sectorialnormalization} for $\arg(y)=-\arg(\xi)$.} 
$$I_1=]-\varpi,\pi-\varpi[,\ \ \ I_2=]-\pi,0[,\ \ \ I_3=I_1+\pi\ \ \ \text{and}\ \ \ I_4=I_2+\pi.$$
Denote by $V_i$ a (small enough) neighborhood of $L_i\subset\PP^1_X\times\PP^1_Y$ where
$$L_1:\{Y=0\},\ \ \ L_2:\{X=\infty\},\ \ \ L_3:\{Y=\infty\}\ \ \ \text{and}\ \ \ L_4:\{X=0\}.$$
Denote $D=L_1\cup L_2\cup L_3\cup L_4$, and $V_i^*=V_i\setminus D$.
Let $V_{i,i+1}=V_i\cap V_{i+1}$ for $i\in\Z_4$ and $V_{i,i+1}^*=V_{i,i+1}\setminus D$.
Recall that
$$\Pi:S_0\setminus C \stackrel{\sim}{\longrightarrow}\C_X^*\times\C_Y^*\ ;\ (z,\xi)\mapsto(e^{2i\pi\xi},ze^{2i\pi\tau\xi}).$$
Then we have:

\begin{figure}[htbp]
\begin{center}
\includegraphics[scale=0.5]{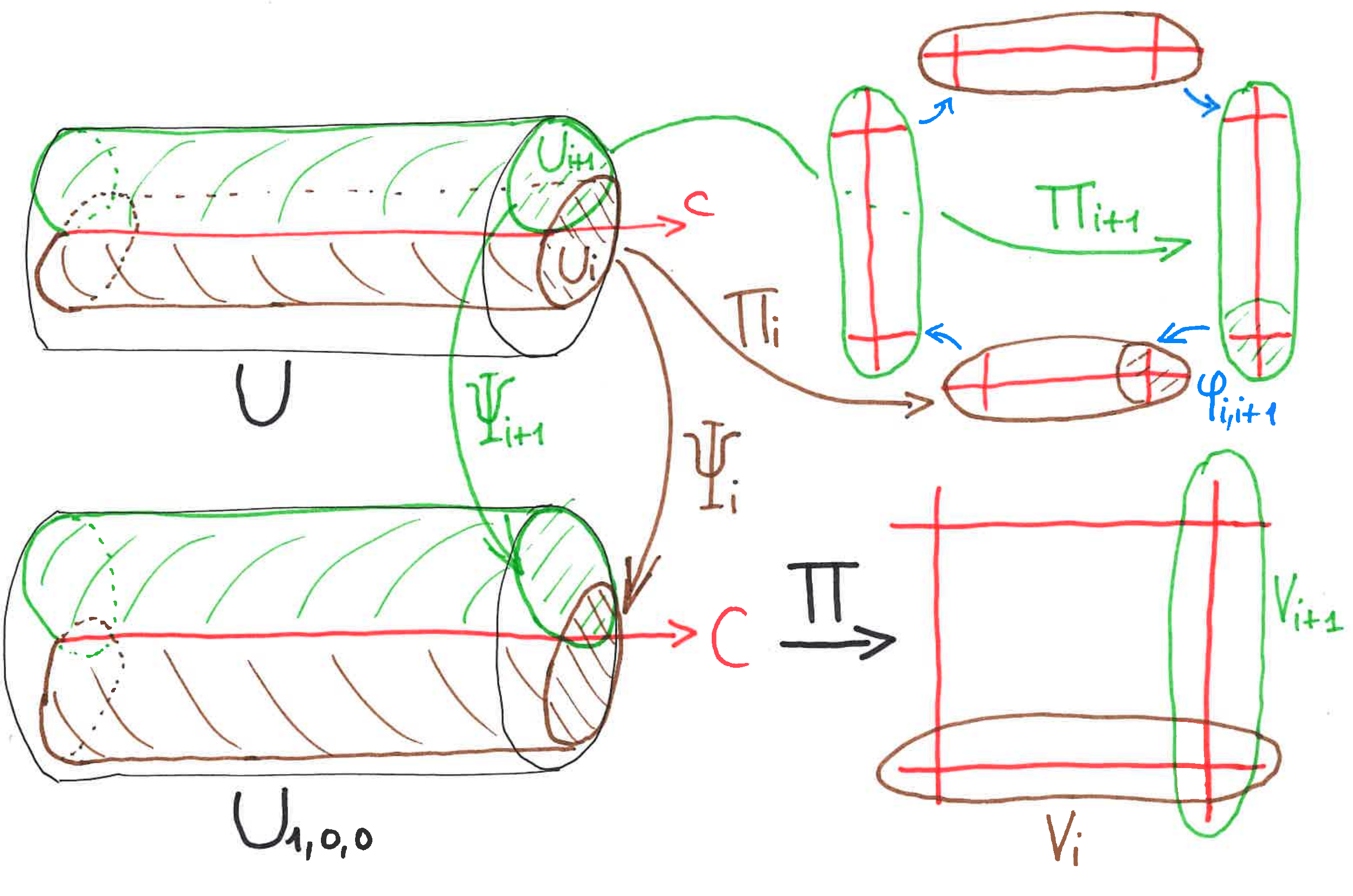}
\caption{{\bf Sectorial open sets}}
\label{pic:SecOpen}
\end{center}
\end{figure}

\begin{prop}\label{prop:sectorialdecomposition}
The preimage $U_i=\Pi^{-1}(V_i^*)$ lifts on $\tilde U=\C^*_z\times\C_\xi$ as a transversely sectorial open set of opening $I_i$
(in the sense of Definition \ref{def:sectorialopen}). Moreover, the lift of $\Pi^{-1}(V_{i,i+1}^*)=U_i\cap U_{i+1}$ is a transversely sectorial of opening 
$I_i\cap I_{i+1}$.
\end{prop}

\begin{proof} For instance, for $a,b,c>0$, we easily check that
$$U_{4,1}=\{(X,Y)\in\C^*\times\C^*\ ;\ \vert X\vert<\exp(-a),\ \vert Y\vert<\exp(-b)\}$$
contains the sectorial open set 
$$\left\{(z,\xi)\in\C^*\times\C\ ;\ e^{-c}<\vert z\vert<e^c,\ \im(\xi)>\frac{a}{2\pi},\ \im(\tau\xi)>\frac{b+c}{2\pi}\right\}.$$
The remaining cases are similar and straightforward.
\end{proof}

Denote $p_{i,i+1}=L_i\cap L_{i+1}$. 
Denote by $\OO^0(V_i,L_i)$ (resp. $\OO^0(V_{i,i+1},p_{i,i+1})$) the set of germs of holomorphic functions 
on $(V_i,L_i)$ (resp.  $(V_{i,i+1},p_{i,i+1})$) vanishing along $L_i$ (resp. at $p_{i,i+1}$). 
Denote by $\mc A^\infty[F_{1,0,0}]$ the subsheaf
of $\mc A^\infty$ whose sections $f$ are invariant by $F_{1,0,0}(z,\xi)=(qz,\xi-1)$.

\begin{prop}\label{prop:CharacSectorialFunctionsF0}
A section $f\in\OO(I_i)$ (resp. $\OO(I_{i,i+1})$) 
belongs to $\mc A^\infty[F_{1,0,0}](I_i)$ (resp. $\mc A^\infty[F_{1,0,0}](I_{i,i+1})$
if, and only if, $f=g\circ\Pi$ with $g\in\OO^0(V_i,L_i)$ (resp. $\OO^0(V_{i,i+1},p_{i,i+1})$).
\end{prop}

\begin{proof}As before, we only give the proof for $I_{4,1}$, the other cases are similar.
If $f=g\circ\Pi$ with $g\in\OO^0(V_{4,1},p_{4,1})$, then $g(X,Y)=X g_1(X,Y)+Y g_2(X,Y)$ 
with $g_k$ holomorphic at $p_{4,1}$ (and therefore bounded), 
so that $f(z,\xi)=e^{2i\pi\xi}f_1(z,\xi)+e^{2i\pi\tau\xi}f_2(z,\xi)$ with $f_k$ bounded: 
clearly, $f$ is (exponentially) flat at $\xi=\infty$ in restriction to any sector $S(J,R;c)\subset U_{4,1}$,
with $J\Subset I_{4,1}$.

Conversely, let $f\in\mc A^\infty[F_{1,0,0}](I_{4,1})$, defined on a sectorial open set 
$U_{4,1}$ of opening $I_{4,1}$. Let $U_{4,1}$ be the domain of definition of $f$,
a transversely sectorial open set of opening $I_{4,1}$ (see definition \ref{def:sectorialopen}).
One can find another one $U_{4,1}'\subset U_{4,1}$ such that
$$\forall(z_0,\xi_0)\in U_{4,1}',\ \ \ \forall(s,t)\in[0,1]\times[0,1],\ \ \ \mbox{then}\ (z,\xi)=(\xi_0+s,e^{2i\pi(\tau s-t)}z_0)\in U_{4,1}.$$
If we denote $(X_0,Y_0)=\Pi(z_0,\xi_0)$, then the image of $(z,\xi)$ while $(s,t)$ runs over the square is
$$(X,Y)=\Pi(z,\xi)=(e^{2i\pi s}X_0,e^{2i\pi t}Y_0)$$
a product of two loops. Therefore, the image $\Pi(U_{4,1})$ contains an open set $W'$ which is saturated 
by the toric action of $\bS^1\times\bS^1$ on $\C^*_X\times\C^*_Y$, i.e. a Reinhardt domain (see \cite[Chap.1,sec.2]{Shabat}), and which contains $U_{4,1}'$
(just take $W'$ to be the image of all those $(z,\xi)$ like above when $(z_0,\xi_0)$ runs over $U_{4,1}'$).
Since $f$ is invariant under $F_{1,0,0}$, i.e. $f\circ F_{1,0,0}=f$, then it factors through $\Pi$ and, maybe passing to another representative,
we have $f=g\circ\Pi$ where $g\in\mathcal O(W')$. Mind that $W'$ (as well as $\Pi(U_{4,1})$) might not 
be of the form $W\setminus(W\cap D)$ for a neighborhood $W$ of $p_{4,1}$, but we will prove that 
the holomorphic hull of $g$ is such a neighborhood.

As $W'$ is a Reinhardt domain, let us consider the (convergent) Laurent series of $g$:
$$g(X,Y)=\sum_{m,n\in\Z}a_{m,n}X^mY^n.$$
The coefficients are given by the integral 
$$a_{n,m}=\frac{1}{2i\pi}\int_{\beta_{\xi_0}}(\frac{1}{2i\pi}\int_{\alpha_{\xi_0}} g(X,Y)X^{-n-1}Y^{-m-1}dX)dY$$
where $\alpha_{\xi_0} (s)=(e^{2i\pi s}X_0, Y_0)$ and $\beta_{\xi_0} (t)= (X_0, e^{2i\pi t}Y_0)$. This can be rewritten as
$$ a_{n,m}=\int_{t=0}^1(\int_{s=0}^1 g(X,Y)X_0^{-m}Y_0^{-n}e^{-2i\pi(ms+nt)}ds)dt$$
from which we deduce the estimate
$$\vert a_{n,m}\vert\le \int_{t=0}^1(\int_{s=0}^1\vert g(X,Y)X_0^{-m}Y_0^{-n}\vert ds)dt$$
$$\vert a_{n,m}\vert\le \Vert g(X,Y)\Vert_{W'} \vert z_0\vert^{n} \vert e^{-2i\pi(m+\tau n)\xi_0}\vert $$
$$\vert a_{n,m}\vert\le \Vert f\Vert_{U_{4,1}}\vert z_0\vert^{n}e^{2\pi\Im\{(m+\tau n)\xi_0\}}.$$

Now, given $m,n\in\Z$, assume that there exists $\theta\in I_{4,1}$ such that  $\Im({e^{i\theta}(m+\tau n)})> 0$. 
The above inequality  promptly implies that $a_{n,m}=0$ by fixing $z_0$ and making $\xi_0\to\infty$
in the direction $\theta$ (which is possible in $U_{4,1}'$ as its opening is $I_{4,1}$).
This is possible if, and only if
$$\arg(m+\tau n)+I_{4,1}\ \ \ \text{intersects}\ \ \ ]-\pi,0[\ \mod\ 2\pi$$
which, since $I_{4,1}=]0,\pi-\varpi[$, means that
$$\arg(m+\tau n)\in\ ]-\pi,0[\ -\ ]0,\pi-\varpi[\ =\ ]-\pi,0[\ +\ ]\varpi-\pi,0[\ =\ ]\varpi-2\pi,0[.$$
It promptly follows that the only non zero coefficients $a_{m,n}$ occur when
$$\arg(m+\tau n)\in [0,\varpi]$$
which means that $m,n\ge0$, and $g$ extends holomorphically at $p_{4,1}:X=Y=0$.
Finally, since $f\to 0$ as $\xi\to\infty$, we get that $a_{0,0}=0$ and $g(0,0)=0$.
\end{proof}

\begin{remark}\label{rem:BoundedSectorFunctions}
The second part of the proof does not use the fact that $f$ is flat
(i.e. admits asymptotic expansion zero) along $\tilde C$, but only the fact that it is bounded.
As a consequence, any bounded holomorphic function on a transversely sectorial 
open set $U_i$ or $U_{i,i+1}$ as above automatically admits a constant as asymptotic expansion along $C$.
We note that bounded functions on $U_1$, $U_3$ (resp. $U_2$, $U_4$) therefore correspond to first integrals
of the foliation $\mathcal F_{1}$ (resp. $\mathcal F_0$).
\end{remark}

\subsection{Sheaves of sectorial automorphisms}\label{sec:sheavesautomorphisms}

Denote by $\mathrm{Aut}(S_0)$ the automorphism group of the ruled surface $S_0$ whose elements induce translations on $C$.
It preserves the ruling as well as the section $C\subset S_0$, inducing an action
on the neighborhood of $C$. The subgroup $\mathrm{Aut}^C(S_0)$ of elements fixing $C$ point-wise
is the one-parameter group generated by the flow of 
the vector field\footnote{See notations of section  \ref{sec:IntroFondamIso}.} 
$$\partial_\xi=2i\pi(X\partial_X+\tau Y\partial_Y).$$
We have an exact sequence
\begin{equation}\label{eq:exactsequenceautomorphism}
1\longrightarrow \mathrm{Aut}^C(S_0)\longrightarrow \mathrm{Aut}(S_0)\longrightarrow \mathrm{Aut}^0(C)\longrightarrow 1
\end{equation}
where $\mathrm{Aut}^0(C)$ is the translation group on $C$.
The  group $\mathrm{Aut}(S_0)$ is connected and generated by the flows of 
\begin{equation}\label{eq:generator2Dgroup}
\partial_\xi+2i\pi\tau z\partial_z=2i\pi X\partial_X\ \ \ \text{and}\ \ \ -2i\pi z\partial_z=2i\pi Y\partial_Y
\end{equation}
It is then easy to check that the full group of automorphisms  $\widetilde{\mathrm{Aut}(S_0)}$  of $S_0$ is generated by $\mathrm{Aut}(S_0)$ and a finite order map which,
for a general curve $C$, is just an involution that can be chosen  to be $(z,\xi)\to (\frac{1}{z}, -\xi)$. 
In fact, specializing $\widetilde{\mathrm{Aut}(S_0)}$ to the neighborhood of the curve, we get all analytic, and even 
formal automorphisms of the neighborhood $(S_0,C)$:

\begin{lemma}\label{L:formalcentralizer}
Any formal automorphism $\hat\Phi:\acts(S_0,C)$ fixing $C$ point-wise is actually
convergent and belongs to $\mathrm{Aut}^C(S_0)$.
\end{lemma}

\begin{proof}Recall \cite{LTT} that the only formal regular foliations on $(S_0,C)$ are those defined
by $\omega=0$ where $\omega$ belongs to the vector space of closed $1$-forms 
$E=\C \frac{dz}{z}+\C d\xi$. Moreover, for $\omega\in E\setminus\C \frac{dz}{z}$, $\mathcal F_\omega$ does not 
admit non constant formal meromorphic first integral, and the only formal closed meromorphic $1$-forms
defining $\mathcal F_\omega$ must be a constant multiple of $\omega$, thus belonging to $E$. If
$$\hat\Phi(z,\xi)=\left(z+\sum_{n>0}\frac{a_n(z)}{\xi^n},\sum_{n\ge0}\frac{b_n(z)}{\xi^n}\right)$$
is a formal automorphism of $(S_0,C)$ fixing $C$ point-wise, then it must preserve the vector space $E$. In particular,
it must preserve $\C \frac{dz}{z}$ (and $z$ actually as it fixes $C$ point-wise) and sends $d\xi$ to some other element $\alpha \frac{dz}{z}+\beta d\xi$. 
A straightforward computation shows that $\hat\Phi$ writes
$$\hat\Phi(x,\xi)=\left(z,\alpha\log(z)+\beta \xi +\gamma\right),\ \ \ \gamma\in\C,$$
and we have $\alpha=0$.
Finally, as $\hat\Phi$ must commute with $F_{1,0,0}(z,\xi)=(qz,\xi-1)$, we get $\beta=1$.
\end{proof}

\begin{cor}\label{cor:formalcentralizer}\footnote{We will generalize this result in subsection \ref{L:formalcentralizergen} 
using the notion of periods as defined in \cite[Section 2.4]{LTT},  and their invariance under automorphisms.} 
Any formal automorphism $\hat\Phi:\acts(S_0,C)$ is actually convergent and belongs to $\widetilde{\mathrm{Aut}(S_0)}$.
\end{cor}

\begin{proof} The formal diffeomorphism $\hat\Phi$ induces an automorphism of $C$. 
Using exact sequence (\ref{eq:exactsequenceautomorphism}), after composing $\hat\Phi$
by a convenient element of $\widetilde{\mathrm{Aut}(S_0)}$, we can assume that it fixes $C$ point-wise,
and then apply Lemma \ref{L:formalcentralizer}.
\end{proof}

\begin{dfn}\label{D:sectorialbiholo}
Let us consider the germs of sectorial biholomorphisms in the direction $\arg(\xi)=\theta$ of $(\tilde U,\tilde C)$
that are tangent to the identity:
$$\Phi(z,\xi)=(z+\frac {f_1(z,\xi)}{\xi},\xi+\frac {f_2(z,\xi)}\xi),\ \ \ f_1,f_2\in{\mathcal A}_{\theta}.$$
The collection of these germs when varying $\theta$ 
naturally gives rise to a sheaf of groups (with respect to the composition law) on $\bS^1$ that will be denoted 
by $\mathscr G^1$. We will consider for further use the subsheaf $\mathscr G^\infty$ of $\mathscr G^1$ of 
germs of sectorial biholomorphisms flat to identity, i.e. when $f_1,f_2\in {\mathcal A}^\infty_{\theta}$.
Denote by $\mathscr G^1 [F_{1,0,0}]$ (resp. $\mathscr G^\infty [F_{1,0,0}]$) the subsheaf of  $\mathscr G^1$ (resp. $\mathscr G^\infty$) defined 
by germs of transformations $\Phi$ commuting with $F_{1,0,0}$: $\Phi\circ F_{1,0,0}=F_{1,0,0} \circ \Phi$.
\end{dfn}

\begin{remark}\label{rem:sectorialcentralizer}
Note that $\Phi\in \mathscr G^1 [F_{1,0,0}]$ implies that its asymptotic expansion $\hat{\Phi}$ also commutes with $F_{1,0,0}$,
i.e. $\hat{\Phi}\circ F_{1,0,0}=F_{1,0,0}\circ\hat{\Phi}$. According to the description of the formal centralizer of $F_{1,0,0}$ in Lemma \ref{L:formalcentralizer}, it turns out that $\mathscr G^1[F_{1,0,0}]=\mathscr G^\infty[F_{1,0,0}]\rtimes \mathrm{Aut}^C(S_0)$ where $\mathrm{Aut}^C(S_0)$ is regarded as a constant sheaf on $\bS^1$.
\end{remark}

We would like to apply characterization of $\mc A^\infty[F_{1,0,0}](I)$ obtained in the previous section
for our special sectors $I_i$ and $I_{i,i+1}$ to obtain a similar characterization of sections of $\mathscr G^\infty [F_{1,0,0}]$.
For this, denote by $\mathrm{Diff}(V_i,L_i)$ the group of germs of biholomorphisms of $(V_i,L_i)$
which preserves the  divisor $(D\cap V_i)$, for instance:
\begin{equation}\label{F:diffV}
\mathrm{Diff}(V_1,L_1)=\{ \varphi(X,Y)=(Xa(Y),Yb(Y))\ ;\ a,b\in\C\{Y\},\ a(0),b(0)\not=0\}
\end{equation}
and by $\mathrm{Diff}^1(V_i,L_i)$ the subgroup of germs tangent to the identity along $L_i$,
i.e. $a(0)=b(0)=1$ in example (\ref{F:diffV}). In a similar way, denote by $\mathrm{Diff}(V_{i,i+1},p_{i,i+1})$ 
the group of germs of biholomorphisms of $(V_{i,i+1},p_{i,i+1})$ which preserve the germ of divisor $(D\cap V_{i,i+1},p_{i,i+1})$
and by $\mathrm{Diff}^1(V_{i,i+1},p_{i,i+1})$ the subgroup of germs tangent to the identity at $p_{i,i+1}$.
For instance:
$$\mathrm{Diff}(V_{4,1},p_{4,1})=\{ \varphi(X,Y)=(Xa(X,Y),Yb(X,Y))\ ;\ a,b\in\C\{X,Y\},\ a(0),b(0)\not=0\},$$
and $\mathrm{Diff}^1(V_{4,1},p_{4,1})$ is characterized by $a(0)=b(0)=1$.

\begin{prop}\label{P:sectorstoP1XP1} We have the following characterizations:
\begin{itemize}
\item $\Phi\in {\mathscr G}^\infty [F_{1,0,0}](I_i)$ if and only if $\Pi\circ \Phi=\varphi\circ\Pi$ where $\varphi\in\mathrm{Diff}^1(V_i^,L_i)$;
\item $\Phi\in{\mathscr G}^\infty [F_{1,0,0}](I_{i,i+1})$ if and only if $\Pi\circ \Phi=\varphi\circ\Pi$ where $\varphi\in\mathrm{Diff}^1(V_{i,i+1},p_{i,i+1})$.
\end{itemize}
\end{prop}

\begin{proof}
For any interval $I$, a section $\Phi$ of $\mathscr G^\infty(I)$ can be written $\Phi(z,\xi)=(z(1+f_1),\xi +f_2)$
with $f_1,f_2\in{\mathcal A}^\infty(I)$. Then $\Phi$ belongs to $\mathscr G^\infty [F_{1,0,0}](I)$
if, and only if, $f_1,f_2$ are invariant by $F_{1,0,0}$, i.e. $f_1,f_2\in{\mathcal A}^\infty[F_{1,0,0}](I)$.
Assume now $I=I_{4,1}$, say. 
Then, by Proposition \ref{prop:CharacSectorialFunctionsF0}, one can write $f_k=g_k\circ\Pi$, i.e. $f_k(z,\xi)=g_k(X,Y)$, 
with $g_k\in\OO^0(V_{4,1},p_{4,1})$. Therefore, one can write
$$\Pi\circ\Phi=(Xa(X,Y),Yb(X,Y))\ \ \ \text{with}\ 
\left\{\begin{matrix}a=e^{2i\pi g_2(X,Y)},\hfill\\ 
b=e^{2i\pi\tau g_2(X,Y)}(1+g_1(X,Y))^{-1}.
\end{matrix}\right.$$
Clearly, $a,b$ are holomorphic at $(X,Y)=(0,0)$ and $a(0,0)=b(0,0)=1$. Conversely,
given $\varphi\in\mathrm{Diff}^1(V_{4,1},p_{4,1})$, thus of the form $\varphi(X,Y)=(Xa(X,Y),Yb(X,Y))$,
we recover $f_1,f_2\in{\mathcal A}^\infty[F_{1,0,0}](I_{4,1})$, and $\Phi(z,\xi)=(z(1+f_1),\xi +f_2)$,  by setting
$$f_1=\left(\frac{a^{\tau}}{b}-1\right)\circ\Pi\ \ \ \text{and}\ \ \ f_2=\frac{\log(a)}{2i\pi}\circ\Pi.$$
The description of elements of ${\mathscr G}^\infty [F_{1,0,0}](I_{i})$, ${\mathscr G}^\infty [F_{1,0,0}](I_{i,i+1})$ can be carried out exactly along the same line. 
\end{proof}

\section{Analytic classification: an overview}\label{sec:AnalyticClass}

Here, we would like to detail our main result, namely the analytic classification
of all neighborhoods that are formally equivalent to $(U_{1,0,0},C)$. The most 
technical ingredient is the sectorial normalization (Lemma \ref{LEM:sectorialnormalization} in the introduction)
which now reads as follows. Let $F$ be a biholomorphism like in Proposition \ref{prop:PrenormalFormU100}
$$F(z,\xi)=\left(qz+\sum_{n\ge2}\frac{\alpha_n(z)}{\xi^n},\xi-1+\sum_{n\ge2}\frac{\beta_n(z)}{\xi^n}\right).$$
In particular, there is a formal diffeomorphism $\hat\Psi$ (that can be assumed to be tangent to the identity along $C$)  
conjugating $F$ to $F_{1,0,0}(z,\xi)=(qz,\xi-1)$,
i.e. $F\circ \hat\Psi=\hat\Psi\circ F_{1,0,0}$.

\begin{lemma}\label{lem:sectorialRappel}
Denote $\varpi=\arg\tau$. For each interval
\begin{equation}\label{eq:sectorxi}
I_1=]-\varpi,\pi-\varpi[,\ \ \ I_2=]-\pi,0[,\ \ \ I_3=I_1+\pi\ \ \ \text{and}\ \ \ I_4=I_2+\pi,
\end{equation}
there is a section $\Psi_i$ of ${\mathscr G}^1(I_i)$ (see Definition \ref{D:sectorialbiholo})
such that
$$\Psi_i\circ F=F_{1,0,0}\circ \Psi_i.$$
\end{lemma}

Section \ref{S:sectorialnorm1} is devoted to the proof of this lemma. Let us see how to use it 
in order to provide a complete set of invariants for the neighborhood $(U,C)=(\tilde U,\tilde C)/<F>$.
First of all, we note that $\Psi_i$ is unique up to left-composition by a section of ${\mathscr G}^1[F_{1,0,0}](I_i)$,
i.e. the composition of an element of the one-parameter group $\mathrm{Aut}^C(S_0)$ with a section
of $\mathscr G^\infty[F_{1,0,0}](I_i)$ (see Remark \ref{rem:sectorialcentralizer}). Using this freedom, 
we may assume that asymptotic expansions coincide:
$$\hat\Psi_i=\hat\Psi_j=\hat\Psi.$$
It follows that, on intersections $I_{i,i+1}=I_i\cap I_{i+1}$, we get sections
$$\Phi_{i,i+1}:=\Psi_i\circ\Psi_{i+1}^{-1}\in\mathscr G^\infty[F_{1,0,0}](I_{i,i+1}).$$
Using Proposition \ref{P:sectorstoP1XP1}, we have
$$\Pi\circ \Phi_{i,i+1}=\varphi_{i,i+1}\circ\Pi\ \ \ \text{for some}\ \ \ \varphi_{i,i+1}\in\mathrm{Diff}^1(V_{i,i+1},p_{i,i+1}).$$
In other words, setting $\Pi_i:=\Pi\circ\Psi_i$, we get
\begin{equation}\label{eq:PiVarphiCocycleIdentity}
\Pi_i=\Pi\circ\Psi_i=\Pi\circ\Phi_{i,i+1}\circ\Psi_{i+1}=\varphi_{i,i+1}\circ\Pi\circ\Psi_{i+1}=\varphi_{i,i+1}\circ\Pi_{i+1}
\end{equation}
which proves Corollary \ref{COR:cocycle}. 
We have therefore associated to each neighborhood $(U,C)$ formally equivalent to $(U_{1,0,0},C)$
a cocycle $\varphi=(\varphi_{i,i+1})_{i\in\Z_4}$ which is unique up to the freedom for the choice of $\Psi_i$'s.

\begin{dfn}\label{def:EquivalentCocycles}
We say that two cocycles $\varphi$ and $\varphi'$ are equivalent if 
$$ \exists t\in\C,\ \exists\varphi_{i}\in\mathrm{Diff}^1(V_i,L_i)$$
\begin{equation}\label{eq:anequivcocycle}
\text{such that}\ \ \ \varphi_{i,i+1}'=\phi^t\circ\varphi_i\circ\varphi_{i,i+1}\circ\varphi_{i+1}^{-1}\circ\phi^{-t}
\end{equation}
where $\phi^t=(e^{2i\pi t}X,e^{2i\pi\tau t}Y)$ is the one-parameter group of the vector field 
$v_\tau=2i\pi(X\partial_X+\tau Y\partial_Y)$.

We will denote this equivalence relation by $\approx$.
\end{dfn}

\begin{prop}\label{prop:AnalyticClassif1} (\textbf{Proof of Theorem \ref{THM:MainClassif}})
Two neighborhood $(U,C)$ and $(U',C)$ formally equivalent to $(U_{1,0,0},C)$ are
analytically equivalent if, and only if, the corresponding cocycles are equivalent
$$(U,C)\stackrel{\text{an}}{\sim}(U',C)\ \ \ \Leftrightarrow\ \ \ \varphi\approx\varphi'.$$
\end{prop}

\begin{proof}  According to the description of $\mathrm{Aut}^C(S_0)$, a biholomorphism germ $(U,C)\to(U',C)$ is indeed tangent to identity along $C$ lifts-up to a global section 
$\Psi\in\mathscr G^1(\bS^1)$ satisfying $\Psi\circ F=F'\circ\Psi$. Let $(\Psi_i)$ and 
$(\Psi_i')$ be the sectorial normalizations used to compute the invariants $\varphi$
and $\varphi'$. Clearly, $\Psi_i'\circ\Psi$ provides a new collection of sectorial trivializations
for $(U,C)$. We can write (using Remark \ref{rem:sectorialcentralizer})
$$\Psi_i'\circ\Psi=\exp(t_i\partial_\xi)\circ\Phi_i\circ\Psi_i\ \ \ \text{with}\ \ \ \Phi_i\in\mathscr G^\infty[F_{1,0,0}](I_i).$$
However, as $\hat\Psi_i=\hat\Psi_j$ and $\hat\Psi_i'=\hat\Psi_j'$, we have $t_i=t_j=:t$ for all $i,j$. Therefore,
we have
$$\Phi_{i,i+1}'=(\Psi_i'\circ\Psi)\circ(\Psi_{i+1}'\circ\Psi)^{-1}$$
$$=(\exp(t\partial_\xi)\circ\Phi_i\circ\Psi_i)\circ(\exp(t\partial_\xi)\circ\Phi_{i+1}\circ\Psi_{i+1})^{-1}$$
$$=\exp(t\partial_\xi)\circ\Phi_i\circ\Phi_{i,i+1}\circ\Phi_{i+1}^{-1}\circ\exp(-t\partial_\xi).$$
After factorization through $\Pi$, using (\ref{eq:generator2Dgroup}) and Proposition \ref{P:sectorstoP1XP1},
we get the expected equivalence relation (\ref{eq:anequivcocycle}) for $\varphi$ and $\varphi'$.
Conversely, if $\varphi\an \varphi'$, then we can trace back the existence of an analytic conjugacy 
$\Phi:(U,C)\to(U',C)$ by reversing the above implications.
\end{proof}

\begin{remark}\label{rem:weakequivalence}
We can weaken the notion of analytic equivalence between neighborhoods by considering 
biholomorphism germs $\Phi:(U,C)\to(U',C)$ inducing translations on $C$.
This means that, in Definition \ref{def:an/forequivalence}, we now allow conjugacies 
$\Phi(z,y)=(cz+O(y),O(y))$ with $c\in\C^*$ in formula (\ref{eq:PrenormNeigh100}), i.e. translations on the elliptic curve.
In that case, the corresponding cocycles are related by 
$$\varphi_{i,i+1}'=\phi\circ\varphi_i\circ\varphi_{i,i+1}\circ\varphi_{i+1}^{-1}\circ\phi^{-1}$$
where $\phi(X,Y)=(aX,bY)$ for arbitrary $a,b\in\C^*$ \footnote{Note that those are precisely the transformation arising from the natural torus action on $\mathbb P^1\times \mathbb P^1$ (see also Section \ref{sec:symmetries}). For the sake of clarity, we will state our general result (Section \ref{S:Generalization}) modulo analytic isomorphisms inducing tranlations (and not only the identity) on $C$. } . 
Having in mind the description given in (\ref{F:diffV}), one observes that two cocycles are equivalent iff they lie on the same orbit over some action (that the reader will easily explicit) of the fiber product ${({\mathcal O}^* \times{\mathcal O}^*)}^{\times_{\C^* \times\C^* }^4}$ of $4$ copies of ${\mathcal O}^* \times{\mathcal O}^*$ with respect to the natural morphism ${\mathcal O}^* \times{\mathcal O}^* \ni(f,g)\to (f(0),g(0))\in \C^* \times\C^*$.
\end{remark}

To summarize, we have just associated to each $(U,C)\formal (U_{1,0,0},C)$ a cocycle
\begin{equation}\label{eq:Cocycle}
\varphi=(\varphi_{i,i+1})_{i\in\Z_4},\ \ \ \varphi_{i,i+1}\in\mathrm{Diff}^1(V_{i,i+1}^*,p_{i,i+1})
\end{equation}
and constructed a map from the moduli space $\mc U_{1,0,0}$ of such neighborhood up to analytic equivalence ${\an}$
to the moduli space $\mc C$ of cocycles $\varphi$ like (\ref{eq:Cocycle}) up to equivalence (\ref{eq:anequivcocycle}):
\begin{equation}\label{eq:ModularMap}
\mu\ :\ \mc U_{1,0,0}=\{(U,C)\formal (U_{1,0,0},C)\}/_{\an}\longrightarrow\mc C=\{\varphi\}/_{\approx}
\end{equation}
which is proved to be injective in Proposition \ref{prop:AnalyticClassif1}. In Section \ref{SS:Construction},
we prove the surjectivity by constructing an inverse map $\varphi\mapsto U_\varphi$.
Before that, we want to reinterpret the cocycle $\varphi$ as transition maps of an atlas 
for a neighborhood $(V_\varphi,D)$.

\section{Construction of $U_\varphi$.}\label{SS:Construction}

In this section, we construct a large class of non analytically equivalent neighborhoods,
all of them formally equivalent to $(U_{1,0,0},C)$. This is done by sectorial surgery, 
extending the complex structure along $C$ by means of Newlander-Nirenberg Theorem.
In order to do this, we have to work with smooth functions (i.e. of class $C^\infty$).

\begin{dfn}\label{def:flatsmoothsectorialfunctions}
For any open sector $\Sigma_I$ (Definition \ref{def:sector}), we denote by $\mc E^\infty(\Sigma_I)$ the $\C$-algebra
of those complex smooth functions $f:\Sigma_I\to\C$ satisfying
the following estimates
$$\forall \alpha=(\alpha_1,\alpha_2, \alpha_3, \alpha_4)\in{\mathbb N}^4,\ 
\forall n\in\N,\ \forall K\subset {\C}^*\ \text{compact},\ \exists C>0\ \text{such that:}$$
$$\forall (z,\xi)\in \Sigma_I,\ z\in K,\ \text{we have}\ 
\left|\dfrac{\partial^{\alpha_1 +\alpha_2+\alpha_3+\alpha_4}f(z,\xi)}{\partial z^{\alpha_1}\partial\xi^{\alpha_2}\partial\bar{ z}^{\alpha_3}\partial\bar{\xi}^{\alpha_2}}\right|\leq \frac{C}{|\xi^{n+1}|}.$$
\end{dfn}

Passing to inductive limits and sheafification as in Section \ref{SS:sheaves}, 
we get a sheaf $\mc E^\infty$ of differential algebra on the circle $\bS^1$.
Like in Section \ref{sec:sheavesautomorphisms}, we can also define 
the sheaf of groups $\mc D^{\infty}$ on the circle, whose sections
$\Psi\in\mc D^{\infty}(I)$ are smooth sectorial diffeomorphisms asymptotic
to the identity, i.e. of the form $\Psi(z,\xi)=(z+h_1,\xi+h_2)$ with $h_1,h_2\in\mc E^\infty(I)$.
The following property somehow expresses that a cocycle defined by a collection 
of sectorial biholomorphisms is a coboundary in the $C^\infty$ category.

\begin{lemma}\label{L:sectordiff}
Let $(J_i)_{i\in I}$ be a covering of $\bS^1$ by open intervals threewise disjoints. 
Assume also that there exist on non empty intersections $J_{ij}:=J_i\cap J_j$ 
a family of sectorial biholomorphisms $\Phi_{ij}\in {\mathscr G}^\infty (J_{ij})$ 
with $\Phi_{ji}={\Phi_{ij}}^{-1}$ (in particular $\Phi_{ii}=\mbox{id}$). 
Then, there exist smooth sectorial diffeomorphisms flat to identy $\psi_i\in {\mathcal D}^\infty (J_i)$ 
such that $\Phi_{ij}=\psi_i\circ {\psi_j}^{-1}$. 
\end{lemma}

\begin{proof}
One can extract from this covering a finite covering $(J_k), k\in {\mathbb Z}_n$ such that only consecutive
sectors $J_k$ and $J_{k+1}$ intersect. It clearly suffices to prove the Lemma for this particular subcovering.   
Let $(\theta_k)$ a partition of the unity subordinate to this covering. 
Write 
$$\Phi_{k,k+1}(z,\xi)=(z+h_{k,k+1}^1,\xi+ h_{k,k+1}^2)\ \ \ \text{with}\ \ \ h_{k,k+1}^1,h_{k,k+1}^2\in {\mathcal A}^\infty (J_{k,k+1}).$$ 
First define $\tilde\psi_k\in {\mathcal D}^\infty (J_k)$ for $k\in {\mathbb Z}_n$ by
$$\tilde\psi_k=\left\{\begin{matrix}
\mbox{Id}&\text{when}&\arg(\xi)\in J_k\setminus J_{k,k+1},\\
\mbox{Id}+\theta_{k+1}(\arg\xi)(h_{k,k+1}^1,h_{k,k+1}^2)&\text{when}&\arg(\xi)\in J_{k,k+1}
\end{matrix}\right.$$
Next, define $\psi_k\in {\mathcal D}^\infty (J_k)$ by 
$$\psi_k=\left\{\begin{matrix}
\Phi_{k,k-1}\circ\tilde\psi_{k-1}&\text{when}&\arg(\xi)\in J_{k-1,k}\hfill\\
\tilde\psi_k&\text{when}&\arg(\xi)\in J_k\setminus J_{k-1,k}\hfill
\end{matrix}\right.$$
One easily check that $\psi_k$ are smooth, equal to the identity outside intersections, 
and satisfy $\psi_{k}=\Phi_{k,k-1}\circ\psi_{k-1}$ on intersections as expected.
\end{proof}

\begin{cor}\label{cor:surgeries}
Notations and assumptions like in Lemma \ref{L:sectordiff}.
There exist sectorial biholomorphisms tangent to identity $\Psi_i\in{\mathscr G}^1 (J_i)$ 
such that $\Phi_{ij}={\Psi_i}\circ {\Psi_j}^{-1}$. 
In particular, asymptotic expansions coincide $\hat\Psi_i=\hat\Psi_j$.
\end{cor}

\begin{proof} Let $\tilde U_i$ be the sectorial domain of definition of $\psi_i$ and 
$\tilde U$ be their union together with the section $\tilde C$ defined by $\xi=\infty$.
Lemma \ref{L:sectordiff} allows to write  $\Phi_{ij}={\psi_i}\circ{\psi_{j}}^{-1}$ where $\psi_i\in{\mathcal D}^\infty (I_i)$. 
In particular, denoting by $I$ the standart complex structure on ${\C}^2$, $J:={\psi_i}^*I={\psi_j}^*I$ 
is a new complex structure on $\tilde{U}\setminus\tilde{C}$ which extends to $\tilde{U}$ 
as a complex structure by Newlander-Nirenberg's Theorem. 
In fact, because of flatness of $\psi_i$ to the identity, the almost complex structure $J$ extends at $0$ as a $C^\infty$ 
almost complex structure on $\tilde U$; by construction, it is integrable on $\tilde U_i$'s and therefore Nijenhuis tensor vanishes
identically on $\tilde U\setminus\tilde C$, and by continuity on $\tilde U$. Then, Newlander-Nirenberg's Theorem tells us that $J$ is integrable.
Note that $I=J$ in restriction to $\tilde{C}$ which is then conformally equivalent to ${\C}^*$ for both structures. 
Now, we use the fact that two-dimensional germs of neighborhood of ${\C}^*$ are analytically equivalent as recalled 
in Section \ref{S:preliminary}. 
This can be translated into the existence of a smooth diffeomorphism $\psi$ of $(\tilde{U},\tilde{C})$ such that $\psi_* I=J$. Up to making a right composition by a biholomorphism of $(\tilde{U}, \tilde{C})$ with respect to $I$, one can suppose (exploiting that $I=J$ on $T\tilde{U}_{|\tilde{C}}$) that $\psi$ is tangent to the identity along $\tilde{C}$. This implies that for every $i$, $\Psi_i:=\psi_i\circ\psi\in\mathscr G^1(U_i)$ and, because the $\Phi_{ij}$'s are flat to identity, admit an asymptotic expansion $\hat{\Psi}_i$ along $\tilde C={\C}^*$ independant of $i$. By construction, we have $\Phi_{ij}=\Psi_i\circ {\Psi_j}^{-1}$ as desired.
Obviously, all along this proof, we might have shrinked the domain $\tilde U$ of definition without mentionning it. 
\end {proof}

\begin{remark}
The use of the Newlander-Nirenberg in this context is not new and can be traced back to Malgrange \cite{Malgrange} and Martinet-Ramis \cite{MartinetRamis}.
\end{remark}

We now specialize to our covering of $\bS^1$ determined by the intervals $I_i$ defined by
(\ref{eq:sectorxi}) in Lemma \ref{lem:sectorialRappel}.
Let us show how to construct a neighborhood realizing a given cocycle $\varphi=(\varphi_{i,i+1})$
as in (\ref{eq:Cocycle}). We first define $\Phi_{i,i+1}\in{\mathscr G}^\infty [F_{1,0,0}](I_{i,i+1})$ satisfying 
$\Pi\circ \Phi_{i,i+1}=\varphi_{i,i+1}\circ\Pi$. 
Then use Corollary \ref{cor:surgeries} to obtain $\Psi_i\in{\mathscr G}^1 (I_i)$ 
such that $\Phi_{i,{i+1}}={\Psi_i}\circ \Psi_{i+1}^{-1}$. As $\Phi_{i,i+1}$ commute to $F_{1,0,0}$,
we have on intersections:
$$({\Psi_i}\circ \Psi_{i+1}^{-1})\circ F_{1,0,0}=F_{1,0,0}\circ ({\Psi_i}\circ \Psi_{i+1}^{-1})$$
which rewrites
$$\Psi_{i+1}^{-1}\circ F_{1,0,0}\circ\Psi_{i+1}=\Psi_i^{-1}\circ F_{1,0,0}\circ\Psi_i.$$
Therefore, we can define a global diffeomorphism of $(\tilde U,\tilde C)$ by setting
$$F_\varphi:=\Psi_i^{-1}\circ F_{1,0,0}\circ\Psi_i$$
on $U_i$'s and extending by continuity as the identity mapping on $\tilde C$.
By construction, the quotient 
$$(U_\varphi,C):=(\tilde U,\tilde C)/<F_\varphi>$$
has cocycle $\varphi$ and is formally equivalent to $U_{1,0,0}$. This proves the surjectivity of the map (\ref{eq:ModularMap})
whose injectivity has been proved in Proposition \ref{prop:AnalyticClassif1}.
It remains to prove the Sectorial Normalization Lemma \ref{lem:sectorialRappel}
(i.e. Lemma \ref{LEM:sectorialnormalization} in the introduction),
which will be done in Section \ref{S:sectorialnorm1}.
Modulo this technical but central Lemma, we have achieved the proof of Theorem \ref{THM:MainClassif}.

\section{Construction of $V_\varphi$.}\label{SS:Vphi}
In this section, keeping notations of Section \ref{sec:IntroFondamIso}, we generalize Serre isomorphism\newline
 $\Pi:U_{1,0,0}\setminus C\to\C^*_X\times\C^*_Y$ to the case of a general neighborhood $(U,C)\formal(U_{1,0,0},C)$.

\begin{thm}\label{thm:SerreIsomGerm}
Given a germ of neighborhood $(U_\varphi,C)\formal(U_{1,0,0},C)$, there exists a neighborhood germ
$(V_\varphi,D)$ of $D$ where each $L_i$ have trivial normal bundle, and an isomorphism germ
\begin{equation}\label{eq:NonLinSerreIsom}
\Pi_\varphi\ :\ (U_\varphi\setminus C,C)\stackrel{\sim}{\longrightarrow} (V_\varphi\setminus D,D)
\end{equation}
canonically attached to the analytic class of $(U_\varphi,C)$ in the following sense: 
if $(U_{\varphi'},C)$ is another neighborhood germ, then
\begin{equation}\label{eq:EquivU=EquivV}
(U_\varphi,C)\an(U_{\varphi'},C)\ \ \ \Leftrightarrow\ \ \ (V_\varphi,D)\an(V_{\varphi'},D)
\end{equation}
where the analytic equivalence allows translations\footnote{We emphasize that this is not exactly 
the equivalence relation defined in Definition \ref{def:an/forequivalence}.} 
on $C$ for the left-hand-side, 
and preserves the numbering of lines $L_i$ on the right-hand-side.
\end{thm}

\begin{proof}
Given a cocycle not necessarily tangent to the identity
$$\varphi=(\varphi_{i,i+1})_{i\in\Z_4},\ \ \ \varphi_{i,i+1}\in\mathrm{Diff}(V_{i,i+1},p_{i,i+1}),$$
we define a new germ of analytic neighborhood of $D$ as follows. We consider the disjoint union 
of neighborhood germs $(V_i,L_i)$, 
and patch them together through the transition maps 
$$\varphi_{i,i+1}:(V_{i+1},p_{i,i+1})\stackrel{\sim}{\longrightarrow}(V_i,p_{i,i+1}).$$
The resulting analytic manifold $V_\varphi$ contains a copy of $D$, namely the union of lines $L_i$
identified at points $p_{i,i+1}$, and only the germ of neighborhood $V_\varphi$ makes sense
$$(V_\varphi,D):=\sqcup_i (V_i,L_i)/(\varphi_{i,i+1}).$$
This germ of neighborhood comes with embeddings 
$$\psi_i:(V_i,L_i)\hookrightarrow(V_\varphi,D).$$
Conversely, if $(V,D)$ is a germ of neighborhood of $D$ where all lines $L_i$ have zero self-intersection,
then there exist trivialization maps (preserving D) 
$$\psi_i:(V_i,L_i)\stackrel{\sim}{\longrightarrow}(V,L_i)$$
(where $ (V,L_i)$ denotes  the germ of $V$ along $D$) 
in such a way that, near $p_{i,i+1}$ we have
$$\psi_i=\varphi_{i,i+1}\circ\psi_j\ \ \ \text{for some}\ \ \ \varphi_{i,i+1}\in\mathrm{Diff}(V_{i,i+1},p_{i,i+1}).$$
It is clear from above arguments that, for another cocycle $\varphi'$, we have
$$(V_\varphi,D)\an (V_{\varphi'},D)\ \ \ \Leftrightarrow\ \ \ \varphi\sim{\varphi'}$$
where
$$\varphi\sim{\varphi'}\ \ \ \stackrel{\text{def}}{\Leftrightarrow}\ \ \ \exists\varphi_i\in\mathrm{Diff}(V_{i},L_{i}),\ \varphi_{i}\circ\varphi_{i,i+1}'=\varphi_{i,i+1}\circ\varphi_{i+1},$$
and in that case, the isomorphism $V_\varphi\stackrel{\sim}{\longrightarrow}V_{\varphi'}$ is given by patching 
$$(V_\varphi,D)\stackrel{\psi_i}{\longleftarrow}(V_i,L_i)\stackrel{\varphi_i}{\longrightarrow}(V_i,L_i)\stackrel{\psi_i'}{\longrightarrow}(V_{\varphi'},D).$$
From the linear part of equivalence relation $\varphi\sim{\varphi'}$, we see that any cocycle $\varphi$
is equivalent to a cocycle such that 
\begin{itemize}
\item $\varphi_{1,2},\varphi_{2,3},\varphi_{3,4}\in\mathrm{Diff}^1(V_{i,i+1},p_{i,i+1})$ (tangent to the identity),
\item $\varphi_{4,1}(X,Y)=(aX+\cdots,bY+\cdots)$ for $a,b\in\C^*$ independant of the choice.
\end{itemize}
The pair $(a,b)$ is an invariant of the neighborhood $V_\varphi$.
Cocycles arising from $(U,C)\formal(U_{1,0,0},C)$ have invariants $a=b=1$.
In order to prove the equivalence (\ref{eq:EquivU=EquivV}), we just have to note
that, for equivalent cocycles $\varphi\sim{\varphi}'$ normalized as above 
(in particular when all $\varphi_{i,i+1},\varphi_{i,i+1}'$ are tangent to the identity)
then all four conjugating maps $\varphi_i$ have the same linear part. Then apply 
Remark \ref{rem:weakequivalence} to show that it corresponds to analytic
equivalence of $(U_\varphi,C)$ and  $(U_{\varphi'},C)$ up to a translation of the curve.

Finally, we construct the isomorphism (\ref{eq:NonLinSerreIsom}) by patching together the 
sectorial ones 
$$U_i\stackrel{\Pi_i}{\longrightarrow} V_i^*\stackrel{\psi_i}{\hookrightarrow} V_\varphi\setminus D$$
using the identity (\ref{eq:PiVarphiCocycleIdentity}) $\Pi_i=\varphi_{i,i+1}\circ\Pi_{i+1}$.
\end{proof}

\section{Foliations}\label{sec:Foliations}

Recall that our model $(U_{1,0,0},C)$ carries a pencil of foliations 
$$\mathcal F_t\ :\ \{\omega_0+t\omega_\infty=0\},\ \ \ \text{where}\ \ \ \omega_0=d\xi\ \ \ \text{and}\ \ \ \omega_\infty=\frac{1}{2i\pi\tau}\frac{dz}{z};$$
moreover, there is no other formal foliation on $(U_{1,0,0},C)$ either tangent, or transversal to $C$ (see \cite[Section 2.3]{LTT}).
Via the isomorphism $\Pi:U_{1,0,0}\setminus C\to V_0\setminus D$, we get the corresponding pencil
$$ \Pi_*\mathcal F_t\ :\ (1-t)\tau\frac{dX}{X}+t\frac{dY}{Y}.$$
The monodromy (or holonomy) of $\mathcal F_t$ is given by 
$$\pi_1(C)\to\mathrm{Aut}(\C)\ ;\ \left\{\begin{matrix}
1\mapsto [\xi\mapsto \xi+\frac{t}{\tau}]\hfill\\
\tau\mapsto [\xi\mapsto \xi+t-1]
\end{matrix}\right.$$
In particular, for $\frac{m}{n}\in\Q\cup\{\infty\}$, are equivalent
\begin{itemize}
\item $\mathcal F_t$ has trivial monodromy along $m+\tau n\in\Gamma\setminus\{0\}$, viewed as a loop of $\pi_1(C)\simeq\Gamma$;
\item $t=\frac{\tau n}{m+\tau n}$, or equivalently $\left(\frac{1}{t}-1\right)\tau=\frac{m}{n}$;
\item $\Pi_*\mathcal F_t$ admits the rational first integral $X^mY^n$.
\end{itemize}
We will say that $\mathcal F_t$ is of {\it rational type} if there is $\frac{m}{n}\in\Q\cup\{\infty\}$ with these properties, 
and of {\it irrational type} if not. We note that rational type foliations are characterized by the fact that their
holonomy group is cyclic (one generator), and also that the space of leaves (after deleting $C$ and regarding them as global foliations on $S_0$) is somehow "rational",
and not "elliptic".

If $(U,C)$ is any analytic neighborhood with a formal conjugacy 
$$\hat\Psi:(U,C)\stackrel{\sim}{\longrightarrow}(U_{1,0,0},C),$$ 
then it also
carries the pencil of formal foliations $\hat{\mathcal F}_t:=\hat\Psi^*\mathcal F_t$.
As we shall prove, these foliations are divergent in general. In fact, recall (see \cite[Theorem 4]{LTT})

\begin{thm}Let $(U, C)$ be an analytic neighborhood formally equivalent to $(U_{1,0,0},C)$. Assume
\begin{itemize} 
\item three elements $\hat{\mathcal F}_{t_1}$, $\hat{\mathcal F}_{t_2}$, $\hat{\mathcal F}_{t_3}$ 
of the pencil are convergent,
\item or two elements $\hat{\mathcal F}_{t_1}$, $\hat{\mathcal F}_{t_2}$ of the pencil are convergent,
both of irrational type: $\left(\frac{1}{t_i}-1\right)\tau\not\in\Q$ for $i=1,2$. 
\end{itemize}
Then the full pencil $\hat{\mathcal F}_{t}$ is convergent, and $(U, C)$ is analytically equivalent to $(U_{1,0,0},C)$
(in fact $\hat\Psi$ is convergent).
\end{thm}

In \cite[Theorem 5]{LTT}, the two first authors and O. Thom construct infinite dimensional deformations
of neighborhoods with two convergent foliations $\hat{\mathcal F}_{t_1}$ and $\hat{\mathcal F}_{t_2}$,
provided that one or two of them is of rational type. In fact, \'Ecalle-Voronin moduli spaces are shown to embed 
in moduli spaces of neighborhoods through these bifoliated constructions. Now, we know from our main
result Theorem \ref{THM:MainClassif}, that the moduli space of neighborhood is larger, comparable with 
$\C\{X,Y\}$, in contrast with \'Ecalle-Voronin moduli space which is comparable with $\C\{X\}$. The point  is 
that we have missed all neighborhoods with only one, or with no convergent foliation in the aforementioned work.

\subsection{Existence of foliations}\label{ssec:ExistenceFoliations}

We now start examinating under which condition on the glueing cocycle $\varphi=(\varphi_{i,i+1})$
the neighborhood $(U_\varphi,C)$ admits a convergent foliation. Here, we follow notations $\Psi_i$, $\Pi_i$, $\Phi_i$, ...
of Section \ref{sec:AnalyticClass}.

\begin{lemma}\label{lem:SectorialFolFtNonRat}
Let $t\in\PP^1\setminus\{0,1\}$. If the formal foliation $\hat{\mathcal F}_t$ of $(U,C)$ is convergent,
then the induced foliation ${\Pi_i}_*\hat{\mathcal F}_t$ on $V_i^*$ extends as a singular foliation on $V_i$ 
and is defined by a closed logarithmic $1$-form
\begin{equation}\label{eq:LogarithmicTheta}
\theta_i=(1-t)\tau\frac{dX}{X}+t\frac{dY}{Y}+\eta_i
\end{equation}
with $\eta_i$ closed holomorphic on $V_i$.
\end{lemma}

\begin{remark}\label{rem:LogFormVi}
On $V_i$ with $i=1,3$ (resp. $i=2,4$), the closed holomorphic $1$-form writes 
$$\eta_i=df\ \ \ \text{with}\ \ \ f\in\C\{X\}\ \ (\text{resp.}\ f\in\C\{Y\}).$$
We note that are equivalent:
\begin{itemize}
\item $\theta$ is a closed logarithmic $1$-form on $V_i$ with poles supported by $D$,
\item $\theta=\alpha\frac{dX}{X}+\beta\frac{dY}{Y}+\eta$ with $\eta$ holomorphic and closed on $V_i$,
\item $\theta={\varphi_i}^*\left\{\alpha\frac{dX}{X}+\beta\frac{dY}{Y}\right\}$ where $\varphi_{i}\in\mathrm{Diff}^1(V_i,L_i)$.
\end{itemize}
For instance, on $V_1$, if $\eta_1=df$ with $f\in\C\{X\}$, then $\varphi_i=(X e^{f(X)},Y)$.
\end{remark}

\begin{proof}Let us start with the tangent case $t\not=\infty$. Then, by transversality of ${\F}_{0}$ 
and ${\F}_{\infty}$, and the fact that the $\omega_t$'s  are $F_{1,0,0}$-invariant, one deduces that
$\hat{\mathcal F}_t$ is defined by a unique $1$-form  writing as $\omega=\omega_0-u\cdot\omega_\infty$ for a function 
$u\in{\mathcal A}[F_{1,0,0}](I_i)$, obviously satisfying $\hat u=t$. If $i=1,3$, then $u=t+f(X)$ with $f\in\C\{X\}$, $f(0)=0$
(see Proposition \ref{prop:CharacSectorialFunctionsF0}). Then, 
$$\frac{t}{u}\omega=t\frac{1-u}{u}\tau\frac{dX}{X}+t\frac{dY}{Y}
=(1-t)\tau\frac{dX}{X}+t\frac{dY}{Y}-\frac{f}{t+f}\tau\frac{dX}{X}.$$
If $i=2,4$, then $u=t+f(Y)$ with $f\in\C\{Y\}$, $f(0)=0$, and we arrive at a similar situation
$$\frac{1-t}{1-u}\omega=(1-t)\tau\frac{dX}{X}+t\frac{dY}{Y}+\frac{f}{1-t-f}\tau\frac{dX}{X}.$$
Of course, we have used $t\not=0,1$ in order to divide. 

Let us end with the case $\hat{\mathcal F}_\infty$. The foliation, in that case, can be defined by a closed holomorphic 
$1$-form $\omega$ extending the holomorphic $1$-form on $C$. Now, up to a multiplicative
constant, we can write $\omega=\omega_t+\eta$ for a closed holomorphic $1$-form $\eta$.
One concludes as above.
\end{proof}

\begin{lemma}\label{lem:SectorialFolFtNonRat01}
If the formal foliation $\hat{\mathcal F}_0$ (resp. $\hat{\mathcal F}_1$) of $(U,C)$ is convergent,
then  the induced foliation ${\Pi_i}_*\hat{\mathcal F}_t$ on $V_i^*$ extends as a singular foliation on $V_i$ 
and is defined by a $1$-form 
$$\theta_i=\left\{\begin{matrix}
\frac{dX}{X}+f_i(X)\frac{dY}{Y}&\text{if}&i=1,3\\
\frac{dX}{X}+\eta_i\hfill&\text{if}&i=2,4
\end{matrix}\right.\ \ \ \left(\text{resp.}\ \ \ 
\theta_i=\left\{\begin{matrix}
\frac{dY}{Y}+\eta_i\hfill&\text{if}&i=1,3\\
\frac{dY}{Y}+f_i(Y)\frac{dX}{X}&\text{if}&i=2,4
\end{matrix}\right.
\right)$$
with $f_i\in\C\{X\}$ (resp. $\C\{Y\}$), $f_i(0)=0$, and $\eta_i$ closed holomorphic $1$-form on $V_i$.
\end{lemma}

\begin{proof}It is similar to the proof of Lemma \ref{lem:SectorialFolFtNonRat}.
For the case $t=0$, once we have defined $\hat{\mathcal F}_t$ by $\omega=\omega_0-u\cdot\omega_\infty$
for a function $u\in{\mathcal A}^\infty[F_{1,0,0}](I_i)$, then
$$\frac{1}{1-u}\omega=\tau\frac{dX}{X}+\frac{u}{1-u}\frac{dY}{Y}.$$
Again, by Proposition \ref{prop:CharacSectorialFunctionsF0}, we see that if $i=2,4$, then $u=f(Y)$ with $f\in\C\{Y\}$, $f(0)=0$, and we are done. However, when $i=1,3$, then $u=f(X)$, but we cannot divide by 
$\frac{f}{1-f}$ to get a closed logarithmic $1$-form as before: as $f(0)=0$, the polar locus will increase (see remark \ref{rem:NonLog01}).
\end{proof}

\begin{remark}\label{rem:NonLog01}
In Lemma \ref{lem:SectorialFolFtNonRat01}, we can always define the foliation ${\Pi_i}_*\hat{\mathcal F}_t$
by a closed meromorphic $1$-form on $V_i$ provided that we allow non logarithmic poles. For instance, 
in the case $t=0$ and $i=1,3$, if $f\equiv0$ (is identically vanishing), there is nothing to do, it is the logarithmic case; 
if $f\not\equiv 0$, then after division, we get 
$$\frac{1}{f}\omega=\tau\frac{1-f}{f}\frac{dX}{X}+\frac{dY}{Y}=\tilde f(X)\frac{dX}{X^{k+1}}+\frac{dY}{Y}$$
with $\tilde f\in\C\{X\}$, $\tilde f(0)\not=0$, and $k\in\Z_{>0}$. As it is well-known (see \cite[Section 2.2]{LTT}),
we can write 
$${\varphi_i}^*\frac{\omega}{f}=\frac{dX}{X^{k+1}}+\alpha\frac{dX}{X}+\frac{dY}{Y}$$
for some $\alpha\in\C$ (the residue of $\tilde f(X)\frac{dX}{X^{k+1}}$) and $\varphi_{i}\in\mathrm{Diff}^1(V_i,L_i)$.
\end{remark}

We can now prove Theorem \ref{THM:ExistenceFoliations}.

\begin{cor}\label{cor:ExistenceFol}
The formal foliation $\hat{\mathcal F}_t$ of $(U_\varphi,C)$ is convergent
if, and only if, there exist $\eta_i$ closed holomorphic $1$-forms on $(V_i,L_i)$ such that
$$({\varphi_{i,i+1}}^*\theta_i)\wedge\theta_{i+1}=0\ \ \ \text{where}\ \ \ \theta_i=(1-t)\tau\frac{dX}{X}+t\frac{dY}{Y}+\eta_i.$$
Equivalently, there exists an equivalent cocycle $\varphi'\sim\varphi$ such that
$$({\varphi_{i,i+1}'}^*\theta^0)\wedge\theta^0=0\ \ \ \text{where}\ \ \ \theta^0=(1-t)\tau\frac{dX}{X}+t\frac{dY}{Y}.$$
\end{cor}

\begin{proof}When $t\not=0,1$, the proof easily follows from Lemma \ref{lem:SectorialFolFtNonRat}.
Indeed, all ${\Pi_i}_*\hat{\mathcal F}_t$ are defined by $\theta_i=0$ and have to patch via the 
glueing maps $\varphi_{i,i+1}$. Conversely, if $\theta_i=0$ patch via the 
glueing maps $\varphi_{i,i+1}$, then this means that we get a foliation $\mathcal F$ on $U_\varphi\setminus C$
which is flat to $\hat{\mathcal F}_t$ along $C$, and therefore extends by Riemann. 
Using Remark \ref{rem:LogFormVi} and Definition \ref{def:EquivalentCocycles}, one easily derives the second 
(equivalent) assertion. Finally, in the case $t=0$ for instance, after applying Lemma \ref{lem:SectorialFolFtNonRat01}
in a very similar way, we note that $\theta_i$ defines a regular foliation on $V_i$ for $i=2,4$ (as $\eta_i=df$, $f\in\C\{X\}$).
On the other hand, on $V_i$ for $i=1,3$, $\theta_i$ defines a singular foliation as soon as $f_i\not\equiv 0$
(non identically vanishing), i.e. with a saddle-node singular points at the two points $p_{i,i+1}$ and $p_{i-1,i}$;
therefore, $f_i\equiv0$ in the case we have a global foliation and we are back to the logarithmic case.
The proof ends-up like before.
\end{proof}

\begin{remark}\label{rem:FolVvarphi}
The statement of Corollary \ref{cor:ExistenceFol} can be reformulated as follows. 
The formal foliation $\hat{\mathcal F}_t$ of $(U_\varphi,C)$ is convergent
if, and only if, there exists a foliation $\mathcal G_t$ on $(V_\varphi,D)$
which is locally defined by a closed logarithmic $1$-form with poles supported on $D$ and having residues
$t$ on $L_1$ and $(1-t)\tau$ on $L_4$ (we have automatically opposite residues on opposite sides of $D$). 
Indeed, the local foliations $\theta_i=0$ patch together.
\end{remark}

We can precise Corollary \ref{cor:ExistenceFol} for generic $t$ as follows.

\begin{prop}\label{prop:IrrationalFolClosed1Form}
If $\hat{\mathcal F}_t$ is not of rational type, i.e. $\left(\frac{1}{t}-1\right)\tau\not\in\Q$,
then are equivalent
\begin{enumerate}
\item $\hat{\mathcal F}_t$ is convergent,
\item $\hat{\mathcal F}_t$ is defined by a closed (convergent) meromorphic $1$-form $\omega$,
\item $({\varphi_{i,i+1}})^*\theta_i=\theta_{i+1}$ with $\theta_i$ like in Corollary \ref{cor:ExistenceFol},
\item there is a closed logarithmic $1$-form $\theta$ on $(V_\varphi,D)$
with poles supported on $D$ and having residues $t$ on $L_1$ and $(1-t)\tau$ on $L_4$.
\end{enumerate}
Obviously, $\omega=\Pi_\varphi^*\theta$ up to a constant.
\end{prop}

\begin{proof}When $\hat{\mathcal F}_t$ is not of rational type, then we have
$$({\varphi_{i,i+1}}^*\theta_i)\wedge\theta_{i+1}=0\ \ \ \Leftrightarrow\ \ \ ({\varphi_{i,i+1}}^*\theta_i)=\theta_{i+1}.$$
Indeed, if ${\varphi_{i,i+1}}^*\theta_i$ is colinear to $\theta_{i+1}$, then it is proportional to $\theta_{i+1}$, i.e. it writes
$f_i\cdot\theta_i$ with $f_i$ meromorphic on $V_i$ (de Rham-Saito Lemma). But since it is also closed, 
we have 
$$0=d(f_i\cdot\theta_i)=df_i\wedge \theta_i+f_i\wedge \underbrace{d\theta_i}_{=0}$$
and $f_i$ is a meromorphic first integral for $\theta_i=0$, which must be constant in the irrational type.
This constant must be $=1$ as the residues are preserved. As a consequence, all $\theta_i$ patch together 
on $V_\varphi$. Finally, note that if $\hat{\mathcal F}_t$ is convergent, then its holonomy is not cyclic 
(because not of rational type) and therefore preserves a meromorphic $1$-form on the transversal
that we can extend as a closed meromorphic $1$-form $\omega$ defining the foliation.
\end{proof}

Let us now illustrate how different is the situation for foliations of rational type 
by revisiting the classification \cite[Theorem 5]{LTT} of neighborhoods with $2$ 
convergent foliations, in the particular case of  $\hat{\mathcal F}_{0}$ and $\hat{\mathcal F}_{1}$,
corresponding respectively to vertical and horizontal foliations on $V_\varphi$.
The proof is a straightforward application of the above criteria.

\begin{prop}\label{prop:ExampleMartinetRamis}
The formal foliations $\hat{\mathcal F}_{0}$ and $\hat{\mathcal F}_{1}$ on $(U,C)$
are convergent if, and only if,  $(U,C)$ can be defined by a cocycle of the form
$$\left\{\begin{matrix}
\varphi_{i,i+1}(X,Y)=(\alpha_{i}(X),Y)&\text{for}&i=1,3\\
\varphi_{i,i+1}(X,Y)=(X,\alpha_{i}(Y))&\text{for}&i=2,4
\end{matrix}\right.$$
for $1$-variable diffeomorphisms $\alpha_{i}$ tangent to the identity,
and the corresponding foliations on $V_\varphi$ 
are respectively defined in charts $V_i$ by $dX=0$ and $dY=0$.
Moreover, this normalization is unique up to conjugacy by $\phi(X,Y)=(e^tX,e^{\tau t}Y)$.
\end{prop}

The space of leaves of $\hat{\mathcal F}_{0}$ on $U\setminus C$ 
corresponds to the space of orbits for its holonomy map,
and therefore to Martinet-Ramis' ``{\it Chapelet de sph\`eres}'' (see \cite[page 591]{MartinetRamis2}).
It is given by two copies of $\C_X^*$ patched together by means of diffeomorphism germs
$\alpha_{1}(X)$ at $X=\infty$ and $\alpha_3(X)$ at $X=0$. A similar description holds
for $\hat{\mathcal F}_{0}$ with Martinet-Ramis' cocycle $\alpha_2$ and $\alpha_4$.
The invariants found by the third author in \cite{Voronin} are related with the corresponding periodic 
transformations in variable $\xi$.

\begin{figure}[htbp]
\begin{center}
\includegraphics[scale=0.5]{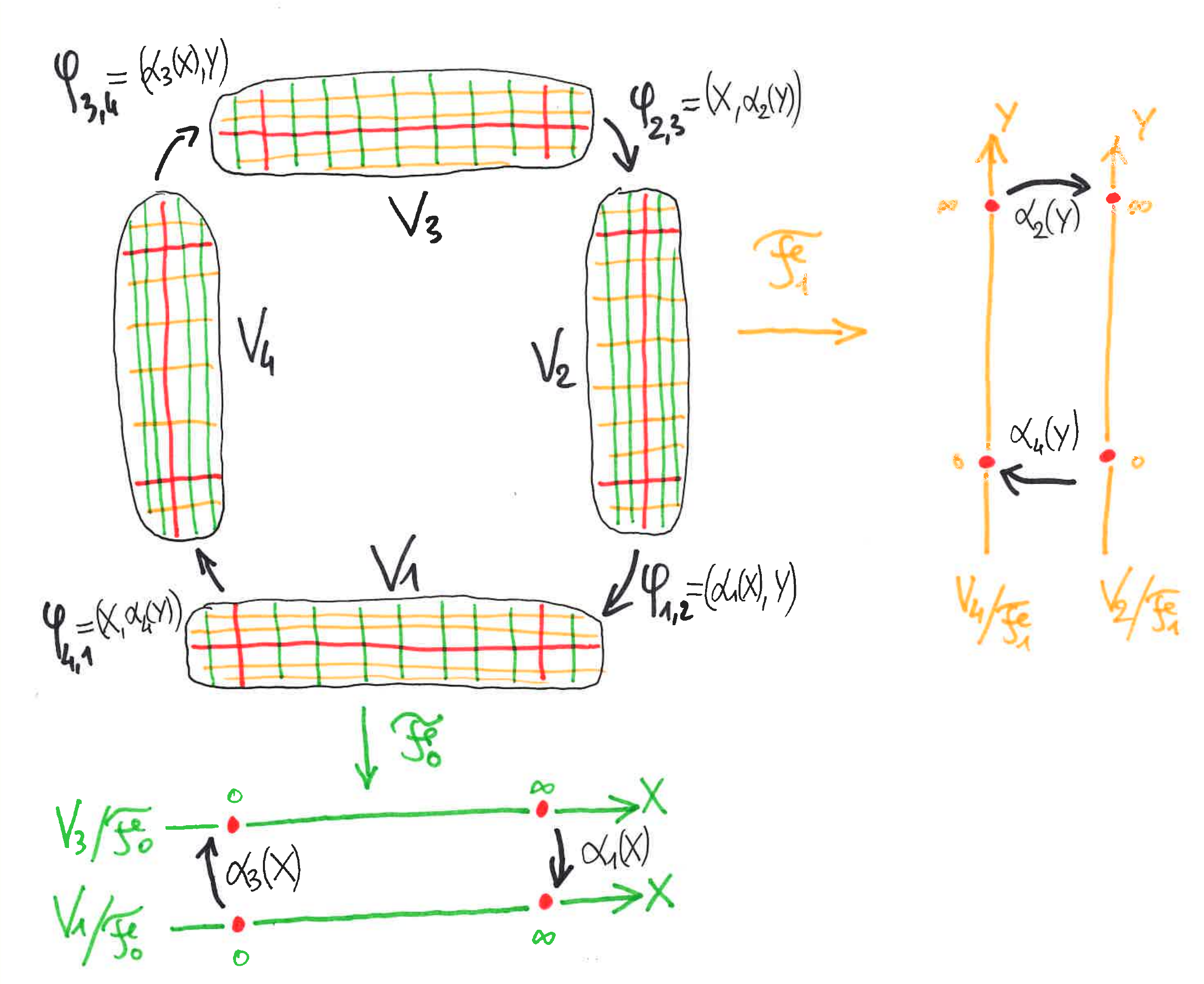}
\caption{{\bf Martinet-Ramis moduli}}
\label{pic:MRmoduli}
\end{center}
\end{figure}

\begin{remark}\label{rem:UnicityFolt}
It follows from \cite{LTT}, or from the unicity of the formal pencil $\hat{\F}_t$, 
that for given $t\in\PP^1$, one cannot find two different collections $(\theta_i)_i$ and $(\theta_i')_i$ 
defining two global logarithmic foliations $\G_t$ and $\G_t'$ on $V$, like in Corollary \ref{cor:ExistenceFol}. 
One way to see this directly from the point of view of this section is as follows.
In the irrational case $\left(\frac{1}{t}-1\right)\tau\not\in\Q$, we see from Proposition \ref{prop:IrrationalFolClosed1Form}
that $\theta_i$'s patch as a global closed logarithmic $1$-form. But the difference between two closed logarithmic $1$-forms 
with the same residues is a closed holomorphic $1$-form $\eta$ on $(V,C)$. Now, $\eta$ must be zero, 
even if we restrict on two consecutive line neighborhoods $(V_i,L_i)\cup_{\varphi_{i,i+1}} (V_{i+1},L_{i+1})$, 
as it only depends on $X$ or on $Y$ depending on the sector. In the rational case, $\left(\frac{1}{t}-1\right)\tau\not\in\Q$,
there is also unicity of $\theta_i$'s on two consecutive line neighborhoods whose residues have quotient $>0$;
indeed, after blowing-up, we get a rational fibration which must be unique by Blanchard Lemma.
\end{remark}

\subsection{Non existence of foliations}
For a generic neighborhood $(U_\varphi,C)$, there is no convergent foliation.
In order to prove this, it is enough to provide a single example without foliation.
Such an example has been given quite recently by Mishustin in \cite{Mishustin}.
With our Corollary \ref{cor:ExistenceFol}, it is not too difficult to provide an example
without foliations.

\begin{thm}\label{TH:nonexistence}Let $(U_\varphi,C)$ be a neighborhood
such that 
$$\varphi_{1,4}(X,Y)=(X(1+XY),Y(1+X^2Y)).$$ Then all foliations $\hat{\F}_t$ belonging
to the formal pencil $(\hat{\F}_t)_{t\in{\PP}^1}$ are divergent.
\end{thm}

Even the transversal fibration $\hat{\F}_\infty$ is divergent in that case.

\begin{proof}
Suppose by contradiction that there exists at least one convergent foliation in the pencil. Then, by Corollary \ref{cor:ExistenceFol}, there exists on each $V_i$ a non trivial logarithmic $1$-form
$$\theta_i=\alpha \frac{dX}{X}+\beta\frac{dY}{Y}+ 
\left\{\begin{matrix}f_i(X)dX&\text{if $i$ even}\\f_i(Y)dY&\text{if $i$ odd}\end{matrix}\right.$$ 
with $f_i:(\C,0) \to\C$ holomorphic, such that $\theta_{i+1}\wedge (\varphi_{i,i+1})^*\theta_i=0$.
One has 
$$(\varphi_{1,4})^*\theta_1=\alpha\frac{dX}{X}+\alpha \frac{d(XY)}{1+XY} +\beta\frac{dY}{Y} + \beta\frac{d(X^2Y)}{1+X^2Y}$$
$$+f_1(Y(1+X^2Y))\cdot((1+2X^2Y)dY+2XY^2dX).$$
The residual parts of the $2$-form $\theta_4\wedge(\varphi_{1,4})^*\theta_1$ at $X=0$ and $Y=0$ respectively write
\begin{equation}\label{eq:NoFolResXY}
\alpha f_1(Y)\frac{dX}{X}\wedge dY
\ \ \ \text{and}\ \ \ \beta f_4(X)dX\wedge \frac{dY}{Y}.
\end{equation}
Both expressions must be vanishing identically. If $\beta=0$, then $\alpha\not=0$ and we deduce from (\ref{eq:NoFolResXY}) 
that $f_1\equiv0$. This implies that $(\varphi_{1,4})^*\theta_1$ only depends on $X(1+XY)$, while
$\theta_4$ only depends on $X$, contradiction. Assume now $\alpha=0$ (and $\beta\not=0$); then, 
by (\ref{eq:NoFolResXY}), we have $f_4\equiv0$. Again, we conclude that
 $(\varphi_{1,4})^*\theta_1$ only depends on $X(1+X^2Y)$, while $\theta_4$ only depends on $Y$, contradiction. 
Finally, assume that $\alpha\not=0$ and $\beta\not=0$; then, by (\ref{eq:NoFolResXY}), we have $f_1,f_2\equiv0$ and  we obtain
$$\theta_4\wedge(\varphi_{1,4})^*\theta_1=
\frac{\alpha(\alpha-\beta)+\beta(\alpha-2\beta)X+(\alpha^2-2\beta^2)X^2Y}{(1+XY)(1+X^2Y)}dX\wedge dY.$$
Clearly, this expression cannot be zero if $\alpha$ and $\beta$ are both non zero. 
\end{proof}

\subsection{Only one convergent foliation}
To complete the picture, it is interesting to provide an example of a foliation  having only one convergent foliation  $\hat{\F}_{t_0}$
in the pencil ${ (\hat{\F}_t)}_{  
t\in{\mathbb P}^1}$ for an arbitrary $t_0$.

\begin{thm}\label{TH:OnlyOneFoliation}
Let $t_0=[u_0:v_0]\in{\mathbb P}^1$, and let $(U_\varphi,C)$ be the neighborhood
such that 
$$\left\{\begin{matrix} 
	\varphi_{1,4}=(X{(1+XYe^Y)}^{-\beta_0}, Y{(1+XYe^Y)}^{\alpha_0})\\
\text{where}\ \ \ \alpha_0=\tau(v_0-u_0)\ \text{and}\ \beta_0=u_0\hfill\end{matrix}\right.
\ \ \ \text{and}\ \ \ 
 \left\{\begin{matrix}
 \varphi_{i,i+1}=\mbox{Id}\hfill\\
\text{for}\ i=1,2,3.\end{matrix}\right.$$
Then $\hat{\F}_{t_0}$ is the unique convergent foliation in the formal pencil 
$(\hat{\F}_t)_{t\in{\PP}^1}$.
\end{thm}

\begin{proof}
Note that $\varphi$ preserves the foliation defined by the logarithmic form $\alpha_0 \frac{dX}{X} +\beta_0\frac{dY}{Y}$ which then descends on $V_\varphi$, i.e. the formal foliation $\hat{\F}_{t_0}$ is indeed convergent on $U_\varphi$. Assuming by contradiction that there is another convergent foliation $\hat{\F}_t$ with $t\not= t_0$, let $\theta_i$ be the associated logarithmic $1$-form on $V_i$. As in the proof of Theorem \ref{TH:nonexistence}, we get 
$$\theta_1=\alpha\frac{dX}{X}+\beta\frac{dY}{Y} + f_1 (Y)dY\ \ \ \text{and}\ \ \ 
\theta_4=\alpha\frac{dX}{X}+\beta\frac{dY}{Y} + f_4 (X)dX$$
with $f_i:(\C,0) \to\C$ holomorphic, and $[\alpha:\beta]\not=[\alpha_0:\beta_0]$. We derive
$$(\varphi_{1,4})^*\theta_1= \alpha\frac{dX}{X} +\beta\frac{dY}{Y}+ (\alpha_0\beta -\alpha\beta_0) \frac{d(XYe^Y)}{1+XYe^Y}$$
$$ +f_1(Y{(1+XYe^Y)}^{\alpha_0})\cdot d(Y(1+XYe^Y)).$$ 
The residual parts of the $2$-form $\theta_4\wedge(\varphi_{1,4})^*\theta_1$ at $X=0$ and $Y=0$ respectively write
\begin{equation}\label{eq:OneFolResXY}
\alpha f_1(Y)\frac{dX}{X}\wedge dY
\ \ \ \text{and}\ \ \ \beta f_4(X)dX\wedge \frac{dY}{Y}.
\end{equation}
We are led to a similar discussion as in the proof of Theorem \ref{TH:nonexistence}. 
When $\beta=0$, then $f_1\equiv0$ and $(\varphi_{1,4})^*\theta_1$ only depends on 
$X{(1+XY)}^{-\beta_0}$, while $\theta_4$ only depends on $X$; we get a contradiction since $\beta_0\not=0$ in this case. 
When $\alpha=0$, then $f_4\equiv0$ and $(\varphi_{1,4})^*\theta_1$ only depends on 
$Y{(1+XY)}^{\alpha_0}$, while $\theta_4$ only depends on $Y$; contradiction.
Finally, when both $\alpha\not=0$ and $\beta\not=0$, then 
$$\theta_4\wedge(\varphi_{1,4})^*\theta_1=
\underbrace{(\alpha_0\beta -\alpha\beta_0)}_{\not=0} \frac{d(XYe^Y)}{1+XYe^Y}\wedge
\left(\alpha\frac{dX}{X} +\beta\frac{dY}{Y}\right)$$
which cannot be zero, again a contradiction.
\end{proof}

\begin{remark}\label{rem:OneFol}
More generally, given $f\in\C\{X,Y\}$ vanishing along $X=0$ and $Y=0$, the same proof shows that 
the cocycle defined by $\varphi_{1,4}=(Xe^{-\beta_0 f}, Ye^{\alpha_0 f})$ and $\varphi_{i,i+1}=\mbox{Id}$ otherwise
also provides a neighborhood $(U_\varphi,C)$ with only one convergent foliation, namely $\hat{\F}_{t_0}$, 
provided that $df\wedge d(X^pY^q)\not\equiv0$ for all $p,q\in\Z_{>0}$. Moreover, one easily checks 
that two different such $f$, says $f$ and $f'$, define non equivalent neighborhoods provided that their
difference do not take the form $f'-f=g(X)+h(Y)$. 
\end{remark}

\section{Symmetries}\label{sec:symmetries}

Let $(U_\varphi,C)$ be a neighborhood formally equivalent to $(U_{1,0,0},C)$.
Via formal conjugation, formal symmetries (or automorphisms) of $(U_\varphi,C)$  which restrict to translations on $C$
are those of $(U_{1,0,0},C)$, i.e. of the form (see Corollary \ref{cor:formalcentralizer}):
$$\left\{\begin{matrix}(z,\xi)\mapsto(cz,\xi+t)\\ c\in\C^*,\ t\in\C\end{matrix}\right.
\ \ \ \leftrightarrow\ \ \ 
(a,b):(X,Y)\mapsto(\underbrace{e^{2i\pi t}}_{a}X,\underbrace{c^{-1}e^{2i\pi\tau t}}_{b}Y)$$
The subgroup $\mathrm{Aut}(U_\varphi,C)$ of convergent automorphisms thus identifies with a 
subgroup of the two dimensional linear algebraic torus:
$$G\subset\mathrm{Aut}^0(\mathbb P^1 \times \mathbb P^1,D)\simeq\C^*\times\C^*.$$

\begin{thm}\label{thm:symmetries}
Let $(U_\varphi,C)$ and $G$ be as above. Then the subgroup $G\subset\C^*\times\C^*$ is algebraic.
In particular, we are in one of the following cases:
\begin{itemize}
\item $G$ is finite and $G=\{(a,b)\ ;\ a^pb^q=a^{p'}b^{q'}=1\}$ for some  non proportional $(p,q),(p',q')\in\Z^2\setminus(0,0)$;
\item $G=\{(a,b)\ ;\ a^pb^q=1\}$ for some $(p,q)\in\Z^2\setminus(0,0)$; in particular, a finite index subgroup of $G$
is generated by the flow of the rational vector field $p X\partial_X+q Y\partial_Y$; 
\item $G=\C^*\times\C^*$ and $(U_\varphi,C)\an(U_{1,0,0},C)$.
\end{itemize}
Moreover, in the first two cases, up to equivalence $\approx$, the cocycle takes the form
$$\varphi_{i,i+1}(X,Y)=(X\cdot u_{i,i+1},Y\cdot v_{i,i+1})$$
where $u_{i,i+1},v_{i,i+1}$ are Laurent series in $X^pY^q$ and $X^{p'}Y^{q'}$ (resp. in $X^pY^q$)
and the action of $G$ is linear in each chart $(V_i,L_i)$.
\end{thm}
\begin{remark}
Note that $G$ is not algebraic in general as a subgroup of $\mathrm{Aut}(U_{1,0,0},C)=\mathrm{Aut}^0(S_0)$. Actually, the correspondance between  $\mathrm{Aut}(U_{1,0,0},C)$ and $\mathrm{Aut}^0(\mathbb P^1 \times \mathbb P^1,D)$ specified above is only of analytic nature.\end{remark}

\begin{proof} By similar arguments as in the proof of Lemma \ref{lem:SectorialFolFtNonRat},
we see that each automorphism of $(U_\varphi,C)$ corresponds to a collection of automorphisms
of $g_i\in\mathrm{Diff}(V_i,L_i)$ satisfying 
$$g_{i}\circ\varphi_{i,i+1}=\varphi_{i,i+1}\circ g_{i+1}$$ and that the corresponding element $(a,b)$ in $\C^*\times\C^*$ is the linear part of $g_i$ at the crossing points $p_{i-1,i},p_{i,i+1}$ (in particular, it is independant of $i$).

We first prove that $g_i$ can be linearized in each chart. 
Indeed, for instance on $(V_4,L_4)$, $g_4$ acts by transformations of the form 
$$(X,Y)\mapsto(aX\cdot u(X),bY\cdot v(X)),\ \ \ u(0)=v(0)=1.$$
The gluing condition for $\varphi_{4,1}$
shows that the restriction $\varphi_4\vert_{L_1}:X\mapsto aX\cdot u(X)$ becomes linear in the chart $(V_1,L_1)$.
Therefore, after changing $X$ coordinates on $(V_4,L_4)$, we can assume 
$$g_4(X,Y)=(aX,bY\cdot v(X)),\ \ \ v(0)=1.$$
But $g_4$ must also preserve the fibration $dY=0$ of $(V_1,L_1)$ which is preserved by $g_1$; 
that can be normalized to $dY=0$ in the chart $(V_4,L_4)$ and this implies that $v(X)\equiv 1$ also.
We therefore conclude that for each automorphism in $\mathrm{Aut}(U_\varphi,C)$, the corresponding 
transformation in $(V_\varphi,D)$ can be linearized in all charts $(V_i,L_i)$.

As a by-product,  $G$ must contain the Zariski closure of $<(a,b)>\subset\C^*\times\C^*$.
Indeed, all $\varphi_{i,i+1}$ have to commute with $g(X,Y)=(aX,bY)$.  
Writing $\varphi_{i,i+1}(X,Y)=(X\cdot u(X,Y),Y\cdot v(X,Y))$, we see that $u,v$ have to be invariant by $g$,
i.e. $u\circ g=u$ for instance; equivalently, all non zero monomials of $u$ and $v$ are $g$-invariant.
These monomials define an algebraic subgroup $H\subset\C^*\times\C^*$ which is the group of 
linear transformations commuting with  $\varphi_{i,i+1}$. We conclude that $(a,b)\in H\subset G$.

If $g$ was Zariski dense in $\C^*\times\C^*$, we are done: the commutation of  $\varphi_{i,i+1}$ with all linear transformations 
shows that  $\varphi_{i,i+1}$ is linear (hence trivial) and $(U_\varphi,C)\an(U_{1,0,0},C)$. Now, assuming that $G$ is a strict
subgroup of $\C^*\times\C^*$, we want to prove that it can be linearized globally.

Assume first that $G$ is finite. Then it can be linearized on each line neighborhood $(V_i,L_i)$.
Indeed, for instance on $(V_4,L_4)$, $G$ acts by transformations of the form 
$$g(X,Y)=(aX\cdot u(X),bY\cdot v(X)),\ \ \ u(0)=v(0)=1.$$
and we denote by $\mathrm{lin}(g)$ its linear part $(aX,bY)$.
Then the transformation 
$$\varphi_4:=\frac{1}{\# G}\sum_{g\in G}\mathrm{lin}(g)^{-1}\circ g$$
is of the form $\varphi_4(X,Y)=(X\cdot u(X),Y\cdot v(X))$, $u(0)=v(0)=1$ and linearizing the group: 
$$\varphi_4\circ g=\mathrm{lin}(g)\circ \varphi_4,\ \ \ \forall g\in G.$$
We can therefore assume that $G$ acts linearly in each chart $(V_i,L_i)$ and the cocycle $\varphi$
has to commute with all elements. It is well known that the group $G$ is generated by two elements
$(a_1,b_1)$ and $(a_2,b_2)$ of finite order; moreover, by duality, $G$ is defined by $2$ independant monomial equations
$a^pb^q=a^{p'}b^{q'}=1$. Gluing conditions with $\varphi_{i,i+1}$ show that  $u_{i,i+1},v_{i,i+1}$ must 
be $G$-right-invariant and therefore factor through the two monomial equations.

On the other hand, if $G$ contains an element $g$ of infinite order, then we can first linearize this element.
The Zariski closure $H$ of its iterates $<g>$ in $\C^*\times\C^*$ is one dimensional,
 if strictly smaller than  $\C^*\times\C^*$, and defined by a monomial equation $a^pb^q=0$. 
 If $G$ is larger than $H$, then it is generated by an element of finite order $g'$ and we can linearize 
 the finite group $<g'>$ like above; since $g'$ and its linear part both commute with $H$, 
 the linearizing transformations also commute with $H$ and $G$ is linearized. It is therefore algebraic,
 defined by monomial equations of $\varphi_{i,i+1}$, and they all factor into a single monomial. 
\end{proof}

\begin{remark}\label{rem:TransversalAutomorphism}
From the description above, we note that the convergence of a non trivial automorphism $g$ of $(U_\varphi,C)$ inducing the identity on $C$
implies that $(U_\varphi,C)\an(U_{1,0,0},C)$, since the group generated by $g$ 
must be Zariski dense in $\C^*\times\C^*$.
\end{remark}

\begin{remark}\label{rem:VectorField}
In Proposition \ref{prop:ExampleMartinetRamis}, the foliation $\hat{\F}_0$ is defined by a 
holomorphic vector field if, and only if, $\alpha_2(Y)=\alpha_4(Y)=Y$. Equivalently, 
the Martinet-Ramis invariant of $\hat{\F}_1$ are trivial, i.e. $\hat{\F}_1$ can be defined by a closed $1$-form.
\end{remark}

\section{Sectorial normalization}\label{S:sectorialnorm1}
We maintain the foregoing notations. 
Recall that we have set $\varpi=\arg{\tau}\in ]0,\pi[$.
Let $(U,C)$ be formally equivalent to $(U_{1,0,0},C)$. We want to show that $(U,C)$ has the form $U_\varphi$.
In other word, we want to prove Lemma \ref{LEM:sectorialnormalization}, or equivalently Lemma \ref{lem:sectorialRappel}.

\subsection{Overview of the proofs} \label{SS:overview}
One can suppose that  $(U,C)=(\tilde U, \tilde C)/F$ where 
$$F(z,\xi)=\underbrace{F_{1,0,0}(z,\xi)}_{(qz, \xi -1)}+(\Delta_1,\Delta_2)\ \ \ \text{with}\ \ \ \Delta_i=O(\xi^{-N})$$
where $N>>0$ is an arbitrarily large integer, so that there exists a formal diffeomorphism 
$$\hat{\Psi}(z,\xi)=(z+\hat{g},\xi+\hat{h}),\ \ \ \hat{h}=\sum_{n\geq 1}a_n \xi^{-n},\ \ \ \hat{g}=\sum_{n\geq 1}b_n \xi^{-n}$$
 where $a_n,b_n$ are entire functions on $\tilde C={\C}^*$ such that
$$F\circ \hat{\Psi}=\hat{\Psi}\circ F_{1,0,0}.$$

This can be reformulated as 
\begin{equation}\label{E:F2}
 \hat{g}\circ F_{1,0,0}-q \hat{g}=\Delta_1\circ \hat{\Psi}
\end{equation}
\begin{equation}\label{E:F1}
\hat{h}\circ F_{1,0,0} -\hat{h}=\Delta_2\circ \hat{\Psi}
\end{equation}

Basically, we will show that there exists a \textit{holomorphic} solution $\Psi=\mbox{Id} + (h,g)$ 
of the previous fonctional equations $F\circ {\Psi}={\Psi}\circ F_{1,0,0}$, i.e. with $h,g$ satisfying
\begin{equation}\label{E:F2hol}
 {g}\circ F_{1,0,0}-q {g}=\Delta_1\circ {\Psi}
\end{equation}
\begin{equation}\label{E:F1hol}
{h}\circ F_{1,0,0} -{h}=\Delta_2\circ {\Psi}
\end{equation}
defined on "suitable sectorial domains", namely on $U_i:=\Pi^{-1}(V_i -L_i)$, $i=1,2,3,4$, 
and admitting asymptotic expansion along $\tilde C$ compatible with the formal conjugacy map $\hat{\Psi}$.
More precisely $h,g\in {\mathcal A } (I_i)$ 
with the notations of section \ref{SS:sheaves} with respective asymptotic expansions $\hat h$ and $\hat g$.
To this end, we will first exhibit solutions in  ${\mathcal O } (I_i)$  satisfying some suitable growth behaviour  
of  the following linearized equations:
 \begin{equation}\label{E:lin2}
 g\circ F_{1,0,0}-q g=\Delta_1
\end{equation}
\begin{equation}\label{E:lin1}
h\circ F_{1,0,0}-h=\Delta_2
\end{equation}
Most of section \ref{S:sectorialnorm1} is devoted to the construction of such sectorial solutions 
on the sector $U_1$, and it will be explained in subsection \ref{SS:othersectors} how to deduce normalization
on other sectors.

\begin{remark}\label{R:renorm}
The equation (\ref{E:lin2}) can be reduced to equation (\ref{E:lin1}); indeed, after setting
$$g(z,\xi)=z\tilde{g}(z,\xi)\ \ \ \text{and}\ \ \ \Delta_1(z,\xi)=q z \widetilde{ \Delta}_1(z,\xi),$$
we get
$$\tilde{g}\circ F_{1,0,0}-\tilde{g}=\widetilde{\Delta}_1.$$
\end{remark}

This will enable us to solve by a fairly standard fixed point method the initial  functional equations 
(\ref{E:F2hol}) and (\ref{E:F1hol}). In order to get rid of the coefficient $q$ on the left hand side, 
note that both equations can be reformulated  as:
\begin{equation}\label{E:lin2an}
\tilde{g} (q z,\xi -1) -\tilde{g}(z,\xi)=(1+\tilde{g}(z,\xi))\widetilde{\Delta}_1\big( z(1+\tilde{g}(z,\xi)),\xi+h(z,\xi)\big)
\end{equation}
\begin{equation}\label{E:lin1an}
h( q z,\xi -1) - h(z,\xi)=\Delta_2\big(z(1 + \tilde{g}(z,\xi)),\xi +h(z,\xi)\big)
\end{equation}
where the symbol $\ \widetilde\ $ stands for the same modification than in Remark \ref{R:renorm} for the linear case.

\subsection{The linearized/homological equation}\label{SS:linarized} Our purpose is to construct some sectorial solution of the linearized functional equations (\ref{E:lin1}) (and therefore (\ref{E:lin2}) by Remark \ref{R:renorm}) belonging to $\mathcal O (I_1)$. Actually, one will just firstly state some results and use this material to undertake the resolution of the complete (non linear) conjugacy equation (over $I_1$). We will  detail the resolution of the linearized equation (the most technical part) in
subsection \ref{SS:sollin}.
The existence of other sectorial conjugacy maps over $I_i, i\not=1$  can be obtained in a very similar way and we indicate briefly how to proceed in subsection \ref{SS:othersectors}.

  Let us first settle some notations. Recall that $q=e^{2i\pi\tau}$ with $\Im\tau>0$, so that $\vert q\vert<1$.
As one only focuses on transversal sectorial domain determined by $I_1=]-\varpi,\pi-\varpi[$, $\varpi=\arg(\tau)$, 
we are  going to work in domains $S$ of the following shape. Fix $0<a<b$ such that $a<|q|b$ and 
consider  the annulus 
 $${\mathcal C}_{a,b}=\{{a}\leq|z|\leq b\}.$$ 
  Let $\delta_{(a,b)}>0$ small enough 
  and for $0<\delta\leq \delta_{(a,b)}$, set 
 $$S_{a,b, \delta}=\{(z,\xi)\in {\C}^2|\ |Y(z,\xi)|\le\delta\ \mbox{and}\ z\in {\mathcal C}_{a,b}\}$$ 
where $Y(z,\xi)=z e^{2i\pi\tau \xi}$ (recall that $\vert Y\vert<\delta$  corresponds to a neighborhood $V_1$ of $L_1$).
 Alternatively, this set can be described by the equation $\Re{(2i\pi\tau\xi)}<\log(\frac{ \delta}{ |z|})$, $z\in {\mathcal C}_{a,b}$ so that in particular $\arg{\xi}\in I_1=]-\varpi,\pi-\varpi[$ or equivalently $\Im{(\tau\xi)}>0$.  Note that $F_{1,0,0}(S_{a,b,\delta})=S_{|{q}| a,|{q}| b,\delta}$.
It is thus coherent to investigate the existence of a solution $h$ of (\ref{E:lin1}) on the domain $S_{ |q|a,b,\delta}=S_{ a, b,\delta}\cup F_{1,0,0}(S_{ a, b,\delta})$.

 In what follows, we will indeed provide a solution of (\ref{E:lin1}) with "good estimates" on a domain of the form $S_{ |{q}|a, b,\delta}$ using a "leafwise" resolution with respect to the foliation defined by the levels of $Y$. For the sake of notational simplicity we will omit for a while the subscript $a,b$ by setting $S_\delta:= S_{a,b,\delta}$ and ${S}_\delta':=  S_{|{q}|a, b,\delta}$. If $\delta_1\leq \delta_2$, remark that $S_{\delta_1}$ and $S_{\delta_1}'$ are respectively subdomains of  $S_{\delta_2}$ and $S_{\delta_2}'$. To state precisely our result, let us fix some additional notations and definitions. Let $m\geq 3$ a positive integer and consider the subspace $H_\delta^m$ of ${\mathcal H}(S_\delta)$  \footnote{Let $A\subset \C^N$, in this paragraph and hereafter,  
 ${\mathcal H}(A)$ will denote the algebra of \textit{holomorphic functions on $A$}, that is the  $\C$-valued continuous functions 
 on $A$ which are holomorphic in the interior in the usual sense.} 
 defined by the  functions $\Delta$ such that 
 $${\left\|\Delta\right\|}_m:=\sup_{(z,\xi)\in S_{\delta}}|\Delta(z,\xi)|{{|\xi|}^m}<\infty.$$ 
 We will also introduce the space ${H_\delta^\infty}'$ of \textit{bounded}  holomorphic functions $h$  on $S_{\delta}'$ equipped with the natural norm 
 $${\left\|h\right\|}_\infty:=\sup_{(z,\xi)\in S_{\delta}'}|h(z,\xi)|<\infty.$$
 
 \begin{thm}\label{TH:lin}
 Fix $a,b$ as above. 
 Then, there are positive constants $\delta_{(a,b)},C$ such that:
 for every $\delta\leq\delta_{(a,b)}$ and every function 
 $\Delta_\delta\in H_\delta^m$, $m\ge3$,
 there exists a unique function $h_\delta\in {H_\delta^\infty}'$ satisfying
\begin{itemize}
\item[(1)]\label{Item:holomorphicsolution} $h_\delta\circ F_{1,0,0}  -h_\delta=\Delta_\delta$.
\item[(2)] ${\left\|h_\delta\right\|}_\infty\leq C{\left\|\Delta_\delta\right\|}_m$.
\item[(3)] For every $(z,\xi)\in S_{\delta}'\cap \{\Im{(\xi)}\geq 1\}$, we have $|h_\delta(z,\xi)|\leq \frac{C{\left\|\Delta_\delta\right\|}_m}{\sqrt{\Im{({\xi})}}}$.
\end{itemize}
In addition, there exists a positive constant $D_\theta$ only depending on $\theta\in]0, \frac{\pi}{2}]$
such that 
$$|h_\delta(z,\xi)|\leq\frac{D_\theta{\left\|\Delta_\delta\right\|}_m}{|\xi|^{m-2}}$$
for every $(z,\xi)\in S_{\delta}'\cap \{\theta-\varpi\leq\arg{\xi}\leq \pi-\theta-\varpi\}$.
\end{thm}

As mentioned before, we will postpone the proof of Theorem \ref{TH:lin} to  subsection \ref{SS:sollin}. 
Condition (3) is needed for the unicity, and after to produce a norm with a unique fixed point 
when solving the functional equation.
For the time being, we detail how it provides a section $\Psi_{1}$ of ${\mathscr G}^1(I_1)$ of the form $(z+g, \xi +h)$ such that the pair $(g,h)$ admits $(\hat g, \hat h)$ as asymptotic expansion and satisfies in addition the equations (\ref{E:lin1an}) and (\ref{E:lin2an}). In other words, we are going to exhibit a transversely sectorial conjugacy map between $F$ and $F_{1,0,0}$:
$$F\circ \Psi_1=\Psi_1\circ F_{1,0,0}.$$
   
\subsection{Solving the functional equation} \label{SS:solvefunctional}
Notations as in Theorem \ref{TH:lin}.  It is worth mentioning that the strategy developped here as well as the resolution of the linearized/homological equation in the forthcoming Section \ref{SS:sollin} owes a lot to \cite{VF}.

 As before, $a,b$ are fixed, $\delta_{(a,b)}>0$ is small enough  and may be adjusted from line to line in order to  guarantee the validity of the estimates below.  We will denote by $\delta$  any positive number such that $0<\delta\leq\delta_{(a,b)}$.
We will omit for a while the subscript $(a,b)$. Let us introduce two Banach spaces.

First, let $m\geq 3$, and consider $H_\delta^{m,m}$ the subspace of those 
$(D_1,D_2)\in{\mathcal H}(S_{\delta})\times{\mathcal H}(S_{\delta}) $ defined by $N_m(D_1,D_2)<\infty$ where 
$$N_m(D_1,D_2):=\|D_1\|_m + \|D_2\|_m$$
(recall that ${\left\| D\right\|}_m:=\sup_{(z,\xi)\in S_{\delta}}|D(z,\xi)|{{|\xi|}^m}$).

On the other hand, let ${H_\delta^{\infty,\infty}}'$ be the subspace of those $(h_1,h_2)\in{\mathcal H}(S_{\delta}')\times{\mathcal H}(S_{\delta}')$ defined by $N_\infty' (h,\tilde{g})<\infty$ where 
\[N_\infty' (h_1,h_2):=\Vert h_1\Vert_\infty'+\Vert h_2\Vert_\infty'$$
$$\text{with}\ \ \ 
\Vert h\Vert_\infty':=\underbrace{\sup_{(z,\xi)\in S_{\delta}'}|h(z,\xi)|}_{\Vert h\Vert_\infty} +\sup_{(z,\xi)\in S_{\delta}'\cap \mathfrak{I}{(\xi)}\geq 1}|h(z,\xi)|\sqrt{|\mathfrak{I}(\xi)}|.\]
Note that both normed spaces are Banach spaces, 
and from Theorem \ref{TH:lin}, one inherits a continuous linear map  between them: 
$$\mathcal L\ :\ (H_\delta^{m,m}, N_m)\to({H_\delta^{\infty,\infty}}',N_\infty')\ ;\ \overrightarrow{D}=(D_1,D_2)\mapsto \overrightarrow{h}=(h_1,h_2)$$ 
defined by solving
\[ h_i\circ F_{1,0,0}-h_i=D_i\ \ \ \text{for}\ \ \  i=1,2.\]
To be more precise, for every $\delta$ small enough, 
and every  $\overrightarrow{D}\in H_{\delta}^{m,m}$, one has 
\[N_\infty'({\mathcal L} (\overrightarrow{D}))\leq C\cdot N_m(\overrightarrow{D})\]
with the positive constant $C$ given by Theorem \ref{TH:lin}. 

We now define a non linear continuous map in the other way.
Let us come back to the expression of the transformation $F=F_{1,0,0}+(\Delta_1,\Delta_2)$ defining the formally equivalent neighborhood $(U,C)$ as explicited in subsection \ref{SS:overview}.
One can assume $\Delta_i(z,\xi)=O({\xi}^{-N})$ for a fixed arbitrary integer $N\ge4$. 
Recall that the $\Delta_i$'s are analytic on a neighborhood of $\{\xi=\infty\}$ and consequently are well defined as an element of $H_\delta^m$ whose ${\left\|\ \right\|}_m$ norm tends to zero when $\delta$ goes to zero. 
For every $M>0$, set us denote by 
$$H_\delta^{m,m}(M)\subset H_\delta^{m,m}\ \ \ \text{and}\ \ \  {H_\delta^{\infty, \infty}}'(M)\subset{H_\delta^{\infty, \infty}}$$ 
the respective balls of radius $M$.
Then, for $\delta$ small enough, we a have a   well defined map 
\[\mathcal {R}\ :\  {H_{\delta}^{\infty, \infty}}^{'}(1) \to {H_{\delta}^{m, m}} \ ;\ 
\overrightarrow{h}=(h_1,h_2)\mapsto\overrightarrow{D}=(D_1,D_2)\]
$$\text{where}\ \ \ \left\{\begin{matrix}D_1(z,\xi)&=&(1+h_2(z,\xi))\widetilde{\Delta}_1\big(z(1+ h_2(z,\xi)),\xi+h_1(z,\xi)\big)\\
D_2(z,\xi)&=&\hfill\Delta_2\big(z(1+ h_2(z,\xi)),\xi +h_1(z,\xi)\big)
\end{matrix}\right.$$
Indeed, if $N_\infty' (h_1,h_2)\le1$, then in particular $\Vert h_i\Vert_\infty\le1$, $i=1,2$, and therefore 
$\widetilde{\Delta}_1,\Delta_2$ are holomorphic at $(z(1+ h_2(z,\xi)),\xi+h_1(z,\xi))$ whenever $(z,\xi)\in S_\delta$.
Because $\Delta_i=O(\xi^{-N})$, note also that the image of ${H_{\delta}^{\infty, \infty}}^{'}(1)$ by  $\mathcal R$ 
lies in  ${H_{\delta}^{m, m}}(R_\delta)$ where $\lim\limits_{\delta\to 0}R_\delta=0$. 
The proof of the following is straighforward:

\begin{lemma}
Let $\varepsilon >0$.
Then,  for $\delta$ small enough,
one has
\[  \forall \overrightarrow h, \overrightarrow g \in {H_{\delta}^{\infty,\infty}}' (1)\  \ \ \Rightarrow\ \ \ N_m(\mathcal R ( \overrightarrow h)- \mathcal R ( \overrightarrow g))\leq \varepsilon\cdot  N_\infty^{'} ( \overrightarrow h- \overrightarrow g).\]
In particular, $\mathcal R$ is continuous (Lipschitz). 
 \end{lemma}

Let $\varepsilon >0$ such that $\varepsilon C<1$.
Then, the composition ${\mathcal L}\circ \mathcal R$ induces a (non linear) contracting map  of the complete metric space 
${H_{\delta}^{\infty,\infty} }^{'}(1)$. The unique fixed point is a  solution to the functional equations (\ref{E:lin1an}), (\ref{E:lin2an}). This provides a solution of the original functional equations (\ref{E:lin1}) and (\ref{E:lin2}), taking into account the renormalization indicated in Remark \ref{R:renorm}. By uniqueness, the solution $\overrightarrow h_\delta$ attached to $\delta$ induces by restriction the solution attached to $\delta'$ for $\delta'\leq \delta$. 
  
  One can complete this picture by taking into account all the properties required in the statement of Theorem \ref{TH:lin}. This leads to the following list of properties of the solution exhibited above as a fixed point of a non linear operator.  We reintroduce the susbcript $(a,b)$ (with obvious notations) in order to recall that the choice of $\delta$ depends on a fixed arbitrary annulus in the $z$ variable:

  \begin{prop}\label{P:solfunctional}
 Notations as above.
 Let $\widetilde {\Delta}_1,\Delta_2=O({\xi}^{-N})$ two germs of holomorphic functions in the neighborhood of $\tilde{C}\subset {\C}^2$ with $N\geq 4$ (as defined from the conjugation equation introduced in Section \ref{SS:overview}). Let $m<N$. Let $0<a<b< +\infty$ such that $a<|q|b$. Then there exists $\delta_{(a,b)}>0$ such that for every $0<\delta\leq \delta_{(a,b)}$, the system of equations  (\ref{E:lin2an}), (\ref{E:lin1an})  admits a unique solution $(h_{\delta,a,b},{\tilde {g}}_{\delta,a,b})\in {H_{\delta,a,b}^{\infty,\infty} }^{'}(1)\times{H_{\delta,a,b}^{\infty,\infty} }^{'}(1)$. 
 Moreover,
 \begin{itemize}
 \item $(h_{\delta',a,b},{\tilde {g}}_{\delta',a,b})$ is the restriction of $(h_{\delta,a,b},{\tilde {g}}_{\delta,a,b})$ if $0<\delta'\leq \delta\leq\delta_{(a,b)}$.
 \item $\lim\limits_{\delta\to 0} N_\infty'(h_{\delta,a,b},{\tilde {g}}_{\delta,a,b})=0$.
 \item there exists a positive number $D=D(\theta)$ depending only $\theta\in]0, \frac{\pi}{2}]$ such that for every 
 $\delta\leq\delta_{(a,b)}$ and every $(z,\xi)\in S_{\delta,a,b}'\cap \{\theta-\varpi\leq\arg{\xi}\leq \pi-\theta-\varpi\}$, 
 one has
$$|h_{\delta,a,b}(z,\xi)|  \leq\frac{D}{|\xi|^{m-2}}\  \ \ \mbox{and}\ \ \ |{\tilde {g}}_{\delta,a,b}(z,\xi)| \leq\frac{D}{|\xi|^{m-2}}.$$
  \end{itemize}
\end{prop}

  \subsection{Asymptotic expansion}\label{SS:asymptoticexpansion}
 Notations as above.  We start by fixing $N\geq 4$ and $a,b$ as before. Let us denote by $(h_{a,b}, ,{\tilde {g}}_{a,b})$ the germ of sectorial solution induced by  $(h_{\delta,a,b},{\tilde {g}}_{\delta,a,b})$ by taking $\delta\to 0$. Let $a',b'$ be positive real numbers such that $a'\leq a<b\leq b'$. By Proposition \ref{P:solfunctional}, note that the unique solution of (\ref{E:lin2an}), (\ref{E:lin1an}) lying in ${H_{\delta,a',b'}^{\infty,\infty} }^{'}(1)\times{H_{\delta,a',b'}^{\infty,\infty} }^{'}(1)$
 induces by restriction the unique solution of the same functional equation in ${H_{\delta,a,b}^{\infty,\infty} }^{'}(1)\times{H_{\delta,a,b}^{\infty,\infty} }^{'}(1)$.   In particular,$(h_{\delta,a,b},{\tilde {g}}_{\delta,a,b})$ is the restriction of $(h_{\delta,a',b'},{\tilde {g}})$. Then, if one takes projective limit with respect to $a\to 0$, $b\to +\infty$ and exploits the last asymptotic estimate in the Proposition \ref{P:solfunctional}, one get a solution $(h,\tilde{g})$ well defined as a flat element of ${\mathcal A}^{m-3}(I_1)\times {\mathcal A}^{m-3}(I_1)$. 
  
  Now, consider  an integer $p>>N$ arbitrarily large. Let  $k$ be a positive integer and consider the truncation 
  (or $k$-jet) $J^k\hat{\Psi}$ of $\hat{\Psi}$ at order $k$: 
  $$J^k\hat{\Psi} (z,\xi)=\left(z+\sum_{n=1}^k b_n \xi^{-n}\ ,\ \xi+\sum_{n=1}^k a_n \xi^{-n}\right).$$  
If $k$ is large enough, then one has
 $$
 {(J^k{\hat{\Psi}})}^{-1}\circ F\circ J^k{\hat{\Psi}}\ (\xi,z)= \big(q z + \Delta_1^k(z,\xi),\xi - 1+ \Delta_2^k(z,\xi)\big) 
 $$
where $\Delta_i^k(z,\xi)=O(\frac{1}{\xi^p})$. One can apply Proposition \ref{P:solfunctional} to get existence and uniqueness of $h_\delta^k, \tilde{g}_\delta^k\in  {H_{\delta,a,b}^{\infty,\infty} }^{'}(1)$  ($\delta$ small enough) such that 
 \[F\circ J^k{\hat{\Psi}} \circ \Psi^k=J^k{\hat{\Psi}}\circ \Psi^k\circ F_0\]  
 with $\Psi^k=\mbox{Id} +( g_\delta^k,h_\delta^k)$, where $g_\delta^k (z,\xi)=z{\tilde g}_\delta^k(z,\xi)$. 
 As before, these solutions are in fact induced by  a flat element of ${\mathcal A}^{p-4}(I_1)\times {\mathcal A}^{p-4}(I_1)$.
  Set $\Psi=\mbox{Id} +( g,h)$ with $g=z\tilde{g}$ and recall that $F\circ \Psi=\Psi\circ F_{1,0,0}$.
  Invoking again uniqueness and restrictions considerations, one  obtains that $\Psi=J^k{\hat{\Psi}} \circ \Psi^k$. 
  Thus, 
  $$\Psi= \left(z+\sum_{n=1}^{\mbox{Inf}(\alpha, p-4)}a_n \xi^{-n}\ ,\ \xi+\sum_{n=1}^{\mbox{Inf}(\alpha, p-4)}b_n \xi^{-n}\right) + R_k$$ 
  where $R_k\in \mathcal{A}^{p-4}(I_1)$ is flat. As $k$ (hence $p$) can be chosen arbitrarily large, 
  we eventually get that $\Psi\in  {\mathscr G}^1(I_1)$ and admits $\hat{\Psi}$ as asymptotic expansion. 
  We have thus obtain the sought normalization $\Psi_1:=\Psi$ on the germ of sector  of opening $I_1$.

\subsection{ Solving the linearized equation}\label{SS:sollin}

The goal of this section is to prove Theorem \ref{TH:lin}.  

Consider the foliation defined by the level sets $ \{Y=c \}$ of $Y=z e^{2i\pi\tau \xi}$.  
For every complex number $c$, $0<|c|< \delta$, consider 
 $$S_{a,b,c}:=\{J=c\}\cap S_{a,b,\delta} =\{(z,\xi)=(c e^{-2i\pi\tau\xi},\xi)\ ;\ \xi\in \Sigma_{a,b,c}\}$$ 
 where  
 $$ \Sigma_{a,b,c}=\{\xi\in\C:(\log{{|c|}}-\log{b})\leq\Re{(2i\pi\tau\xi)}\leq (\log{{|c|}}-\log{a})\}.$$
 Note that $\bigcup_{0<|c|<\delta} S_{a,b,c}=S_{a,b,\delta}$ and the linearized equation has a simple form 
 restricted to these slices. 
 To simplify the presentation, we introduce the {\bf notation}
$$\zeta=2i\pi\tau \xi\ \ \ \text{and}\ \ \ \lambda:=-2i\pi\tau\ \ \ \text{with}\ \ \ \lambda=\lambda_1+i\lambda_2,\ \ \ \lambda_1,\lambda_2\in\mathbb R$$
and note that the real part $\lambda_1>0$. In particular, the linearized/homological equation 
\[h_\delta\circ F_{1,0,0}-h_\delta=\Delta_\delta\] 
can be then rewritten as 
\begin{equation}\label{E:HE1}
\varphi_c(\zeta+\lambda)-\varphi_c(\zeta)=\Delta_c(\zeta) 
\end{equation}
where $\varphi_c(\zeta)=h(c e^{-\zeta},-\frac{\zeta}{\lambda})$, 
and $\Delta_c(\zeta)=\Delta_\delta(c e^{-\zeta},-\frac{\zeta}{\lambda})$.
We are then led to solve the family of difference equations (\ref{E:HE1}) with respect to the parameter $c$ in the vertical strip 
$$\Sigma_{a,b,c}=\{\zeta\in\C:\log{{|c|}}-\log{b}\leq\Re{(\zeta)}\leq \log{{|c|}}-\log{a}\}$$ 
where we impose $\varphi_c$ to be holomorphic,  defined on the larger strip 
$$\Sigma_{|q|a,b,c}=\{\zeta\in\C:\log{{|c|}}-\log{b}\leq\Re{(\zeta)}\leq \log{{|c|}}-\log{a}+\lambda_1\}$$ 
and to depend analytically on the parameter $c$ in order to recover a holomorphic solution to (1) in Theorem \ref{TH:lin}.

\begin{figure}[htbp]
\begin{center}
\includegraphics[scale=0.5]{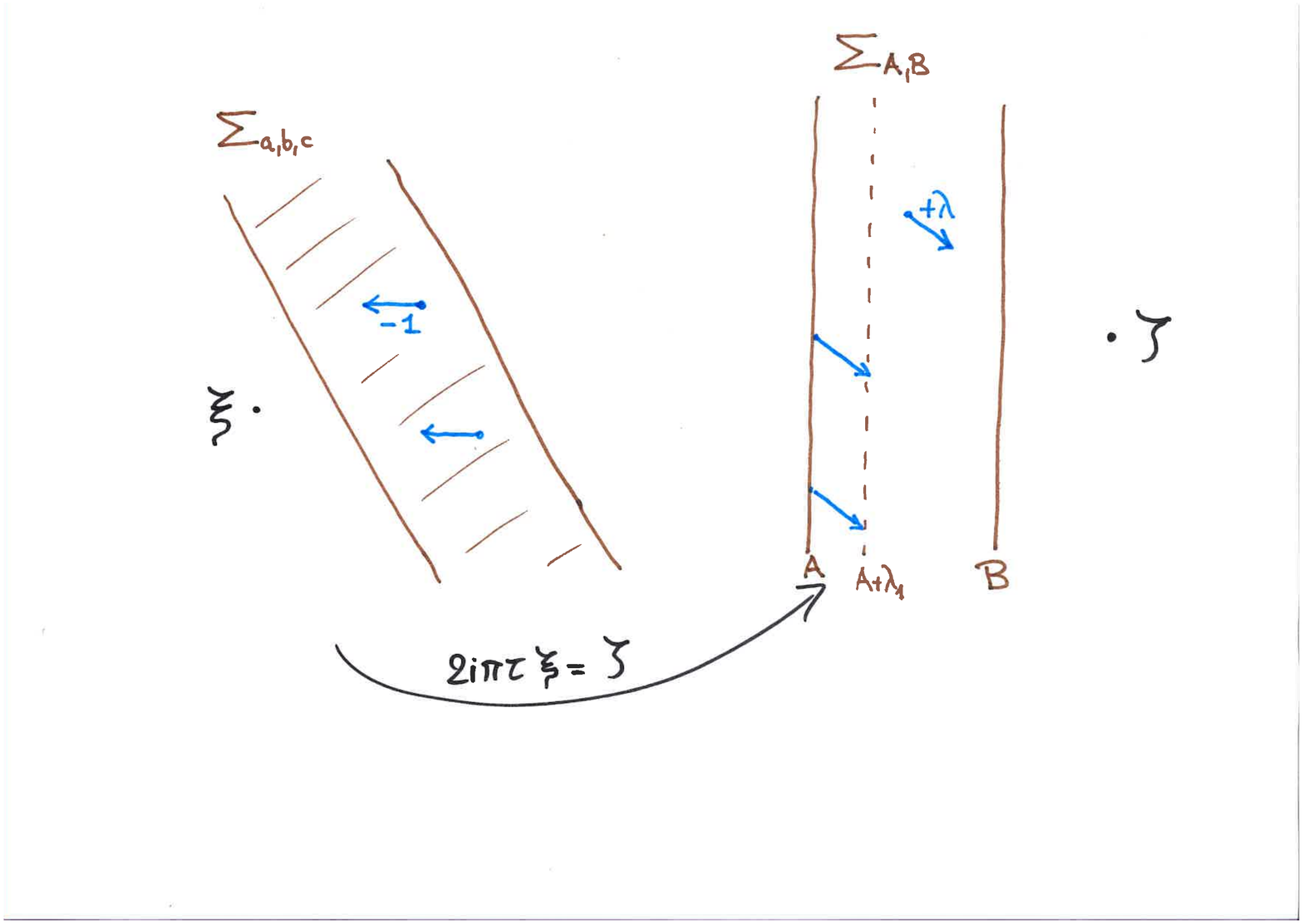}
\caption{{\bf Strips $\Sigma$ from $\xi$-plane to $\zeta$-plane}}
\label{pic:xitozeta}
\end{center}
\end{figure}

 \subsubsection{Resolution of a difference equation}\label{SS:resolHE}
 We now proceed to the contruction of $\varphi_c$. It is essentially a consequence of the following general result (with additional estimates).
 
 \begin{thm}\label{TH:soldiffeq}Let $A<A+\lambda_1<B\le1$, and 
let $\Delta$ be holomorphic on the strip 
$$\Sigma_{A,B}=\{\zeta\ ;\ A\leq\Re\zeta \leq B\}.$$
Suppose moreover that, for some $m\ge3$, we have:
$${\left\|\Delta\right\|}_m:={\mbox{sup}}_{\zeta\in S_{A,B}}|\Delta(\zeta)|{{|\zeta|}^m}<\infty.$$ 
Then, there exists a bounded holomorphic function $\varphi$ on $\Sigma_{A,B+\lambda_1}$ which solves
\begin{equation} \label{E:DE1}
 \varphi (\zeta+\lambda)-\varphi (\zeta)=\Delta (\zeta).
\end{equation}
Moreover $\varphi$ is unique modulo an additive constant.
\end{thm} 

\begin{proof}
First notice that, if $\varphi_1, \varphi_2$ are two \textit{bounded} holomorphic functions solving (\ref{E:DE1}), 
the difference $\varphi_1 -\varphi_2$ extends as a bounded $\lambda$-periodic entire function, hence constant. 
Uniqueness part of Theorem \ref{TH:soldiffeq} is therefore obvious. 

Concerning the existence part, let us first observe, by Cauchy formula, that
\begin{equation}\label{eq:CauchyF0}
\Delta(\zeta)=\underbrace{\frac{1}{2i\pi}\int_{L_+}\frac{\Delta (t)}{t-\zeta} dt}_{F_0^+(\zeta)} 
-\underbrace{\frac{1}{2i\pi}\int_{L_-} \frac{\Delta (t)}{t-\zeta} dt}_{F_0^-(\zeta)},\ \ \ A<\Re\zeta<B
\end{equation}
where $L_-=\{\Re{t}=A\}$ and $L_+=\{\Re{t}=B\}$ are both oriented from bottom to top. 
Since ${\left\|\Delta\right\|}_m<\infty$, we see that the two integrals 
are well defined, and holomorphic in $\zeta$.
Observe that $F_0^-$ and $F_0^+$ are respectively defined on the half-planes $A<\Re\zeta$ and $\Re\zeta<B$,
and can be extended to the boundary by continuity, likely as $\Delta$, by using equality (\ref{eq:CauchyF0}). 
Then we define, for $n\ge0$
\begin{equation} \label{E:defFnpm}\left\{\begin{matrix}
F_n^-(\zeta):=F_0^-(\zeta+n\lambda) & \text{holomorphic on}\ A-n\lambda_1\le\Re\zeta<\infty\\
F_n^+(\zeta):=F_0^+(\zeta-n\lambda) & \text{holomorphic on}\ -\infty<\Re\zeta\le B+n\lambda_1
\end{matrix}\right.\end{equation}
The solution $\varphi$ to (\ref{E:DE1}) is therefore given by the following series 
\begin{equation} \label{E:solseries}
\varphi(\zeta):=\sum_{n\ge1}F_n^+(\zeta)+\sum_{n\ge0}F_n^-(\zeta)\ 
=\ F_0^-(\zeta)+\sum_{n\ge1}\underbrace{F_n^+(\zeta)+F_n^-(\zeta)}_{F_n(\zeta)}
\end{equation}
that will be proved to converge uniformly on the large strip $\Sigma_{A,B+\lambda_1}$ in Lemma \ref{L:majDeltam}.
We can already check that it is indeed a solution:
$$\varphi(\zeta+\lambda)=\sum_{n\ge0}F_n^+(\zeta)+\sum_{n\ge1}F_n^-(\zeta)
=\varphi(\zeta)+\underbrace{F_0^+(\zeta)-F_0^-(\zeta)}_{\Delta(\zeta)}.$$
Therefore, Theorem \ref{TH:soldiffeq} is an immediate consequence of the following Lemma which actually provides further informations.
\end{proof}
 
\begin{lemma}\label{L:majDeltam}
For $\Delta$ like in Theorem \ref{TH:soldiffeq}, the series (\ref{E:solseries})
is well defined, holomorphic on $\Sigma_{A,B+\lambda_1}$ providing a solution of 
(\ref{E:DE1}).
Moreover, there exists a positive number $C=C(B-A)$ only depending on $B-A$ \footnote{In order to give an unambiguous statement, $\tau,m$ are fixed but $A,B,\Delta$ are allowed to vary provided they satisfy assumptions of Theorem \ref{TH:soldiffeq}.} such that
\begin{equation}\label{E:majDeltam}  
\sup_{\zeta\in \Sigma_{A,B+\lambda_1}}|\varphi (\zeta)|\leq C {\left\|\Delta\right\|}_m\int_{L^+}\frac{|dt|}{{|t|}^m}
\end{equation}
\end{lemma}
  
 \begin{proof}
 Set $I_m:= \int_{L_+}\frac{|dt|}{{|t|}^m}$.  We will also use repeatedly (and without mentioning it) that, 
 for $a< 0$,  $\int_{\Re{t}=a}\frac{|dt|}{{|t|}^m}=|a|^{1-m}\int_{\Re{t}=1}\frac{|dt|}{{|t|}^m}$. 
 In particular, since $A<B\le1$, we have
 $$\int_{L_-}\frac{|dt|}{{|t|}^m}<\int_{L_+}\frac{|dt|}{{|t|}^m}= I_m.$$
 If $A +\frac{\lambda_1}{2}\le \Re{u} $, one has 
 $$\left|\frac{1}{2i\pi}\int_{L_-}\frac{\Delta (t)}{t-u} dt\right|
 \leq \frac{1}{2\pi}\int_{L_-}\frac{|\Delta (t)| |t|^m}{|t-u|} \frac{|dt|}{|t|^m}
 \leq \frac{1}{2\pi}\frac{\|\Delta\|_m}{\lambda_1/2}\int_{L_-}\frac{|dt|}{|t|^m}
 \leq\frac{\left\|\Delta\right\|_m}{\pi\lambda_1}I_m. $$ 
 If $A\leq \Re{u}\leq A +\frac{\lambda_1}{2}$, Cauchy's formula yields the inequality  
 $$\left\vert\frac{1}{2i\pi}\int_{L_-}\frac{\Delta (t)}{t-u} dt\right\vert
 \leq |\Delta(u)|+\frac{{\left\|\Delta\right\|}_m I_m}{\pi\lambda_1}
 \leq {\left\|\Delta\right\|}_m+\frac{{\left\|\Delta\right\|}_m I_m}{\pi\lambda_1}
 \leq C_1{\left\|\Delta\right\|}_mI_m$$
 with
 $$C_1=\left(\frac{1}{I_m} +\frac{1}{\pi\lambda_1}\right).$$
Following the same principle, we get the inequality 
 $$\left|\frac{1}{2i\pi}\int_{L_+}\frac{\Delta (t)}{t-u} dt\right|\leq C_1{\left\|\Delta\right\|}_m I_m$$
 where $A\leq \Re{u}\leq B$.
 This eventually leads to the upper bound
 \begin{equation}\label{eq:firstupperboundFn}
 |F_n(\zeta)|\leq  2C_1{\left\|\Delta\right\|}_m I_m\ \ \ \forall n\geq 0,\ \ \ \forall\zeta\in \Sigma_{A,B +\lambda_1}.
 \end{equation}

Moreover, when $n>\frac{A-B}{\lambda_1}$, Cauchy's formula allows to write 
$$\int_{L_-} \frac{\Delta (t)}{t-\zeta-n\lambda}dt=\int_{L_+} \frac{\Delta (t)}{t-\zeta-n\lambda}dt.$$ 
Consequently, for these values of $n$ one has
\[F_n(\zeta)=\frac{1}{2i\pi}\int_{L_+} \frac{2(t-\zeta)\Delta (t)}{(t-\zeta-n\lambda)(t-\zeta +n\lambda)}dt\]
The key point will be to find a suitable upper bound for 
$$A(t,\zeta)=\sum_{n\geq n_0}\left|\frac{1}{(t-\zeta-n\lambda)(t-\zeta +n\lambda)}\right|$$ 
where $n_0>2\left(\frac{B-A}{\lambda_1}+1\right)$, $t\in L_+, \zeta\in \Sigma_{A,B+\lambda_1}$.

For $n\geq n_0$, Cauchy-Schwarz inequality gives 
\begin{equation}\label{E:schwin}
A(t,\zeta)\leq {\left(\sum_{n\geq n_0}\frac{1}{{(v-n\lambda_2)}^2 +{(\frac{n}{2}\lambda_1)}^2}\right)}^\frac{1}{2}{\left(\sum_{n\geq n_0}\frac{1}{{(v+n\lambda_2)}^2 +{(\frac{n}{2}\lambda_1)}^2}\right)}^\frac{1}{2}
\end{equation}
where $v=\Im(t-\zeta)$ and $\lambda=\lambda_1+i\lambda_2$. Here, we have used the fact that 
$$
\vert\Re(t-\zeta)\vert\le B-A+\lambda_1\le \frac{n}{2}\lambda_1\ \ \ \text{whenever}\ n\ge n_0.
$$
Therefore, 
\begin{equation}\label{E:schwinbis}
A(t,\zeta)\le\sum_{n\geq n_0}\frac{1}{{(\frac{n}{2}\lambda_1)}^2}
\le\left(\frac{2}{\lambda_1}\right)^2\sum_{n\geq n_0}\frac{1}{n^2}
\le\left(\frac{2}{\lambda_1}\right)^2\int_{n_0-1}^{+\infty}\frac{dt}{t^2}
\le\left(\frac{2}{\lambda_1}\right)^2.
\end{equation}
This, together with the fact that ${\left\|\Delta\right\|}_m<\infty$ for some $m\geq 3$, 
 proves immediately that the serie $\sum F_n$ converges uniformly on every compact of $\Sigma_{A,B+\lambda_1 }$. The function $\varphi$ is then well defined and holomorphic on $\Sigma_{A,B+\lambda_1 }$.

For $n\geq n_0$, consider the decomposition $F_n=G_n +H_n$ where 
$$G_n(\zeta)=\frac{1}{2i\pi}\int_{\{t\in L_+, |v|\leq 1\}} \frac{2(t-\zeta)\Delta (t)}{(t-\zeta-n\lambda)(t-\zeta +n\lambda)}dt$$
and
$$H_n(\zeta)=\frac{1}{2i\pi}\int_{\{t\in L_+, |v|>  1\}} \frac{2(t-\zeta)\Delta (t)}{(t-\zeta-n\lambda)(t-\zeta +n\lambda)}dt.$$ 
Consequently, when $\vert v\vert\le1$, we deduce that $\vert t-\zeta\vert\le\sqrt{1+(B-A)^2}$ so that,
by (\ref{E:schwinbis}), we find $\sum_{n\geq n_0}|G_n(\zeta)|\leq C_2{\left\|\Delta\right\|}_m I_m$ where 
$$C_2= \left(\frac{2}{\lambda_1}\right)^2\frac{\sqrt{1 +{(B-A)}^2}}{\pi}.$$   
In order to achieve the proof of Lemma \ref{L:majDeltam}, one needs a little bit more analysis to estimate 
$\sum_{n\geq n_0} |H_n(\zeta)|$ (and justify a posteriori the choice of $|v| >1$). This also relies on (\ref{E:schwin}),  
firstly noticing that the first factor in the left-hand-side can be rewritten as 
$${\left(\sum_{n\geq n_0}\frac{1}{{(v-n\lambda_2)}^2 +{(\frac{n}{2}\lambda_1)}^2}\right)}^\frac{1}{2}
=\frac{1}{c_1\vert v\vert}{\left(\sum_{n\geq n_0}\frac{1}{{(c_2\frac{n}{v}+c_3)}^2 +1}\right)}^\frac{1}{2}$$
with constants $c_1,c_2,c_3\in\C$ depending only on $\lambda_1,\lambda_2$, so that it can be bounded by
\begin{equation}\label{eq:LHSA}
\frac{1}{c_1\vert v\vert}{\left(\int_\R\frac{dt}{{(c_2\frac{t}{v}+c_3)}^2 +1}\right)}^\frac{1}{2}=\frac{\sqrt{\pi}}{c_1c_2\sqrt{\vert v\vert}}.
\end{equation}
By a similar computation, we get the same bound for the second factor of $A(t,\zeta)$:
$${\left(\sum_{n\geq n_0}\frac{1}{{(v+n\lambda_2)}^2 +{(\frac{n}{2}\lambda_1)}^2}\right)}^\frac{1}{2}
\le\frac{\sqrt{\pi}}{c_1c_2\sqrt{\vert v\vert}}.$$
Therefore, using that $\frac{\vert t-\zeta\vert}{\vert v\vert}=\sqrt{\frac{(B-A)^2}{\vert v\vert^2}+1}\le \sqrt{(B-A)^2+1}$
for $\vert v\vert>1$, we conclude that
$$\sum_{n\ge n_0}\vert H_n(\zeta)\vert\le \underbrace{\pi\frac{\sqrt{(B-A)^2+1}}{2(c_1c_2)^2}}_{C_3}{\left\|\Delta\right\|}_m I_m$$
for a constant $C_3$ depending only on $B-A$. By putting together the inequalities above involving $C_1,C_2,C_3$, 
one obtains that $\varphi$ satisfies the estimate (\ref{E:majDeltam}) of Lemma \ref{L:majDeltam}, 
thus proving Theorem \ref{TH:soldiffeq}.
\end{proof}

 \subsubsection{Choice of a canonical solution}\label{SS:additional}
 In this section, we are going to replace the solution $\varphi(\zeta)=F_0^-(\zeta)+\sum_{n\ge1}F_n(\zeta)$ 
 constructed in the proof of Theorem \ref{TH:soldiffeq} by another one $\psi(\zeta)=\varphi(\zeta)+m_0$
 where $m_0$ is a constant, so that $\psi\to0$ while $\Im(\zeta)\to-\infty$.
 As we shall see, the right constant is 
 $$m_0=\frac{1}{2\lambda}\int_{L_-}\Delta(t) dt=\frac{1}{2\lambda}\int_{L_+}\Delta(t) dt.$$ 
Indeed, we have:
 
 \begin{lemma}\label{L:imto0} 
 Keeping notations as in the proof of Theorem \ref{TH:soldiffeq}, set
 $$\psi(\zeta)=\underbrace{F_0^-(\zeta)+\sum_{n\ge1}F_n(\zeta)}_{\varphi(\zeta)}+\underbrace{\frac{1}{2\lambda}\int_{L_+}\Delta(t) dt}_{m_0}.$$
 Then, $\psi$ is the unique solution to the difference equation (\ref{E:DE1}) such that,
for every $\zeta\in \Sigma_{A,B+\lambda_1}$ satisfying 
 $\Im{(\frac{\zeta}{\lambda})} \le -1$ (i.e. $\Im(\xi)\ge 1$), one has
 \begin{equation}\label{E:decpartieim}
  |\psi (\zeta)|\leq \frac{F{\left\|\Delta\right\|}_m}{\sqrt{|\Im{(\frac{\zeta}{\lambda})|}}}.
  \end{equation}
  for a constant $C=C(B-A)$ depending only on $B-A$.
 
 Moreover,  for every $\theta\in ]0,\frac{\pi}{2}]$, there exists a  real number $D=D(\theta, B-A)>0$ 
  such that for  $\zeta\in K_\theta:= \Sigma_{A,B+\lambda_1}\cap\{\frac{\pi}{2} +\theta\leq \arg{\zeta}\leq \frac{3\pi}{2}-\theta\}$,
  we have
 \begin{equation}\label{E:decsector}
 |\psi(\zeta)|\leq{(\frac{A}{B})}^{m-2}\frac{D{\left\|\Delta\right\|}_m}{|\zeta|^{m-2}} .
 \end{equation}
\end{lemma}

The first condition (\ref{E:decpartieim}) insures that the solution $\psi\to0$ at least when $\Im{(\frac{\zeta}{\lambda})}\to-\infty$
(or equivalently $\Im{({\zeta})}\to-\infty$ in the strip $S_{A,B+\lambda_1}$). The second condition is used
to prove existence of asymptotic expansions of the sectorial normalization.

 \begin{proof}
 Let $\zeta\in \Sigma_{A,B+\lambda_1}$ such that $\Im{(\frac{\zeta}{\lambda})}<0$, so that $\zeta\notin\lambda\Z$. 
 Write 
\[F_n(\zeta)=
\frac{1}{2i\pi}\left[
\left(\int_{L_+}\frac{\Delta (t)}{t-\zeta +n\lambda} dt + \int_{L_+}\frac{\Delta (t)}{\zeta -n\lambda} dt\right)
\right.\] 
\[\left. 
+\left(\int_{L_-} \frac{\Delta (t)}{t-\zeta -n\lambda} dt +\int_{L_-}\frac{\Delta (t)}{\zeta +n\lambda} dt\right)
-\int_{L_+}\frac{\Delta (t)}{\zeta -n\lambda} dt-\int_{L_-}\frac{\Delta (t)}{\zeta +n\lambda} dt 
\right]\]
Now using that $\int_{L_-}\Delta(t)dt= \int_{L_+}\Delta(t)dt$ and 
$$
\sum_{n\in\Z}\frac{1}{\zeta+n\lambda}=\frac{\pi}{\lambda}\cot{(\pi\frac{\zeta}{\lambda})},
$$
and summing-up the previous equalities, one obtains
$$
\varphi(\zeta)=\frac{1}{2i\pi}\sum_{n\geq 1}\int_{L_+}\frac{t\Delta (t)}{(\zeta-n\lambda)(t-\zeta +n\lambda)} dt
+\frac{1}{2i\pi}\sum_{n\geq 0}\int_{L_-}\frac{t\Delta (t)}{(\zeta+n\lambda)(t-\zeta -n\lambda)} dt 
$$
$$
-\frac{\pi}{2i\pi\lambda}\cot{(\pi\frac{\zeta}{\lambda})}\int_{L_+}\Delta(t) dt
$$
which makes sense whenever $\Im(\frac{\zeta}{\lambda})\not=0$.

Now, we start proceeding mimicking the proof of Lemma \ref{L:majDeltam} to bound the terms in $\varphi$
assuming that $\Im{(\frac{\zeta}{\lambda})} \le -1$. For each $n$, we can majorate
$$
\left|\frac{1}{2i\pi}\int_{L_-}\frac{t\Delta (t)}{(\zeta+n\lambda)(t-\zeta -n\lambda)} dt\right|
\le \frac{1}{2\pi|\lambda|}\int_{L_-}\frac{\Vert\Delta\Vert_m}{|\frac{\zeta}{\lambda}+n|\, |t-\zeta -n\lambda|} \frac{|dt|}{|t^{m-1}|}
$$
$$
\le \frac{1}{2\pi|\lambda|}\int_{L_-}\frac{\Vert\Delta\Vert_m}{|\Im(\frac{\zeta}{\lambda})|(\frac{n}{2}\lambda_1)} \frac{|dt|}{|t^{m-1}|}
\le \frac{\Vert\Delta\Vert_m I_{m-1}}{\pi n \lambda_1^2 |\Im(\frac{\zeta}{\lambda})|}
$$
(we have used that $|t-\zeta -n\lambda|\ge |\Re(t-\zeta -n\lambda)|=|\Re(t-\zeta) -n\lambda_1|\ge\frac{n}{2}\lambda_1$).
This allow us to bound terms for small values of $n$. 

Now, for large $n$, consider a positive integer 
$n_0\geq 2(\frac{B-A}{\lambda_1}+1)$, and define 
$$B(t,\zeta)=\sum_{n\geq n_0}\left|\frac{1} {(\zeta+n\lambda)(t-\zeta -n\lambda)}\right|\ \ \ \text{for}\ \ \ 
(t,\zeta)\in L^-\times \Sigma_{A,B+\lambda_1}.$$
 Here, proceeding as for proving (\ref{E:schwin}) and (\ref{E:schwinbis}), Cauchy-Schwarz inequality gives us
 $$
B(t,\zeta)\leq \left(\sum_{n\geq n_0}\frac{1} {|\zeta+n\lambda|^2}\right)^{1/2}
\left(\sum_{n\geq n_0}\frac{1} {|t-\zeta -n\lambda|^2}\right)^{1/2}.
$$
The left-hand-side is bounded by 
$$
\sum_{n\geq n_0} \frac{1}{{|\zeta+n\lambda|}^2}
\leq \frac{1}{{|\lambda|}^2}\int_\R\frac{dt}{(t+\Re(\frac{\zeta}{\lambda}))^2+\vert\Im(\frac{\zeta}{\lambda})\vert^2}
= \frac{\pi}{{|\lambda|}^2 |\Im{(\frac{\zeta}{\lambda})}|}
$$ 
and right-hand-side by (\ref{eq:LHSA}) so that we eventually get the bound
$$
B(t,\zeta)\leq \frac{C_1}{\sqrt{|\Im{(\frac{\zeta}{\lambda}})v|}}
$$
where $v=\Im{(t-\xi)}$ and $C_1$ is some positive constant. Then, as in the proof of Lemma \ref{L:majDeltam},
we end-up considering separately the cases $|v|\leq 1$ and $|v|>1$ when integrating along $L_-$
and establish in a similar way the following bound:
for every 
$\zeta\in \Sigma_{A,B+\lambda_1}$ satisfying $\Im{(\frac{\zeta}{\lambda})} \le -1$, we get:
$$
\left|\frac{1}{2i\pi}\sum_{n\geq 0}\int_{L-}\frac{t\Delta (t)}{(\zeta+n\lambda)(t-\zeta -n\lambda)} dt\right|
\leq  \frac{C_2{\left\|\Delta\right\|}_m}{\sqrt{|\Im{(\frac{\zeta}{\lambda})|}}}
$$
where $C_2>0$ only depends on $B-A$. We skip the details. Similarly, one gets
$$
\left|\frac{1}{2i\pi}\sum_{n\geq 1}\int_{L_+}\frac{t\Delta (t)}{(\zeta-n\lambda)(t-\zeta +n\lambda)} dt\right|
\leq \frac{C_3{\left\|\Delta\right\|}_m}{\sqrt{|\Im{(\frac{\zeta}{\lambda})|}}}
$$
where $C_3>0$ only depends on $B-A$ as well. 
Finally, for $\Im{(\frac{\zeta}{\lambda})}\le -1$, a straightforward calculation gives 
$$
\left|\frac{\pi}{2i\pi\lambda}\cot{(\pi\frac{\zeta}{\lambda})}-\frac{1}{2\lambda}\right|\leq 
\frac{e^{2\pi\Im{(\frac{\zeta}{\lambda})}}}{|\lambda|(1-e^{-2\pi})}\leq \frac{C_4}{\sqrt{|\Im{(\frac{\zeta}{\lambda})|}}}
$$
for some constant $C_4>0$. 
These different estimates prove (\ref{E:decpartieim}).

Now, let us establish the upper-bound (\ref{E:decsector}). For this end, observe that, on $K_\theta$, we have
 $$|B+\lambda_1|\leq |\zeta|\leq \frac{|A|}{\sin \theta}$$
so that we have in particular:
$$|\frac{t}{\zeta}|\geq \frac{B\sin{\theta}}{A}\ \ \ \text{for every}\ \ \ t\in L^+.$$
From Lemma \ref{L:majDeltam}, one deduces that 
 $$\sup_{\zeta\in K_\theta}|\varphi (\zeta)|\leq \frac{C{\left\|\Delta\right\|}_m}{{|\zeta|}^{m-2}}{\left(\frac{A}{B\sin{\theta}}\right)}^{m-2}\int_{L^+}\frac{|dt|}{{|t|}^2} $$ 
 where $C=C(B-A)$ is the constant appearing in \textit{loc.cit}. Moreover one gets trivially that    
 
 $$|m_0|\leq  {\left\|\Delta\right\|}_m\int_{L^+}\frac{|dt|}{{|t|}^m}$$
Consequently $D:= \frac{C+1}{{(\sin{\theta})}^{m-2}}\int_{\Re{t}=1}\frac{|dt|}{{|t|}^2}$ satisfies  (\ref{E:decsector}).  \end{proof}

\subsubsection{Version with parameters and end of the proof of Theorem \ref{TH:lin}}
We resume to notations introduced in Section \ref{SS:linarized} and the beginning of Section \ref{SS:sollin}.
We have fixed $\tau$ with $\Im\tau>0$ and set 
$$\lambda=-2i\pi\tau=\underbrace{\lambda_1}_{>0}+i\lambda_2.$$
Consider $0<a<b$ such that $a<|q| b$ or equivalently  $\log{a}+\lambda_1<\log{b}$, and let $\delta_{(a,b)}>0$ small enough.
For every positive number $\delta\leq\delta_{(a,b)}$ and  complex number $c$ such that $0<|c|\leq\delta $, 
consider the vertical strips defined in the complex line by
$$
\Sigma_{a,b,c}=\{\zeta\in\C:\underbrace{\log{|c|}-\log{b}}_{A_c}\leq\Re{(\zeta)}\leq \underbrace{\log{|c|}-\log{a}}_{B_c}\}.
$$
Consider also the subset of ${\C}^2$  
$${\mathscr S}_{a,b,\delta}=\bigcup_{|c|\leq\delta} \{c\}\times \Sigma_{a,b,c}.$$
It is worth mentioning that $B_c-A_c=\log{b}-\log{a}>\Re{(2i\pi\tau)}$ does not depend on $c$. 
By a suitable choice of $\delta_{(a,b)}$, one can moreover assume that $A_c, B_c\leq -1$ and $\frac{A_c}{B_c}\leq 2$. 
Let $\Delta(c,\zeta)$ be a holomorphic function defined on   ${\mathscr S}_{a,b,\delta}$   and assume in addition that 
$${\left\|\Delta\right\|}_m:=\sup_{{\mathscr S}_{a,b,\delta}}|\Delta(c,\zeta)|{{|\zeta|}^m}<\infty$$ 
for some $m\geq 3$. From the results collected in subsections \ref{SS:resolHE}, \ref{SS:additional}, one promptly obtains the following statement:
  
  \begin{prop}\label{P:sol1param}
  Let $\delta$ small enough and $\Delta\in{\mathcal H}({\mathscr S}_{a,b,\delta})$ with ${\left\|\Delta\right\|}_m<\infty$, $m\ge3$. 
  Then, there exists a unique function $\varphi\in{\mathcal H}({\mathscr S}_{|q|a,b,\delta})$ with the following properties
  \begin{enumerate}
  \item For every $(c,\zeta)\in {\mathscr S}_{a,b,\delta}$, one has 
$\varphi(c,\zeta+\lambda)-\varphi(c,\zeta)=\Delta(c,\zeta)$.
\item There exists a positive number $C=C(a,b)$  such that for all $0<\delta\leq \delta_{(a,b)}$,
    $$\sup_{{\mathscr S}_{|q|a,b,\delta}}|\varphi(c,\zeta)|
    \leq C{\left\|\Delta\right\|}_m$$
$$\sup_{{\mathscr S}_{|q|a,b,\delta}\cap\{\Im(\frac{\zeta}{\lambda})\le-1\}}
|\varphi(c,\zeta)|  \leq\frac{C{\left\|\Delta\right\|}_m}{\sqrt{|\Im{(\frac{\zeta}{\lambda})|}}}.$$
 \end{enumerate}
 \end{prop}
 
 \begin{proof}
 Define $\varphi_\delta$ by the formula (compare Lemma \ref{L:imto0})
$$\varphi_\delta(c,\zeta):=\underbrace{\frac{1}{2i\pi}\int_{L_-^c} \frac{\Delta_\delta (c,t)}{t-\zeta} dt}_{F_0^-(c,\zeta)}$$
$$+\sum_{n=1}^\infty \underbrace{\frac{1}{2i\pi}(\int_{L_+^c}\frac{\Delta_\delta (c,t)}{t-\zeta +n\lambda} dt +\int_{L_-^c} \frac{\Delta_\delta (c,t)}{t-\zeta -n\lambda} dt)}_{F_n(c,\zeta)} 
+\underbrace{\frac{1}{2\lambda}\int_{L_+^c}\Delta_\delta(c,t) dt}_{m_0(c)}$$ 
setting $L_-^c=\{\Re{t}=A_c\}$, $L_+^c=\{\Re{t}=B_c\}$. 
From this integral formula, it is clear that $F_0^-$ and $F_n$ are well defined as  holomorphic functions on 
${\mathscr S}_{|q|a,b,\delta}$ and that $m_0$ depends analytically on $c$. Moreover, one can easily verifies 
(as in subsection \ref{SS:resolHE}) that the series $\sum F_n$ converges uniformly on every compact subset 
of ${\mathscr S}_{|q|a,b,\delta}$. Thus $\varphi\in\mathcal H({\mathscr S}_{|q|a,b,\delta})$ and fullfills the properties 
stated in the Proposition as a direct application of the construction performed in subsections \ref{SS:resolHE} and \ref{SS:additional}. 
 \end{proof}
 
 Then, the proof of Theorem \ref{TH:lin} immediately follows when translating this existence and uniqueness result into the original variable $(z,\xi)$, with $c=ze^{-2i\pi\tau\xi}$ and $\zeta=2i\pi\tau\xi$ (see Section \ref{SS:sollin})  together with the upper bound (\ref{E:decsector}) of Lemma \ref{L:imto0}.\qed
\vskip0.5cm

Recall that once we have solutions to the linearized equation as achieved above, we are in position to obtain a solution to the general functional equations (\ref{E:lin1an}) and (\ref{E:lin2an}) by a standard fixed point Theorem as developped in subsection \ref{SS:solvefunctional}. \textbf{This eventually finishes (at least on sectors of opening $I_1$) the proof of Lemma \ref{LEM:sectorialnormalization}} (and more precisely Lemma \ref{lem:sectorialRappel}). One can construct others normalizing  conjugacy biholomorphisms on the remaining sectors, namely $I_2,I_3,I_4$ without further real complications. This is explained below.   

  \subsection{Construction of other sectorial normalizations} \label{SS:othersectors}
  
  \subsubsection{On $I_3$} It is just a slight modification of the previous construction for $I_1$. The starting domain consists again in the resolution of the linearized equation (\ref{E:lin1}) on the corresponding domain. In this situation, it is relevant to deal with the levels $Y=c$ of the first integral $Y=ze^{2i\pi\tau\xi}$ for $|c|>>0$. One adapts the notation of Section \ref{SS:linarized} by defining   $S_{a,b, \delta}=\{(z,\xi)\in {\C}^2|\ |Y(z,\xi)|>\delta\ \mbox{and}\ z\in {\mathcal C}_{a,b}\}$  ($\delta>>0$). Here again, this amounts to solve an analytic family of difference equations where the solutions satisfy some estimates which eventually leads by the same fixed point consideration to a normalizing conjugation map. This can be carried out following \textit{verbatim} the same method.
  
 \subsubsection{On $I_2$ and $I_4$}  Here, we can use the $\mathrm{SL}_2(\Z)$ action described in subsection \ref{ss:SL2intro} to deduce existence of sectorial normalization on remaining sectors.
Indeed, if we consider the cyclic covering determined by $M=\begin{pmatrix}
0 & -1 \\ 
1 & 0
\end{pmatrix}$
(see subsection \ref{ss:SL2intro}) then we are led to an alternate presentation of $(U,C)$.
Our model $(U_{1,0,0},C)$ is now viewed as the quotient of $(\C_{z'}^*\times\C_{\xi'}, \{ {\xi}'=\infty\})$ 
by the semi-hyperbolic transformation 
$$F_{1,0,0}'(z',\xi')=({q}' z',\xi' -1 )\ \ \ \text{where}\ \ \ 
\tau'=\frac{-1}{\tau},\ \ \ q'=e^{-\frac{2i\pi}{\tau}}$$ 
$$\text{and}\ \ \ z'=z^{\frac{1}{\tau}},\ \ \ {\xi}'=\tau\xi+\frac{\log(z)}{2i\pi}.$$
In these coordinates $(U,C)$ is isomorphic to a quotient of the form
$({\C}_{z'}^*\times{\C}_{{\xi}'}\ \{{\xi}'=\infty\})$ by a biholomorphism $F'$ such that $F'(z',\infty)=(z',\infty)$.
 We can then apply the results of previous sections and obtain an analytic sectorial conjugation 
$$F'\circ {\Psi}'={\Psi}'\circ F_{1,0,0}'$$
defined on sectorial domains of the shape 
$$\arg(\xi')\in I_1'=]\varpi-\pi,\varpi[\ \ \ \text{and}\ \ \ \arg(\xi')\in I_3'=I_1'+\pi,$$
$\varpi'=\arg(\tau')=\pi-\varpi$.
Moreover, $\Psi'$ admits an asymptotic expansion $\hat{\Psi}'$ solving the same conjugacy equation.
Finally, going back the $(z,\xi)$ coordinates, we get that $\Psi'$ provides the sought sectorial conjugation
on 
$$\arg(\xi)\in I_1'-\varpi= ]-\pi,0[=I_2\ \ \ \text{and}\ \ \ \arg(\xi)\in I_3'-\varpi=I_4$$
taking into account that $\arg(\xi')\sim\varpi+\arg(\xi)$ asymptotically.
Maybe composing $\Psi'$ by an automorphism of our model $(U_{1,0,0},C)$
(see subsection \ref{sec:sheavesautomorphisms}), we can assume that the asymptotic expansion
fits with that of previous sectorial normalizations:
$\hat{\Psi}'=\hat{\Psi}$.

 \section{Generalization to the case of trivial normal bundle}\label{S:Generalization}
 
 \subsection{General formal classification:  recollections and symmetries}\label{SS:formalclassification}
 We first recall here the results obtained in \cite{LTT} (and partially exposed in \ref{SS:formalclass}) under a slightly more synthesized form including the case of torsion normal bundle and finite Ueda type.  As already mentioned in Section \ref{SS:formalclass}, we do not address the case where $C$ fits into a formal fibration (corresponding to infinite Ueda type). According to a result due to Ueda, this fibration is indeed analytic and we fall into the case which satisfies the formal principal: there is no differences between analytic and formal classification and this latter is very easy to describe, see for instance \cite[Section 5.1]{LTT}.

  Let $m$ be the torsion order of the normal bundle $N_C$ and let $k>0$ be the Ueda type. Recall that this latter is necessary a multiple of $m$. The linear monodromy of the corresponding unitary connection along the loops $1$ and $\tau$ is respectively determined by two roots of unity $a_1$, $a_\tau$ of respective orders $m_1$ and $m_\tau$ such that $lcm(m_1,m_\tau)=m$. The triple $(a_1,a_\tau,k)$ is obviously a formal invariant of the neighborhood. A complete set of invariants is indeed provided in \cite[Section 5.2]{LTT}, from which we borrow and adapt the notation. In particular, there exists on each normal form described below a pencil of regular foliations $(\F_t)_{t\in {\mathbb P}^1}$ fullfilling one of the following properties:
  \begin{enumerate}
  	\item  $C$ is $\F_t$ invariant for $t\in\C$ \label{E:pencilproperty1},
  	\item $\F_{\infty}$ is either tangent to $C$, either totally transverse to $C$ \label{E:pencilproperty2}.
  \end{enumerate}

Conversely every formal foliation satisfying one of these two properties fits into this pencil and is in particular convergent. Moreover, this family of foliations  can be defined by a pencil of closed meromorphic forms $\omega_t=\omega_0+t\omega_\infty$ whose expression is recalled below according the value of the parameters characterizing the normal forms that we proceed to describe now.
 
Set $\varphi_{k,\nu}=\exp{(\frac{y^{k+1}}{1+\nu y^k}{\partial_y})}$,
and given $P(z)=\sum_{i=0}^{k'-1}\lambda_i z^i$  a polynomial, $k=mk'$,
define
\begin{equation}\label{eq:PrincPartOneForm}
\omega_P:=P(\frac{1}{y^m})\frac{dy}{y},\ \ \ \end{equation}
$$\ \ \ \text{and}\ \ \ g_{k,\nu,P}(y):=\tau\int_0^y \lbrack{(a_\tau\varphi_{k,\nu})}^*\omega_P-\omega_P\rbrack=\tau\int_0^y \lbrack{\varphi_{k,\nu}}^*\omega_P-\omega_P\rbrack.$$

The group $\Z_{k'}$ of $k'^{\text{th}}$ roots of unity acts on the set of polynomials $P$ as follows:
\begin{equation}\label{eq:RootUnityAction}
(\mu, P(X))\mapsto P(\mu X)
\end{equation}

\begin{thm} \cite{LTT}\label{thm:NormalFormNeigh0<p<k}
Notations as above.
There exist $\nu\in\C$ and $P\in \C[z]$ of degree at most $k'-1$, unique up to the $\Z_{k'}$-action
(\ref{eq:RootUnityAction}), such that $(U,C)$ is formally equivalent \footnote{As noticed in \cite{LTT}, 
this result is unchanged if one allows formal conjugacy maps inducing translation on $C$.} 
to the quotient $U_{k,\nu,P,a_1,a_\tau}$ of
 $(\C_x\times \C_y,\{y=0\})$ by the group generated by
\begin{equation}\label{eq:NormalFormNeighGroup0<p<k}
\left\{\begin{array}{ccc} \phi_1(x,y)&=& (x+1\ ,\ a_1 y) \\ 
\phi_\tau(x,y)&=&(x+\tau+g_{k,\nu,P}(y)\ ,\ a_\tau \varphi_{k,\nu}(y))\end{array}\right. 
\end{equation}
The pencil $\omega_t=\omega_0+t\omega_\infty$ of closed $1$-forms
 is generated by
\begin{equation}\label{eq:NormalFormNeigh1Form0<p<k}
-\omega_0=\frac{dy}{y^{k+1}}+\nu\frac{dy}{y}\ \ \ \text{and}\ \ \ \omega_\infty=\frac{dx}{\tau}-\omega_P.
\end{equation}
\end{thm}
Actually (see \cite{LTT}), the case $P=0$ (i.e $\omega_P=0$) corresponds exactly to the existence 
of a (necessarily unique) transverse fibration (Property (\ref{E:pencilproperty2}) above), given at the level of the formal normal forms described above by $\omega_\infty=\frac{dx}{\tau}$. 
If in addition, $k=1$, $\nu=0$ (and necessarily $m=1$) one recovers Serre's example. The case where $P=\mu$ is the constant polynomial, is covered by \cite[Theorem 5.3]{LTT} whereas $\deg(P)>0$ is covered by \cite[Theorem 5.4]{LTT}.

When $m=1$, that is $N_C$ analytically trivial, one will denote (as in \ref{SS:formalclass}) the corresponding neighborhood by $U_{k,\nu,P}$.   One then observes that $\phi_\tau$ is the flow at time $1$ of the holomorphic vector field
$$v_{k,\nu,P}=v_0+ v_\infty $$ where $$v_0=\frac{y^{k+1}}{1+\nu y^k}{\partial_y} +\tau P(\frac{1}{y} )\frac{y^{k}}{1+\nu y^k}{\partial_x}.$$
and
$$v_\infty=\tau{\partial_x}$$ are the dual vector fields of the pair $(-\omega_0,\omega_\infty)$. For general neighborhoods $U_{k,\nu,P,a_1,a_\tau}$, one can notice that  $\phi_\tau=A_\tau\circ \exp{v_{k,\nu,P}}$ where $A_\tau(x,y)=(x,a_\tau y)$. 

 In the non torsion case ($m=1$) and in the coordinates $z=e^{2i\pi x}$ and $\xi=1/y$, the corresponding presentations are (in accordance with \ref{SS:formalclass})  
$$(U_{k,\nu,P}, C):= ({\C}_z^*\times\overline{{\C}_\xi} , \{\xi=\infty\})/\langle F_{k,\nu,P }\rangle$$ 
where $F_{k,\nu,P}= \exp{v_{k,\nu,P}}$  and ${v_{k,\nu,P}}=v_0+ v_\infty$ with $v_0=  \frac{1}{{\xi}^k +\nu} (-\xi{\partial_\xi} +2i\pi\tau zP(\xi){\partial_z})$, $v_\infty=2i\pi \tau z{\partial_z}$.  This can be directly borrowed from the presentation given in the $(x,y)$ variable. 

Concerning the structure of the automorphism group $\mathrm{Aut}(U_{k,\nu,P}, C)$ of formal automorphisms inducing translations on $C$, we have the following characterization which generalizes Lemma \ref{L:formalcentralizer}

\begin{lemma}\label{L:formalcentralizergen}
	  Any formal automorphism in $\mathrm{Aut}(U_{k,\nu,P}, C)$  is actually
	convergent and there exists a nonnegative integer $n=n(k,P)$ dividing $k$ such that   $\mathrm{Aut}(U_{k,\nu,P}, C)$ is isomorphic to $\Z/n \times ({\C}^* \times  {\C}^*) $.  Moreover the Lie algebra associated to the infinitesimal action of the second factor  ${\C}^* \times  {\C}^*$ is spanned by $v_0$ and $v_\infty$
\end{lemma}

\begin{proof}
	 Let $E$ be the $\C$-vector space generated by $\omega_0$ and $\omega_\infty$. Set $G= \mathrm{Aut}(U_{k,\nu,P}, C)$. This group acts naturally on the pencil of foliations $(\F_t)$ and because there is no constant meromorphic function constant on $C$ (finiteness of the Ueda type), also on $E-\{0\}$. Moreover,  any $g\in G$ acts trivially on $H^1(C)$, then  for any $\omega\in E-\{0\}$, $\omega\not=\omega_\infty$, $g^*\omega$ and $\omega$ must have the same periods (one refers to \cite[Section 2.4]{LTT} for the notion of periods involved). As the period mapping is injective on $ E-\{0\}$ (see \cite[Corollary 2.7]{LTT}), one concludes that $g^*$ is the identity on $E$. In particular, according to the writing of $\omega_0$ and $\omega_\infty$ and up to composition with the flow of $v_\infty$, we deduce that $g$ takes the form 
$$g(x,y)=\big(x+y h(y), a\exp{(\frac{ty^{k+1}}{1+\nu y^k}{\partial_y})}\big)$$  
where 
\begin{itemize}
\item $a^k=1$, 
\item $\omega_P$ is invariant by the rotation $y\to ay$ and 
\item $yh(y)=\tau\int_{0}^y [{\exp{(\frac{ty^{k+1}}{1+\nu y^k}{\partial_y})}}^*\omega_P -\omega_P]$.
\end{itemize}
 Note that, when $a=1$, $g$ is  nothing but $\exp{tv_0}$. Conversely, the transformation defined by $(x,y)\to (x,ay)$ where $a$ satisfies the conditions above and $\exp{tv_0}$ are elements of $G$. Set $n=k$ if $P$ is constant. When $P$ has positive degree, let $d_P$ be the supremum of the positive integer $d$ such that $P$ is invariant under the action of the group of $d^{th}$ roots of unity over the polynomials

 $$(\mu,P(X))\to P(\mu X)$$
 
  Set $n=\mbox{gcd}(k,d_P)$. From the remarks above, we observe that for every $g\in G$, there exists $(l,t)\in \Z/n\times \C$ uniquely defined such that $g={h_n}^l\circ \exp{tv_0}$ and $h_n(z,\xi)=(z, e^{\frac{2i\pi}{n}}\xi)$, whence the sought isomorphism. Note that $n=1$ in general.
\end{proof}

We are now ready to undertake the analytic classification.

 \subsection{Trivial normal bundle, finite Ueda type: the fundamental isomorphism}\label{Uedatypek}
 Still using the formal classification of \cite{LTT} as recalled in Section \ref{SS:formalclassification}, one can investigate, without any further fundamental change, the analytic classification of neighborhood of elliptic curves with arbitrary finite Ueda type. For the sake of simplicity, we will first focus on the case where the normal bundle $N_C$ of $C$ is analytically trivial.  The formal normal forms  $(U_{k,\nu,P}, C)$ are parametrized by the triple $(\nu,k,P)$ where $\nu$ is a complex number, $k\in {\mathbb N}_{>0}$ is the Ueda type and $P$ is a polynomial map  of one indeterminate of degree $<k$; $P$ is uniquely defined modulo a certain action of the $k^{th}$ roots of unity on the coefficients of $P$ as described in subsection \ref{SS:formalclassification}.

Following subsection \ref{SS:formalclassification}, one has  
$$F_{k,\nu, P} ={\Phi}^{-1}\circ F_{1,0,0} \circ {\Phi}\ \ \ \text{where}\ \ \ \Phi(z,\xi)= (ze^{\int \frac{2i\pi\tau P(\xi)d\xi}{\xi}},\frac{\xi^k}{k} +\nu\log{\xi} )$$ 
and consequently 
$${\Phi}^{-1}(z,\xi)= (ze^{-Q(\varphi^{-1}(\xi))},\varphi^{-1}(\xi))\ \ \ \text{where}\ \ \ 
\left\{\begin{matrix}
\varphi(\xi)=\frac{\xi^k}{k} +\nu\log{\xi},\ \text{and}\\
Q(\xi)={\int \frac{2i\pi\tau P(\xi)d\xi}{\xi}}\hfill\hfill
\end{matrix}\right.$$ 
Note that $\Phi$ and $\Phi^{-1}$ make sense as univalued functions on sectors (with respect to the $\xi$ variable) of opening $< 2\pi$ and that their expressions depend on the choice of the determination of  $\log{\xi}$. Moreover $\varphi^{-1}(\xi)= k^{\frac{1}{k}}\xi^{\frac{1}{k}}  +o({|\xi|}^{\frac{1}{k}})$. Note that $\Phi$ provides a conjugation between the pencils of foliations respectively attached to $(U_{k,\nu,P},C)$ and $(U_{1,0,0},C)$.

 Let us choose $4$ intervals $I_i^l$, $i=1,...,4 \in\Z_4$,  $l=0,...,k-1 \in \Z_{k}$ as follows:
 $$I_1^0=]-\frac{\varpi}{k},\frac{\pi-\varpi}{k}[,\ \ \  I_2^0=]0,\frac{\pi}{k}[,\ \ \ I_3^0=I_1^0+\frac{\pi}{k},\ \ \ I_4^0=I_2^0+\frac{\pi}{k}$$
so that $kI_i^0=I_i$ and $I_i^0\cap I_{i+1}^0\not=\emptyset$, and for $l\not=0$, set
$$I_i^l=I_i^0 +\frac{2l\pi}{k}.$$
Define  
$$\Pi_{k,\nu,P}(z,\xi):=\Pi\circ \Phi(z,\xi)=\left(e^{\frac{2i\pi\xi^k}{k}}\xi^{2i\pi\nu},z e^{P_1(\xi)+2i\pi\tau\frac{\xi^k}{k}}\xi^{2i\pi\tau(\nu+ a_0)}\right)$$ 
where $a_0=P(0)$ is the constant term of $P$ and $P_1$ is a polynome of degree $<k$ whose exact expression is not relevant in the sequence. 
This transformation $\Pi_{k,\nu,P}$ maps  the germ of transversely sectorial domains $U_i^l$ of opening $I_i^l$ (which is $F_{k,\nu, P}$ invariant) onto the germ of deleted neighborhood of $L_i$. The full picture  is thus obtained by taking the successive family of sectors with consecutive overlaps $(I_1^0, ..., I_4^0, I_1^1,..........., I_1^{k-1},...., I_4^{k-1})$ and their corresponding system of transversal sectorial neighborhoods. This defines a cover of the deleted neighborhhood $U_{k,\nu,P}\setminus C$ equipped with a system of semi-local charts defined by $\Pi_{k,\nu,P}$. If one follows the cyclic order defined by the succession of overlapping sectors,  one conclude that this deleted neighborhood can be represented by the union of deleted neighborhoods  $V_i^l-L_i^l$   of the components of a cycle  $D=\bigcup L_i^l$ of $4k$ rational curves of zero type, namely the image under $\Pi_{k,\nu,P}$ of these sectors. The neighborhood $V_i^l$ of $L_i^l$ is equipped with the standard coordinate system of $V_i$ (neighborhood of $L_i$ in ${\mathbb P}^1\times {\mathbb P}^1$).

We will also set $V_{i,i+1}^l:=V_i^l\cap V_{i+1}^l$ if $i\not=0 \mod 4$,  $V_{4,1}^l:=V_4^{l-1}\cap V_{1}^l$ and consider $p_{i,i+1}^l\in V_{i,i+1}^l$ the crossing points of $D$.

\begin{figure}[htbp]
\begin{center}
\includegraphics[scale=0.5]{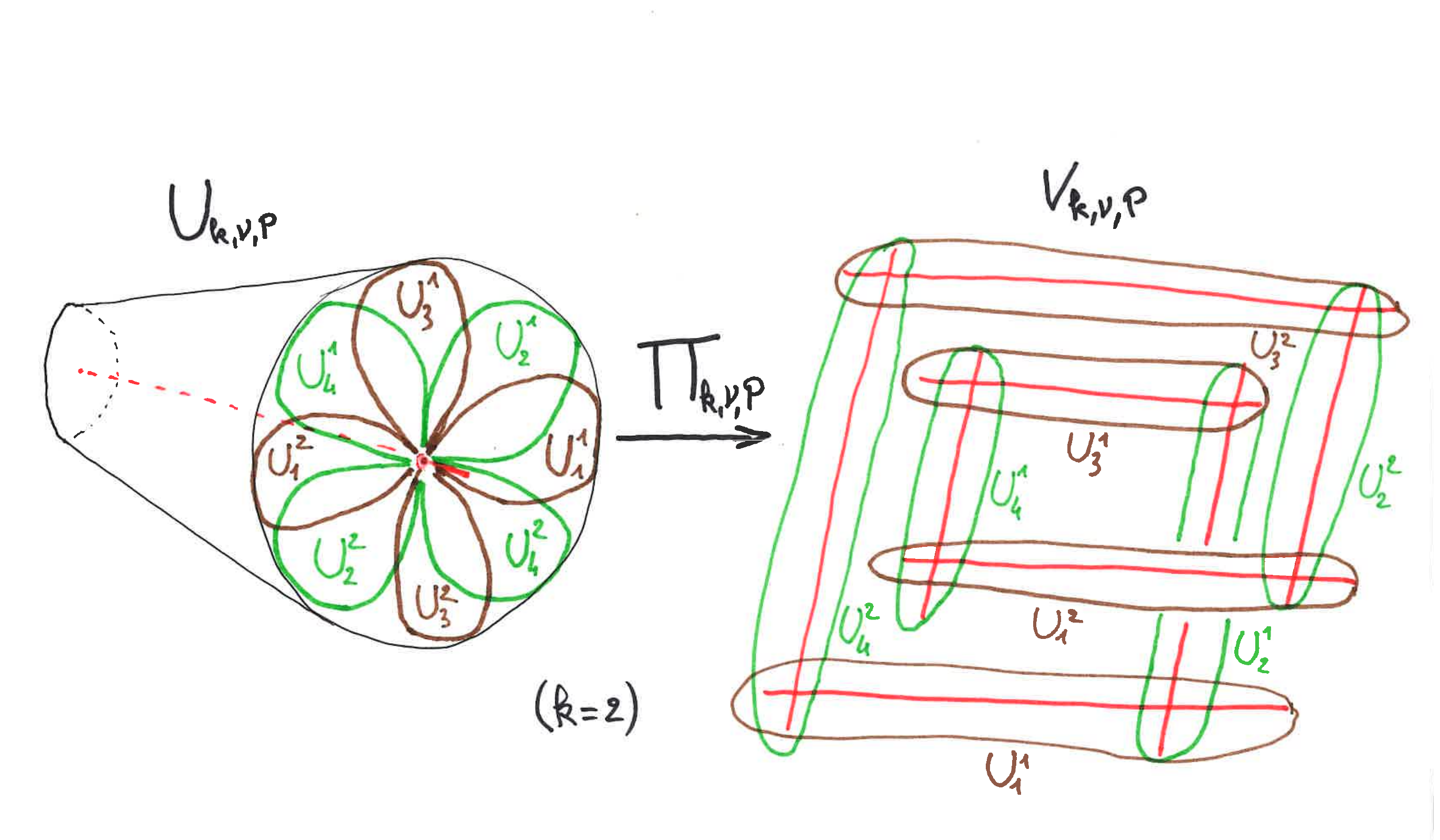}
\caption{{\bf The isomorphism  $\Pi_{k,\nu,P}$ for $k=2$}}
\label{pic:GeneralizedSerrek=2}
\end{center}
\end{figure}

  To summarize we have thus constructed via $\Pi_{k,\nu,P}$ a two dimensional germ of neighborhood $V_{k,\nu,P}$ of $D$ such that  each component of $D$ is embedded with zero self intersection and such that 
$$V_{k,\nu,P}\setminus D\simeq U_{k,\nu,P}\setminus C.$$
  
 The complex structure of $V_{K,\nu,P}$ is determined by trivial glueings (with respect to the $(X,Y)$ coordinate) along the intersections $V_{i,i+1}^l$ except on $V_{4,1}^0$  
 where identification is made by the monodromy  of $\Pi_{k,\nu,P}$ which is given by the linear diagonal map 
 
 \[d_{\nu,P}:\left.\begin{array}{ccc}(V_1^0,p_{4,1}^0) & \to&(V_4^{k-1},p_{4,1}^0)  \\(X,Y) & \to &(e^{-4\pi^2 \nu}X, e^{-4\pi^2\tau(\nu+a_0)}Y)\end{array}\right.\]

 Let us denote by $V_{k,\nu,P}$ the neighborhood of four lines obtained by this process. If one specializes this serie of observations to $k=1$ and then $P=\mu$ automatically constant, one obtains an isomorphism between the deleted neighborhoods 
 
  $$V_{1,\nu,\mu}\setminus D\simeq U_{1,\nu,\mu}\setminus C.$$
  
  In this context,  $D$ is a cycle of $4$ rational curve, and one recovers the description of Serre's example for the value $(0,0)$ of the parameter $(\nu,\mu)$.
 
 \begin{remark}\label{R:autvknp}
 On $V_{k,\nu,P}$,  the  vector fields induced respectively by $v_\infty$ and $v_0$ simply read as $X_\infty=2i\pi\tau Y{\partial_Y}$ and $X_0=-2i\pi X{\partial_X}-2i\pi\tau Y{\partial_Y}$.
 In particular the $\C^*\times \C^*$ part of $$\mathrm{Aut}(U_{k,\nu,P}, C)$$ corresponds at the level of $V_{k,\nu,P}$ to transformations of type $(X,Y)\to (aX,bY)$, $a,b\in\C^*$.
According to the description   given in (the proof of) Lemma \ref{L:formalcentralizergen} 
 it remains to determine the action corresponding to 
 $g=h_n:(z,\xi)\to (z,e^{\frac{2i\pi}{n}}\xi)$. These latter is given in restriction to $V_i^l$, $0\leq l<k$, of the following transformation 
 \begin{enumerate}
 \item $\Theta_n:(X,Y)\in V_i^l\to \left(e^{\frac{-4\pi^2\nu}{n}}X,e^{\frac{-4\pi^2\tau(\nu +a_0)}{n}}Y\right)\in V_i^{{l+\frac{k}{n}}}$, 
 when $l\leq k-1-\frac{k}{n}$, 
 \item 
 $\Theta_n:(X,Y)\in V_i^l\to \left(e^{{4\pi^2\nu}(1-\frac{1}{n})}X,e^{-4\pi^2\tau(\nu +a_0)(1-\frac{1}{n})}Y\right)\in V_i^{{l+\frac{k}{n}}}$
 otherwise.
\end{enumerate}
This is done by a straightforward computation.
\end{remark}

  \subsection{The analytic moduli space}
  Notations of Section \ref{Uedatypek}.
  Our purpose is to give a simple characterization of neighborhhoods $(U,C)$ formally conjugated to $(U_{k,\nu,P},C)$ up to analytic equivalence. The way to proceed is actually suggested in the previous Section  where things are  basically "modeled" on  Serre's example.
  
   Let us settle some notations in accordance with Section \ref{sec:SectorialDecompSym}.  
 For this, denote by $\mathrm{Diff}^1(V_i^{l},L_i^l)$ the group of germs of biholomorphisms of $(V_i^l,L_i^l)$
 which preserves the germ of divisor $(D\cap V_i^l)$ and tangent to the identity along $L_i^l$ and by 
$\mathrm{Diff}^1(V_{i,i+1}^{l},p_{i,i+1}^l)$ the  group  germs of diffeomorphisms of $(V_{i,i+1}^l,p_{i,i+1}^l)$ which preserves the germ of divisor $(D\cap V_{i,i+1})$ and tangent to the identity at $p_{i,i+1}^l$. 

\begin{dfn}
By definition, a cocycle consists in the datum  of a family of $4k$ germs of diffeomorphisms $\varphi=(\varphi_{i,i+1}^l)$ where $\varphi_{i,i+1}^l\in \mathrm{Diff}^1(V_{i,i+1}^{l},p_{i,i+1}^l)$.
\end{dfn}

Analogously to the case of Serre's example, we investigate in terms of cocycle the structure of the moduli space $\mathcal U_{k,\nu,P}$ of neighborhoods $(U,C)$ formally equivalent to $(U_{k,\nu,P},C)$ up to analytic equivalence.\footnote{Recall that we allow for this statement analytic isomorphisms inducing translations on $C$} 

\subsubsection{Normalizing sectorial maps}

Concerning the construction of normalizing maps,  let us explicit as before the modifications needed: let $(U,C)$ be formally equivalent to $(U_{k,P,\nu},C)$.

One can suppose that  $(U,C)=(\tilde U, \tilde C)/F$ where $F(z,\xi)=F_{k,\nu,P}(z,\xi)+(\Delta_1,\Delta_2)$ with $\Delta_i=O(\xi^{-N})$ where $N>>0$ is a positive integer arbitrarily large and that there exists a formal diffeomorphism of $(\tilde U,\tilde{C})$ 
$$\hat{\Psi}(z,\xi)=(z+\hat{g},\xi+\hat{h})=\left(z+\sum_{n\geq 1}b_n \xi^{-n},\xi+\sum_{n\geq 1}a_n \xi^{-n}\right)$$
where $a_n,b_n$ are entire functions on $\tilde C={\C}^*$, such that
\[F\circ \hat{\Psi}=\hat{\Psi}\circ F_{k,\nu,P}.\] 

One wants to generalize the construction performed  in Section \ref{S:sectorialnorm1} and replace  $\hat{H}$ by a collection of transversely sectorial biholomorphisms $\Psi_i^l\in \mathscr G^1({I_i^l})$, $\widehat{\Psi_i^l}=\hat{\Psi}$, \footnote{Notations and definitions of Section \ref{sec:sheavesautomorphisms}} 
that is
\begin{equation}\label{E:conjequedatype1bis}
F\circ {\Psi_i^l}={\Psi_i^l}\circ F_{k,\nu,P}
\end{equation}
 One adopts the notation of loc.cit without systematically mentioning the parameters $(k,\nu,P)$.

As before, we will just detail the construction of a normalizing conjugation map on the germ of transversal sectorial domain determined by $I_1:=I_1^{l_0}$ where $l_0\in \Z_k$ has been fixed.  Consider  the annulus ${\mathcal C}_{a,b}=\{{a}\leq|z|\leq b\}$, $a,b>0, b>{|q|}^{-1} a$. Consider the locus where the foliation $\F_1 $ is defined by the level sets $ \{Y_{k,\nu,P}=c \}\cap \{\arg{\xi}\in I_1\}$, $c$ "small" of 
$$Y_{k,\nu,P}=Y\circ\Phi=z e^{P_1(\xi)+2i\pi\tau\frac{\xi^k}{k}}\xi^{2i\pi\tau(\nu+ a_0)}
=ze^{2i\pi\tau\left(\frac{\xi^k}{k} +R(\xi)\right)}$$ 
where $R(\xi) 
=o(\xi^k)$.  It corresponds to the fibration $dY=0$ on a neighborhood $V_1^l$ of $L_1^l$.  Let $\delta (a,b)>0$ sufficiently small. 
 For every $0<\delta\leq \delta(a,b)$,  set 
 $$S_{a,b, \delta}=\{(z,\xi)\in {\C}^2|\ |Y_{k,\nu,P}(z,\xi)|\leq\delta,\  \ z\in {\mathcal C}_{a,b},\ \arg{\xi}\in I_1\}.$$ 
  For every complex number $c$, $0<|c|\leq\delta$, consider $S_{a,b,c}=\{Y_{k,\nu,P}=c\}\cap S_{a,b,\delta}$. 
  Note that $\bigcup_{0<|c|\leq\delta} S_{a,b,c}=S_{a,b,\delta}$ and that $F_{k,\nu,P}(S_{a,b,\delta})\approx S_{|q| a,|q| b,\delta}$.
 It is thus coherent to investigate the existence of a solution $\Psi=\Psi_{1}\in  \mathscr G^1({I_1}) $ of (\ref{E:conjequedatype1bis}) on the domain $S_{ a, b,\delta}'=S_{ a, b,\delta}\cup F_{k,\nu,P}(S_{ a, b,\delta})$. 
 
 Remark now that the conjugacy equation (\ref{E:conjequedatype1bis}) can be equivalently rewritten
  \[F_\Phi\circ \Psi_\Phi=\Psi_\Phi\circ F_{1,0,0}\]
 where $F_\Phi= \Phi\circ F\circ \Phi^{-1}$, $\Psi_\Phi=\Phi\circ \Psi\circ \Phi^{-1}$.   
  This amounts to determine $\Psi_\Phi$ on the domain $\Phi(S_{ a, b,\delta}')$.
 To do this, note that in view of asymptotic behavior of $\varphi$ and $\varphi^{-1}$ described in Section \ref{Uedatypek}, one has $F_\Phi=F_{1,0,0} +(\Delta_{1,\Phi},\Delta_{2,\Phi})$ where one can verifies that $(\widetilde{\Delta_{1,\Phi}},\Delta_{2,\Phi})= O(\frac{1}{\xi^{N}})$ setting $z\widetilde{\Delta_{1,\Phi}}= \Delta_{1,\Phi}$. 

 We are looking for solution of the form $\Psi_\Phi=\mbox{Id} +(h_\Phi,g_\Phi)$.
Here again,  this can be reformulated as 
\begin{equation}\label{E:F1U1gen}
{h}_\Phi\circ F_{1,0,0} -{h_\Phi}=\Delta_{2,\Phi}\circ {\Psi_{\Phi}}
\end{equation}
\begin{equation}\label{E:F2U2gen}
 {g_\Phi}\circ F_{1,0,0}-q {g_\Phi}=\Delta_{1,\Phi}\circ {\Psi_\Phi}
\end{equation}

First, we will still deal with the linearized equations
This can be reformulated as 
\begin{equation}\label{E:lin1U1gen}
{h}\circ F_{1,0,0} -{h}=\Delta_2
\end{equation}
\begin{equation}\label{E:lin2U2gen}
 {g}\circ F_{1,0,0}-q {g}=\Delta_1
\end{equation}
where we have omit the subscript $\Phi$ for notational convenience.

 In what follows, we will indeed provide a solution of (\ref{E:lin1U1gen}) (hence also for (\ref{E:lin2U2gen}) by the usual transform) with "good estimates" on a domain of the form $\Phi(S_{ a, b,\delta}')$ using a "leafwise" resolution with respect to the foliation defined by the levels of $Y$. This consists as before in solving a family of difference equations parametrized by the leaves space in the variable $\xi$ where the domains are slightly modified as explained now. For every complex number $c$, $0<|c|\leq\delta$, consider $S_{a,b,c}=\{Y_{k,\nu,P}=c\}\cap S_{a,b,\delta}$ and $\Phi(S_{a,b,c}) =\{(z,\xi)=(ce^{-2i\pi\tau\xi},\xi)\}$ where $\xi$ belongs to  
 \[  \Sigma_{c,a,b}=\{\xi\in\C:(\log{|c|}-\log{b})\leq\Re{(2i\pi\tau \xi+Q({\varphi^{-1}(\xi))}}\leq (\log{|c|}-\log{b})\}\]
 where a determination of the logarithm has been chosen in the sector in which we are working, namely $\arg{\xi}\in I_1$. Note also that $Q({\varphi^{-1}(\xi))}=o(\xi)$ and consequently the middle term in the above inequation "behaves" like $\Re{(2i\pi\tau \xi)}.$ 
  The remaining part of the proof then follows \textit{mutatis mutandis} the same line   by noticing that the 
   equation \[h\circ F_{1,0,0}-h=\Delta_2\] can be then rewritten as 
   \begin{equation}\label{E:HE2}
\varphi_c(\zeta+\lambda)-\varphi_c(\zeta)=\Delta_c(\zeta) 
\end{equation}
where $\lambda=-2i\pi\tau$, $\zeta=2i\pi\tau \xi$,  $\varphi_c(\zeta)=h(c e^{-\zeta},\frac{\zeta}{2i\pi\tau})$, and $\Delta_c(\zeta)=\Delta_\delta(c e^{-\zeta},\frac{\zeta}{2i\pi\tau})$  with $0<|c|\leq\delta$.

 We are then reduced to solve a family difference equations  in the "quasi" vertical strip  $$\Sigma_{c}=\{\zeta\in\C:\log{|c|}-\log{b}\leq\Re{(\zeta+Q{(\varphi^{-1}(\frac{\zeta}{2i\pi\tau}))})}\leq \log{|c|}-\log{a}\}$$ depending analytically on the parameter $c$ and we investigate the existence of an solution $\varphi_c$ on the domain $\Sigma_c\cup (\Sigma_c +\lambda)$. This can be carried out by resolving equation (\ref{E:DE1}) using the same method (that is, essentially Cauchy formula) replacing accordingly in the expression (\ref{E:solseries}) the integration along the vertical lines $L^-$, $L^+$ by $l(\zeta)=A,B$ where $l(\zeta)= \Re{(\zeta-R{(\varphi^{-1}(\frac{\zeta}{2i\pi\tau}))})}$. One also obtain in a similar way the same kind of estimates imposing the uniqueness of the solution. This allows by the fixed point method detailed in the previous section to solve the conjugation equation under the forms (\ref{E:F1U1gen}) and (\ref{E:F2U2gen}) on the relevant domains and finally exhibit a conjugacy sectorial transformation $$\Psi=\Phi^{-1}\circ \Psi_\Phi\circ\Phi,$$ well defined as a section of ${\mathscr G}^1$ over $I_1$ and having $\hat{\Psi}$ as asymptotic expansion. One can analogously construct conjugacy maps on  other sectors of opening $I_i^l$, $i=1,3$, $l\in \Z_k$ and also on $I_i^l$, $i=2,4$ by exchanging the roles of the foliations $\F_0$ and $\F_1$ following the process described in Section \ref{SS:othersectors}.

 \subsubsection{Cocycles versus analytic class and statement of the main result}

One borrows notation from Section \ref{sec:SectorialDecompSym} and defines  ${\mathscr G}^\infty [F_{k,\nu,P}]$ to be the subsheaf of ${\mathscr G}^\infty$ formed by germs sectorial holomorphic transformations flat to identity along $\tilde C$ and commuting to $F_{k,\nu,P}$. By exploiting the Proposition \ref{P:sectorstoP1XP1} and the fact that $(U_{k,\nu,P},C)$ is "sectorially modeled" on $(U_0,C)$ as described in Section \ref{Uedatypek}, one obtains the following Proposition:

\begin{prop}\label{P:sectorstoP14k} We have the following characterizations:
	\begin{itemize}
		\item $\Phi\in {\mathscr G}^\infty [F_{k,\nu,P}](I_i^l)$ if and only if $\Pi_{k,\nu,P}\circ \Phi=\varphi\circ\Pi_{k,\nu,P}$ where $\varphi\in\mathrm{Diff}^1(V_i^l,L_i^l)$;
		\item $\Phi\in{\mathscr G}^\infty [F_{k,\nu,P}](I_{i,i+1}^l)$ if and only if $\Pi_{k,\nu,P}\circ \Phi=\varphi\circ\Pi_{k,\nu,P}$ where $\varphi\in\mathrm{Diff}^1(V_{i,i+1}^l,p_{i,i+1}^l)$ except for $i=4$ and $l=k-1$ for which one has $$\Pi_{k,\nu,P}\circ \Phi=d_{\nu,P}\circ\varphi\circ\Pi_{k,\nu,P}.$$
\end{itemize}
\end{prop}

 \begin{dfn}\label{def:EquivalentCocyclesgen}
 	We say that two cocycles $\varphi$ and $\varphi'$ are equivalent if
 	$$\exists t,t'\in\C,\ \exists\varphi_{i}^l\in\mathrm{Diff}^1(V_i^l,L_i^l), \exists m\in \Z_n$$
 	\begin{equation}\label{eq:anequivcocyclegen}
 	\text{such that}\ \ \ {\varphi_{i,i+1}^{l+m\frac{k}{n}}}'=\theta_n^m\circ\phi\circ\varphi_i^l \circ\varphi_{i,i+1}^l \circ{\varphi_{i+1}^{l+\varepsilon (i)}}^{-1}  \circ\phi^{-1}\circ\theta_n^{-m}
 	\end{equation}
 	where $\phi=\exp{(tX_0}+{t'X_\infty})$  has thus the form $\phi(X,Y)=(aX,bY)$ (cf. Remark \ref{R:autvknp}).
 	We will denote this equivalence relation by $\approx$.
 \end{dfn}

	 Firstly, let us explain as in Section \ref{sec:SectorialDecompSym} how to associate a cocycle to a neighborhood  $(U,C)=(\tilde U,\tilde C)/<F>$ formally conjugated to $(U_{k,\nu,P},C)$. 
Recall that this means that there is a formal diffeomorphism (that can be assumed to be tangent to the identity along $\tilde C$) $\hat\Psi$ conjugating $F$ to $F_{k,\nu,P}$
	i.e. $F\circ \hat\Psi=\hat\Psi\circ F_{k,\nu,P}$.

	Recall (cf.\ref{Uedatypek}) that  for every pair $(i,l)$, there exists a section $\Psi_i^l$ of ${\mathscr G}^1(I_i^l)$
	such that
	$$\Psi_i^l\circ F=F_{k,\nu,P}\circ \Psi_i^l.$$
According to the description of $\mathrm{Aut}(U_{k,\nu,P},C)$, namely the fact that it contains only \textit{convergent} automorphisms,  
one can assume that the asymptotic expansion of the $\hat{\Psi_i^l}$ is constant equal to $\hat \Psi$. 
The $\Psi_i^l$ are unique up to post composition by a sectorial diffeomorphism $g_i^l\in {\mathscr G}^1[F_{k,\nu,P}](I_i^l)$ having the form 
\begin{equation}\label{E:leftcomposition}
 g_i^l=\exp{(tv_0)}\circ h_{i}^l
 \end{equation}
where $h_{i,l}\in\mathscr G^\infty[F_{k,\nu,P}](I_i^l)$. Actually, the flow of $v_0$ contains exactly all the automorphisms tangent to identity along $C$.

It follows that, on intersections $I_{i,i+1}^l= I_i^l\cap I_{i+1}^{l+\varepsilon(i)}$, $\varepsilon(i)=\delta_{i4}$, one obtains $4k$ germs of sectorial glueing biholomorphisms 
$$\Phi_{i,i+1}^l:=\Psi_i^l\circ\left({\Psi_{i+1}^{l+\varepsilon(i)}}\right)^{-1}\in\mathscr G^\infty[F_0](I_{i,i+1}^l).$$
One can now invoke Proposition \ref{P:sectorstoP14k}: setting $\Pi_i^l:=\Pi_{k,\nu,P}\circ\Psi_i^l$ (and taking into account the determination of $\Pi_{k,\nu,P}$ on the corresponding sectors), one can associate $4k$ corresponding cocycles:
\begin{equation}\label{eq:PiVarphiCocycleIdentity}
\Pi_i^l=\Pi\circ\Psi_i=\Pi\circ\Phi_{i,i+1}\circ\Psi_{i+1}=\varphi_{i,i+1}^l\circ\Pi\circ\Psi_{i+1}^l=\varphi_{i,i+1}^l\circ\Pi_{i+1}^l
\end{equation}
except for $i=4$ and $l=k-1$ for which one has 
$$\Pi_4^{k-1}=d_{\nu,P}\circ \varphi_{1,4}^0 \circ \Pi_1^0.$$
We have therefore associated to each neighborhood $(U,C)$ formally equivalent to $(U_{k,\nu,P},C)$
a cocycle $\varphi=(\varphi_{i,i+1}^l)_{i\in\Z_{4k}}$ which is unique up to the freedom for the choice of $\Psi_i$'s.
 We have now all the ingredients to state and prove our main Theorem
 
  \begin{THM}\label{prop:AnalyticClassifgen}
 Two neighborhood $(U,C)$ and $(U',C)$ formally equivalent to $(U_{k,\nu,P},C)$  are
analytically equivalent \footnote{Recall that that formal/analytic conjugations are allowed to induce tranlations on $C$.} 
if, and only if, the corresponding cocycles are equivalent
 	\begin{equation}\label{eq:anequivcocyclegen}
 	(U,C)\stackrel{\text{an}}{\sim}(U',C)\ \ \ \Leftrightarrow\ \ \ \varphi\approx\varphi'.
 	\end{equation}
Moreover, every cocycle can be realized by the process described above.
 \end{THM}
 
\begin{proof}
Now, consider a biholomorphism between two neighborhood  $(U,C)\to(U',C)$ formally conjugated to $(U_{k,\nu,P},C)$.  According to the description of $\mathrm{Aut}(U_{k,\nu,P},C)$, it is represented by a bihomorphism of $(\tilde{U},\tilde{C})$ taking the form $h_n^m\circ \exp{(t'v_\infty)} \circ\Psi$ where
	$\Psi\in\mathscr G^1(\bS^1)$ and satisfying $h_n^m \circ \exp{(t'v_\infty)} \circ\Psi\circ F=F'\circ h_n^m \circ \exp{(t'v_\infty)}\circ\Psi$. Let $(\Psi_i^l)$ and 
	$({\Psi_i'}^l)$ be the sectorial normalizations used to compute the invariants $\varphi$
	and $\varphi'$. Clearly, ${(h_n^m\circ \exp{(t'v_\infty)})}^{-1}\circ{\Psi_i'}^{l+m\frac{k}{n}}\circ(h_n^m\circ \exp{(t'v_\infty)})\circ\Psi$ provides for $(U,C)$ a new collection of sectorial trivializations defined on $I_i^{l}$ 
tangent to identity.

 By virtue of \ref{E:leftcomposition},  we can then write
	$${\Psi_i'}^{l+m\frac{k}{n}}\circ h_n^m \circ \exp{(t'v_\infty)}\circ \Psi=(h_n^{m}\circ \exp{(t'v_\infty)})\circ\exp(tv_0)\circ\Phi_i^{l}\circ\Psi_i^{l}$$ {with} $\Phi_i^l\in\mathscr G^\infty[F_{k,\nu,P}](I_i^l).$
Therefore,
	we have
	$$\Phi_{i,i+1}'^{l+m\frac{k}{n}}=(\Psi_i'^{l+m\frac{k}{n}}\circ h_n^m \circ \exp{(t'v_\infty)}\circ\Psi)\circ(\Psi_{i+1}'^{l+m\frac{k}{n}+\varepsilon (i)}\circ h_n^m \circ \exp{(t'v_\infty)}\circ\Psi)^{-1}$$
$$=h_n^m\circ\varphi\circ\Phi_i^{l}\circ\Phi_{i,i+1}^{l}\circ{\Phi_{i+1}^{l+\varepsilon (i)}}^{-1}\circ{\varphi}^{-1}\circ\ h_n^{-m}.$$
	where $\varphi:= \exp{(tv_0)}\circ\exp{(t'v_\infty)}=\exp{(t'v_\infty)}\circ\exp{(tv_0)}$.	After factorization through $\Pi_{k,\nu,P}$, using 
	Proposition \ref{P:sectorstoP14k},
	we get the expected equivalence relation (\ref{eq:anequivcocyclegen}) for $\varphi$ and $\varphi'$.
	Conversely, if $\varphi\approx\varphi'$, then we can trace back the existence of an analytic conjugacy 
	$\Phi:(U,C)\to(U',C)$ by reversing the above implications. Finally, it suffices to mimick the construction performed in Section \ref{SS:Construction} in order to realize every cocycle.
\end{proof}

\begin{remark}\label{rem:weakequivalence}
One thus observes that two  cocycles are equivalent iff they lie on the same orbit over some action (that the reader will easily explicit) of  $\Z/n \times{({\mathcal O}^* \times{\mathcal O}^*)}^{\times_{\C^* \times\C^* }^{4k}}$. 	
One can strenghten the analytic equivalence by demanding that conjugations induce the identity on $C$. In that case, we have to replace $\phi$ by $\exp{(tX_0)}$ in the statement of definition \ref{def:EquivalentCocyclesgen}. \end{remark}

 \section{Torsion normal bundle.}\label{S:torsion}

  Recall that the elliptic curve is regarded as the quotient $C=\C/\Z\oplus \tau\Z$. We maintain notations of Theorem \ref{thm:NormalFormNeigh0<p<k} and we also set $d:=\mbox{gcd}(m_1,m_\tau)$.
   Let $(U_{m}, C_{m})$ be an $m$-cyclic cover  trivializing the normal bundle having the following form: 
$$(U_{m}, C_{m})=(\C_x\times \C_y,\{y=0\})/\langle\phi_1^{m_1}, \phi_1^l\circ \phi_\tau^{\frac{m_\tau}{d}}\rangle$$ where $l$ is an integer such that $a_1^l a_\tau^{\frac{m_\tau}{d}}=1$. The deck transformation group $G=\Z/m$ is then generated by $\phi_1^v \circ \phi_{\tau} ^w$
where $(v,w)$ is any pair of integers fullfilling $a_1^v a_\tau^w=e^{\frac{-2i\pi}{m}}$.
Note that, as an effect of this  cover, the modulus of the elliptic curve changes accordingly and more precisely $C_{m}$ is determined by the lattice $\langle m_1,l+ \frac{m_\tau}{d} \tau\rangle$.

An easy calculus yields
\begin{equation}\label{eq:NormalFormNeighGroup0<p<k}
\left\{\begin{array}{ccc} \phi_1^{m_1}(x,y)= (x+m_1\ , y) \hfill\hfill\\ 
\phi_1^l\circ \phi_\tau^{\frac{m_\tau}{d}}(x,y)=(x+l+\frac{m_\tau}{d} \tau+\tau\int_0^y \lbrack({{\varphi_{k,\nu}}^{\frac{ m_\tau}{d}})}^*\omega_P-\omega_P\rbrack \ ,\ {\varphi_{k,\nu}}^{\frac{ m_\tau}{d}}(y))\end{array}\right. 
\end{equation}
and conjugating by the transformation $\alpha:(x,y)\to (\frac{x}{m_1}, ay)$, where $a={(\frac{m_\tau}{d})}^{\frac{1}{k}}$,  one can reduce to the simplest and usual normal formal form
$$(U_{m}, C_{m})\simeq(\C_x\times \C_y,\{y=0\})/(F_1, F_{\tau'})$$ 
where 
$$F_1(x,y)= (x+1,y),\ \ \ 
F_{\tau'}(x,y)= (x+\tau' + h_{k,\nu', P'}(y), \varphi_{k, \nu'}(y)),$$
and 
$$\tau'=\frac{1}{m_1}(l+\frac{m_\tau}{d} \tau),\ \ \ \nu'= \frac{d\nu}{m_\tau},\ \ \ P'(X)=\frac{\tau P(X^m)}{\tau' m_1},$$
 $$h_{k,\nu', P'}(y)=\tau\int_0^y \lbrack\varphi_{k,\nu'}^*\omega_{P'}-\omega_{P'}\rbrack.$$ 
 If one adopts the presentation in the variable $(z,\xi)$, this corresponds to the neighborhood  labeled identically as in Section \ref{SS:formalclassification} (except that $C$ is replaced by its cover $C_m$ or equivalently $\tau$ par $\tau'$): its presentation as a quotient is given by $({\C}_z^*\times\overline{{\C}_\xi} , \{\xi=\infty\})/\langle F_{k,\nu',P' }\rangle$ where $F_{k,\nu',P'}= \exp{v_{k,\nu',P'}}$ and $v_{k,\nu',P'}=  \frac{1}{\xi^k +\nu'} (-\xi{\partial_\xi} +2i\pi\tau' zP'(\xi){\partial_z})+2i\pi\tau'z{\partial_z}$. The foliations $\F_0$ and $\F_1$ are respectively defined by the levels  of $\xi$ and $ Y_{k,\nu',P'}=ze^{2i\pi\tau'(\frac{\xi^k}{k} +V(\xi^m))}\xi^{2i\pi\tau'{(\nu'+ a_0')}}$ where $V$ is a polynomial of degree $<k'=\frac{k}{m}$ that we do not need to explicit and $a'(0)=P'(0)$. From Section \ref{Uedatypek}, one knows that $(U_m,C_m)$ is  parametrized by a neighborhood $V_{k,\nu', P'}=\bigcup V_i^l$ of $4k$ lines via the map 
  $$\Pi_{k,\nu',P'}=(e^{\frac{2i\pi\xi^k}{k}}\xi^{2i\pi\nu'},J_{k,\nu',P'}). $$ 
  For the sake of clarity and coherence, it will be natural to denote this neighborhood by $(U_{k,\nu,P'},C')$ where $C'=C_m$.
  
In order to recover the structure of the original neighborhood, we have to determine the identifications induced by the deck transformation group $G$ which is cyclic of order $m$ and generated by a transformation 
     $$g_{v,w}=\Phi_1^v \circ \Phi_{\tau} ^w=\mathcal A\circ\exp{(-\beta_1 v_{k,\nu',P'})},\ \ \ \Phi_1:=\alpha\circ\phi_1\circ \alpha^{-1},\ \ \  \Phi_{\tau}:=\alpha\circ\phi_{\tau}\circ\alpha^{-1}$$ 
     where 
     \begin{itemize}
     	\item  $(v,w)$ is the fixed pair of integers such that $a_1^v a_\tau^w=e^{-\frac{2i\pi}{m}}$,
     	\item $\beta_1=-\frac{wd}{m_\tau}$,
     	\item  $\mathcal A$ is the affine map $\mathcal A (x,y)=(x+\beta_2,  e^\frac{-2i\pi}{m} y)$ with  $\beta_2=\frac{v}{m_1}-\frac{wld}{m}$.
     \end{itemize}
        In the $(z,\xi)$ coordinates, $\mathcal{A}$ corresponds to the linear map $d(z,\xi)= (e^{{2i\pi}{\beta_2}}z, e^\frac{2i\pi}{m} \xi )$. In particular, one has $d^m=\mbox{id}$. Note that the vector field $X_{k,\nu',P'}=-2i\pi X\partial_X$ on $V_{k,\nu', P'}$ corresponds via $\Pi_{k,\nu',P'}$  to $v_{k',\nu',P'}$. Its flow at time $-\beta_1$ preserves each individual neighborhood $V_i^l$ of the component $L_i^l$ of the cycle and reads as  
     $$(X,Y)\to (e^{2i\pi\beta_1}X,Y).$$
     
     The action 
  of $g_{v,w}$  by "sectorial permutation" 
  is then explicitely given in the  $(X,Y)$ coordinates (determining the neighborhood of $4k$ lines as described in Section\ref{Uedatypek})  by 
 \begin{equation}
   	 \mathscr D:(X,Y)\in V_i^l\to (e^{\frac{-4\pi^2\nu'}{m}+2i\pi{\beta_1}}X,e^{\frac{-4\pi^2\tau'(\nu' +a_0')}{m}+2i\pi\beta_2}Y) \in V_i^{l+\frac{k}{m}}
 \end{equation}
when $l\leq k-1-\frac{k}{m}$, and by 
   \begin{equation}
    \mathscr D:(X,Y)\in V_i^l\to (e^{{4\pi^2\nu'}(1-\frac{1}{m})+2i\pi\beta_1}X,e^{4\pi^2\tau'(\nu' +a_0')(1-\frac{1}{m})+2i\pi\beta_2}Y)\in V_i^{l+\frac{k}{m}} 
   \end{equation}
   	otherwise.
We will denote by $$(U_{k,\nu',P'}^G,C):=(U_{k,\nu'},C')/G $$ this description of the original neighborhood as a finite quotient.

  At the level of the neighborhood $V_{k,\nu',P'}$ of rational curves, these identifications are given is  by the action of the order $m$ cyclic permutation $\mathscr D$   
  and we eventually end up with a neighborhood of $4k'$ rational curves of zero type, namely the quotient  $$V_{k,\nu',P'}^G:=V_{k,\nu',P'}/\mathscr D. $$
  
  Indeed, set $u=\xi^m$ and consider the function 
  $$\Pi_{k,\nu',P'}^G(z,\xi)=(\xi^{-m\beta_1 }e^{\frac{2i\pi u^{k'}}{k}} u^{\frac{2i\pi\nu'}{m}+\beta_1},\xi^{-m\beta_2}ze^{2i\pi\tau'(\frac{u^{k'}}{k} +V(u))} u^{2i\pi\tau'(\frac{\nu'+ a_0'}{m})+\beta_2})$$
as  defined and univaluate on each of the $4k'$ "multisectorial domains " $$\tilde U_i^{l'}:=\bigsqcup_j U_i^{l+\frac{jk}{m}},$$ $l'=l[\frac{k}{m}]$ by the choice of a determination of the logarithm of $u$. One can notice that $G$ acts transitively on the set of sectors of $\tilde U_i^{l'}$ and that $\Pi_{k,\nu',P'}^G$ is $G$ invariant. On the other hand,  $\Pi_{k,\nu',P'}^G$ coincides with  $\Pi_{k,\nu',P'}$ (up to left composition with a diagonal linear map) on each sector $V_i^l$, hence is constant and separating along the orbits of $F_{k,\nu',P'}$. This shows, as claimed previously, that the (deleted) neighborhood $(U_{k,\nu',P'}^G,C)$ can be represented as the (deleted) neighborhood $V_{k,\nu',P'}^G$ of a cycle $\tilde D=\bigcup \tilde {L}_i^{ l'}$ of $4k'$ rational curves of zero type. More precisely, each individual neighborhood $\tilde{V}_i^{ l'}$, of $\tilde{L}_i^{ l'}$ is the image under $\Pi_{k,\nu',P'}^G$ of $\tilde{U}_i^{l'}$ and is thus equipped as before by coordinates $(X,Y)$ determined by each component of  $\Pi_{k,\nu',P'}^G$. These coordinates glue together trivially on overlaps except on $\tilde {V}_1^0$ (defined as in Section \ref{Uedatypek}) 
where it is given by the linear diagonal map 
$$d^G :(X,Y)\to (e^{(-4\pi^2 \frac{\nu'}{m}+2i\pi \beta_1)}X,e^{(-4\pi^2\tau' (\frac{nu'+ a_0'}{m})+2i\pi\beta_2)}Y ).$$
The reader should compare with the case of resonant diffeomorphisms of one variable \cite[Section 10.3, p. 592 ]{MartinetRamis2}.

Let $(U^1,C)$ be  a neighborhood formally conjugated to $(U_{k,\nu',P'}^G,C)$. Again by making use of \cite{Siu} (See proof of Lemma \ref{L:firstpresentation}), one can suppose that $(U^1,C)$ can be represented as the quotient 
$$ (\C_x\times \C_y,\{y=0\})/\langle\phi_1,  \phi_\tau^1\rangle$$ 
so that there exists a formal diffeomorphism of $(\C_x\times \C_y,\{y=0\})$ commuting with $\phi_1$ that can be assumed to be tangent to the identity along $C$ and conjugating $\phi_\tau^1$ to $\phi_\tau$. 

By taking the same triple of integer $(l,v,w)$ than before, one can argue passing through the $m$-cyclic $(U_m^1,C_m)$ cover of $(U^1,C)$ trivializing the normal bundle. It can be represented as the quotient 
$( \C_z^* \times \bar \C_\xi, \{\xi=\infty\})/F$. Let $G^1$ be the attached deck transformation group with generator $g_{v,w}^1=\Phi_1^v \circ {\Phi_{\tau}^1} ^w$ 
where one sets as before $\Phi_1:=\alpha\circ\phi_1\circ \alpha^{-1}, {\Phi_{\tau}^1}:=\alpha\circ{\phi_{\tau}^1} \circ\alpha^{-1}$. Let $\rho:G\to G^1$ the isomorphism of cyclic group mapping $g_{v,w}$ to $g_{v,w}^1$. 
The formal conjugation between the original neighborhoods, 
can be translated into the existence of a formal transformation map $\hat \Psi$
conjugating $F$ to $F_{k,\nu,P}$, 
i.e. $F\circ \hat\Psi=\hat\Psi\circ F_{k,\nu,P}$ and which is in addition \textit{equivariant} with respect to $\rho$:

$$\forall g\in G, \hat \Psi\circ g=\rho(g)\circ \hat\Psi .$$

Recall (cf.\ref{Uedatypek}) that  for every pair $(i,l)$, there exists a section $\Psi_i^l$ of ${\mathscr G}^1(I_i^l)$ 
such that
$$\Psi_i^l\circ F_{k,\nu,P}=F\circ \Psi_i^l.$$ such that $\hat{\Psi_i^l}=\hat \Psi$. Moreover, the collection of these sectorial normalisations can be chosen equivariantly, ie: $$\Psi_i^{l+j\frac{jk}{m}}\circ g_{v,w}^j=\rho(g_{v,w}^j)\circ\Psi_i^l.$$ This last point is justified by the fact that $G$ and $G^1$ lie respectively in the centralizer of $F_{k,\nu',P'}$ and $F$.

This thus provides a "multisectorial conjugation" on each $\tilde {U_i^{l'}}$. One can mimick the arguments presented in  Section \ref{S:Generalization} and thus obtain the description of the analytic moduli space which is thus determined by a cocycle $\varphi$ with $4k'$ components modulo the identifications given in Theorem \ref{prop:AnalyticClassifgen}.

\end{document}